\documentclass[12pt, reqno]{amsart}
\usepackage{amsfonts}
\usepackage{bbm}
\usepackage{amscd,amsfonts}
\usepackage{amssymb, eucal, amsfonts, amsmath, xypic, latexsym}
\usepackage{pifont}
\usepackage{mathrsfs,color}
\usepackage{amsthm,indentfirst,bm,fancyhdr,dsfont}
\usepackage{graphicx}
\usepackage[all]{xy}
\usepackage[CJKbookmarks=true]{hyperref}

\newtheorem{theorem}{Theorem}[section]
\newtheorem{lemma}[theorem]{Lemma}
\newtheorem{prop}[theorem]{Proposition}
\newtheorem{corollary}[theorem]{Corollary}
\theoremstyle{definition}
\newtheorem{conj}[theorem]{Conjecture}
\newtheorem{defn}[theorem]{Definition}

\newtheorem{rem}[theorem]{Remark}

\newtheorem{conventions}[theorem]{Conventions}

\numberwithin{equation}{section}

\def\ggg{\mathfrak{g}}

\def\gl{\mathfrak{gl}}
\def\sssl{\mathfrak{sl}}

\def\ggg{\mathfrak{g}}
\def\mmm{\mathfrak{m}}
\def\ppp{\mathfrak{p}}

\def\aaa{\mathfrak{a}}
\def\uuu{\mathfrak{u}}

\def\bbc{\mathbb{C}}
\def\bbf{\mathbb{F}}
\def\bbz{\mathbb{Z}}

\def\bbk{{\mathds{k}}}

\def\bo{{\bar 1}}
\def\bz{{\bar 0}}
\def\ev{{\text{ev}}}

\def\sfo{\textsf{o}}
\def\sfO{\textsf{O}}

\def\Lie{\text{Lie}}
\def\ad{\text{ad}}
\def\gr{\text{gr}}

{\vskip-\lastskip\medskip
  \noindent
  {\em #1.}\enspace
  }%
{\qed\par\medskip
  }

\begin{document}
\title[finite $W$-superalgebras and lower bounds]{Finite $W$-superalgebras and the dimensional lower bounds for the representations of basic Lie superalgebras}
\author{Yang Zeng and Bin Shu}
\thanks{\nonumber{{\it{Mathematics Subject Classification}} (2000 {\it{revision}})
Primary 17B35. Secondary 17B81. This work is  supported partially by the NSF of China (No. 11271130; 11201293; 111126062),  the Innovation Program of Shanghai Municipal Education Commission (No. 12zz038). }}
\address{School of Science, Nanjing Audit University, Nanjing, Jiangsu Province 211815, China}
\email{zengyang214@163.com}
\address{Department of Mathematics, East China Normal University, Shanghai 200241, China}
\email{bshu@math.ecnu.edu.cn}


\begin{abstract} In this paper we formulate a conjecture about the minimal dimensional representations of the finite $W$-superalgebra $U(\mathfrak{g}_\bbc,e)$ over the field of complex numbers and demonstrate it with examples including all the cases of type $A$. Under the assumption of this conjecture, we show that the lower bounds of dimensions in the modular representations of basic Lie superalgebras are attainable. Such lower bounds, as a super-version of Kac-Weisfeiler conjecture, were formulated by Wang-Zhao in \cite{WZ} for the modular representations of a basic Lie superalgebra ${\ggg}_{{\bbk}}$ over an algebraically closed field $\bbk$ of positive characteristic $p$.
\end{abstract}
\maketitle
\setcounter{tocdepth}{1}
\tableofcontents
\section*{Introduction}
This paper is a sequel to \cite{ZS2}. On the basis of the structure theory of finite $W$-superalgebras developed there, we study the modular representations of basic Lie superalgebras, as a remarkable application of finite $W$-superalgebras.
\subsection{} A finite $W$-algebra $U(\ggg,e)$ is a certain associative algebra associated to a complex semisimple Lie algebra $\mathfrak{g}$ and a nilpotent element $e\in{\ggg}$. The study of finite $W$-algebras can be traced back to Kostant's work in the case when $e$ is regular \cite{Ko}, then the further study was done by Lynch in the case when $e$ is an arbitrary even nilpotent element (cf. \cite{Ly}). Premet developed the finite $W$-algebras in full generality in \cite{P2}. On his way of proving the celebrated Kac-Weisfeiler conjecture for Lie algebras of reductive groups in \cite{P1}, Premet first constructed the modular version of finite $W$-algebras (they will be called the reduced $W$-algebras in this paper). By means of a complicated but natural ``admissible'' procedure, the finite $W$-algebras over the field of complex numbers were introduced in \cite{P2}, which showed that they are filtrated deformations of the coordinate rings of Slodowy slices.

Aside from the advances in finite $W$-algebras over the field of complex numbers, the modular theory of finite $W$-algebras is also developed excitingly. It is remarkable that in \cite{P7} Premet proved that if the $\bbc$-algebra $U(\ggg,e)$ has a one-dimensional representation, then under the assumption $p\gg 0$ for the positive characteristic field ${\bbk}=\overline{\mathbb{F}}_p$, the reduced enveloping algebra $U_\chi(\ggg_\bbk)$ of the modular counterpart $\ggg_\bbk$ of $\ggg$ possesses an irreducible module of dimension $d(e)$ (where $\chi$ is the linear function on $\ggg_{{\bbk}}$ corresponding to $e$, and $d(e)$ is half of the dimension of the orbit $G_\bbk\cdot \chi$ for the simple, simply-connected algebraic group $G_\bbk$ with $\ggg_\bbk=\Lie(G_\bbk)$), which is a lower bound predicted by Kac-Weisfeiler conjecture mentioned above.

~
The existence of one-dimensional representations for $U(\ggg,e)$ associated to a classical Lie algebra over $\mathbb{C}$ was obtained by Losev in \cite[Theorem 1.2.3(1)]{L3} (see also \cite[\S6]{L1}).
Goodwin-R\"{o}hrle-Ubly \cite{GRU} proved that the $W$-algebras associated to exceptional Lie algebras $E_6,E_7,F_4,G_2$, or $E_8$ with $e$ not rigid, admit one-dimensional representations (see also \cite{P7}).

\subsection{} The theory of finite $W$-superalgebras was developed in the same time.  In the work of De Sole and Kac \cite{SK}, finite $W$-superalgebras were defined in terms of BRST cohomology under the background of vertex algebras and quantum reduction.
The topics on finite $W$-superalgebras attracted many researchers (cf.  \cite{BBG2},  \cite{Peng2}, \cite{Peng3}, \cite{PS}, \cite{PS2}, \cite{PS3}, \cite{WZ2} and \cite{Z2}).

In the work of Wang-Zhao \cite{WZ}, they initiated the study of modular representations of basic Lie superalgebras over an algebraically closed field of positive characteristic, formulating the super Kac-Weisfeiler property for those Lie superalgebras as well as  presenting the definition of modular $W$-superalgebras.

\subsection{}
Based on Premet's and Wang-Zhao's work as mentioned above, our previous paper \cite{ZS2} presents the PBW structure theorem for the finite $W$-superalgebras (along with the reduced $W$-superalgebras), which shows that the construction of finite $W$-superalgebras (also the reduced $W$-superalgebras) can be divided into two cases in virtue of the parity of $\text{dim}\,\mathfrak{g_\bbf}(-1)_{\bar1}$ (where $\bbf$ is an algebraically closed field of any characteristic). To some extent, the situation of finite $W$-superalgebras  is significantly different from that of finite $W$-algebras.

Explicitly speaking, let ${\ggg}={\ggg}_{\bar0}+{\ggg}_{\bar1}$ be a basic Lie superalgebra over $\mathbb{C}$ excluding type $D(2,1;a)$ ($a\in\bbc$ is not an algebraic number), and $e\in{\ggg}_{\bar0}$ a nilpotent element. Fix an $\mathfrak{sl}_2$-triple $f,h,e$, and denote by ${\ggg}^e:=\text{Ker}(\ad\,e)$ in ${\ggg}$. The linear
operator ad\,$h$ defines a ${\bbz}$-grading ${\ggg}=\bigoplus\limits_{i\in{\bbz}}{\ggg}(i)$. Define the Kazhdan degree on ${\ggg}$ by declaring that $x\in{\ggg}(j)$ is of $(j+2)$. A finite $W$-superalgebra is defined by $$U({\ggg},e)=(\text{End}_{\ggg}Q_{\chi})^{\text{op}},$$where $Q_{\chi}$ is the generalized Gelfand-Graev ${\ggg}$-module associated to $e$. In \cite{ZS2} we showed that

\begin{theorem}\label{graded W}(\cite{ZS2})
Under the Kazhdan grading, we have
\begin{itemize}
\item[(1)] $\gr\,U(\ggg,e)\cong S(\ggg^e)$ as $\bbc$-algebras when $\dim\ggg(-1)_\bo$ is even;
\item[(2)] $ \gr\,U(\ggg,e)\cong S(\ggg^e)\otimes \bbc[\Theta]$ as vector spaces over $\bbc$ when  $\dim\ggg(-1)_\bo$ is odd,
\end{itemize}
where $\bbc[\Theta]$ is the exterior algebra generated by one element $\Theta$.
\end{theorem}

\subsection{}
The main purpose of this paper is to further develop the construction and representation theory of finite $W$-superalgebras both over the field of complex numbers and over a field in prime characteristic. The most important part is the accessibility of lower bounds in the super Kac-Weisfeiler property. Our approach is roughly generalizing the ``reduction modulo $p$'' method introduced by Premet for the finite $W$-algebra case in \cite{P7}, with careful analysis and examination on the variation of structural features arising from the parity of $\dim \ggg(-1)_\bo$. Let us explain it roughly as below.

Let ${\ggg}_{\bbk}$ and  $Q_{\chi,{\bbk}}$ be the modular counterparts of $\ggg$, and of the generalized Gelfand-Graev $\ggg$-module $Q_{\chi}$ respectively. To ease notation, we will identify the nilpotent element $e\in\ggg$ over $\bbc$ with the element $\bar{e}=e\otimes1$ in ${\ggg}_{\bbk}$ by ``reduction modulo $p$" in the following.   Define the finite $W$-superalgebra over ${\bbk}$ by $$U({\ggg}_{\bbk},e):=(\text{End}_{{\ggg}_{\bbk}}Q_{\chi,{\bbk}})^{\text{op}}.$$ 
Let $T({\ggg}_{\bbk},e)$ be the transition subalgebra of $U({\ggg}_{\bbk},e)$ which is derived from the ${\bbc}$-algebra $U({\ggg},e)$ by ``reduction modulo $p$''. In our arguments, the transition subalgebras will play some medium role between the  theory of finite $W$-superalgebras over $\bbc$ and that of reduced enveloping algebras over $\bbk$. So in the  first part of the paper, we will investigate the structure of the transition subalgebra  $T(\ggg_\bbk,e)$ and the finite $W$-superalgebra $U(\ggg_\bbk,e)$ over $\bbk$. Explicitly speaking, for any real number $a\in\mathbb{R}$, let $\lceil a\rceil$ denote the largest integer lower bound of $a$, and $\lfloor a\rfloor$ the least integer upper bound of $a$. It is notable that these notations also work for the  numbers in the prime field $\mathbb{F}_p$. Set $d_i:=\text{dim}{\ggg}_i-\text{dim}{\ggg}^e_i$ for $i\in\bbz_2=\{\bz,\bo\}$, then we have

\begin{theorem}\label{graded W2} There is a subspace $\aaa_\bbk$ of $\ggg_\bbk$  with
$\underline{\dim}\,\aaa_\bbk = ({\frac{d_0}{2}}|{\lceil\frac{d_1}{2}\rceil})$ such that
$U({\ggg}_{\bbk},e)\cong T({\ggg}_{\bbk},e)\otimes_{\bbk}Z_p(\aaa_{\bbk})$ as ${\bbk}$-algebras, where $Z_p(\aaa_\bbk)$ is the $p$-center as usually defined,  with respect to the subspace $\aaa_\bbk$.
\end{theorem}

In the second part, we exploit some remarkable applications of finite $W$-superalgebras to the modular representations of basic Lie superalgebras. We provide a super version of Premet's work, as afore-mentioned, on the accessibility of lower-bounds of dimensions in the modular representations of reductive Lie algebras predicted by  Kac-Weisfeiler conjecture. For this, we will formulate a conjecture about the ``small representations'' of $U({\ggg},e)$ over $\bbc$ (in the paper, we call a representation of an algebra ``small'' if its dimension is minimal among all the representations of this algebra). By \cite[Remark 2.7]{ZS2} we know that $d_1$ has the same parity with $\text{dim}\,{\ggg}(-1)_{\bar1}$. As seen before, the variation of the parity of $d_1$ gives rise to the change of the structure of finite $W$-superalgebras, and we will see furthermore, also of the representations of finite $W$-superalgebras. Firstly, we provide the following highly-plausible conjecture, generalizing a conjecture proposed by Premet on the representations of finite $W$-algebras which has been already confirmed (cf. \cite{P7} and \cite{P9}).

\begin{conj}\label{conjecture}
Let ${\ggg}$ be a basic Lie superalgebra over ${\bbc}$. Then the following statements hold:
\begin{itemize}
\item[(1)] when $d_1$ is even, the finite $W$-superalgebra $U({\ggg},e)$ affords a one-dimensional  representation;

\item[(2)] when $d_1$ is odd, the finite $W$-superalgebra $U({\ggg},e)$ affords a two-dimensional  representation.
\end{itemize}
\end{conj}
 For the case ${\ggg}$ is of type $A$, Conjecture \ref{conjecture} is confirmed in the present paper, which is accomplished by conversion from the verification of the attainableness of lower-bounds of modular dimensions for basic Lie superalgebras of the same type by some direct computation; see \cite{ZS3} for more details. However, our final result on the lower bounds of modular dimensions for basic Lie superalgebras is generally dependent on the above conjecture.

\subsection{} Under the assumption of Conjecture~\ref{conjecture},
we finally accomplish a super version of Premet's work on classical Lie algebras.  Explicitly speaking, let $\xi\in({\ggg}_{\bbk})^*_{\bar0}$ be any $p$-character of ${\ggg}_{\bbk}$ corresponding to an element $\bar x\in ({\ggg}_{\bbk})_{\bar0}$ such that $\xi(\bar y)=(\bar x,\bar y)$ for any $\bar y\in{\ggg}_{\bbk}$. Now let $d_0=\text{dim}\,({\ggg}_{\bbk})_{\bar 0}-\text{dim}\,({\ggg}_{\bbk}^{\bar x})_{\bar 0}$ and $d_1=\text{dim}\,({\ggg}_{\bbk})_{\bar 1}-\text{dim}\,({\ggg}_{\bbk}^{\bar x})_{\bar 1}$, where ${\ggg}_{\bbk}^{\bar x}$ denotes the centralizer of $\bar x$ in ${\ggg}_{\bbk}$. Recall that the dimension of any irreducible representation of $\ggg_{\bbk}$ is divisible by $p^{\frac{d_0}{2}}2^{\lfloor\frac{d_1}{2}\rfloor}$ (cf. \cite[Theorem 5.6]{WZ}). The main result of the present paper is  the following theorem.
\begin{theorem}\label{intromain-2} 
Let ${\ggg}_{\bbk}$ be a basic Lie superalgebra over ${\bbk}=\overline{\mathbb{F}}_p$, and let $\xi\in({\ggg}_{\bbk})^*_{\bar0}$. If Conjecture~\ref{conjecture} establishes for all the basic Lie superalgebras over $\bbc$, then for $p\gg0$ the reduced enveloping algebra $U_\xi({\ggg}_{\bbk})$ admits irreducible representations of dimension $p^{\frac{d_0}{2}}2^{\lfloor\frac{d_1}{2}\rfloor}$.
\end{theorem}

The main body of the proof of the above theorem will be to deal with the situation when $\xi$ is nilpotent. The arguments in Premet's work will be exploited in the super case in the present work. The greatest challenge here is to deal with the structural change arising from the variation of the parity of $\dim \ggg(-1)_\bo$. We sketch here some main ingredients in our proof, beyond exploiting Premet's arguments.  Based on the PBW structure theorems of finite $W$-superalgebras established in \cite{ZS2}, we first prove that  when $\dim \ggg(-1)_\bo$ is odd, the $\bbc$-algebra $U(\ggg,e)$ doesn't admit one-dimensional representations (Proposition \ref{no1}).
The possible two-dimensional modules of $U(\ggg,e)$ will turn to be of type $Q$, with parity involution arising from some special odd element $\Theta_{l+q+1}\in U(\ggg,e)$ only appearing in the case of $\dim \ggg(-1)_\bo$ being odd (Proposition \ref{typeq}), while what we need to deal with more are lots of arguments involving the generators and defining relations of $U(\ggg,e)$ associated with $\Theta_{l+q+1}$ (see \S\ref{4.2}).

As for the case when the $p$-character $\chi\in(\ggg_\bbk)^*_\bz$ is nilpotent, corresponding to a nilpotent element $ e\in(\ggg_\bbk)_\bz$ with respect to the non-degenerate bilinear form on $\ggg_\bbk$ (see \S\ref{2.2.1}) such that $\chi(\bar x)=(e,\bar x)$ for any $\bar x\in\ggg_\bbk$, the following theorem releases the condition in Theorem \ref{intromain-2}.

\begin{theorem}\label{intromain-3}
Let ${\ggg}_{\bbk}$ be a  basic Lie superalgebra over ${\bbk}=\overline{\mathbb{F}}_p$, and let $\chi\in({\ggg}_{\bbk})^*_{\bar0}$ be a nilpotent $p$-character, with respect to the element $e\in(\ggg_\bbk)_\bz$ as described above. If the corresponding finite $W$-superalgebra $U({\ggg},e)$ over ${\bbc}$ affords a one-dimensional  (resp. two-dimensional) representation when $d_1$ is even (resp. odd), then for $p\gg0$ the reduced enveloping algebra $U_\chi({\ggg}_{\bbk})$ admits irreducible representations of dimension $p^{\frac{d_0}{2}}2^{\lfloor\frac{d_1}{2}\rfloor}$, where $d_0=\text{dim}\,({\ggg}_{\bbk})_{\bar 0}-\text{dim}\,({\ggg}_{\bbk}^{e})_{\bar 0}$ and $d_1=\text{dim}\,({\ggg}_{\bbk})_{\bar 1}-\text{dim}\,({\ggg}_{\bbk}^{e})_{\bar 1}$.

\end{theorem}

\subsection{} The paper is organized as follows.
In \textsection\ref{2} some basics on algebraic supergroups, Lie superalgebras and finite $W$-superalgebras are recalled. In \textsection\ref{3} the transition subalgebra $T({\ggg}_{\bbk},e)$ over $\bbk$  is introduced and studied, then follows the structure relation among the finite $W$-superalgebras, the transition subalgebras and the $p$-central subalgebras of some subspaces.  In \S\ref{conjectureseceven} and \S\ref{conjecturesecodd},
the minimal dimensions for the representations of $U({\ggg},e)$ over ${\bbc}$ are estimated and  conjectured. We demonstrate that the conjecture is true for some  cases, including  the whole case of type $A$. The concluding section will be devoted to the proof of our main theorems.
In \S\ref{proofmain0.5sub} we first complete the proof of Theorem \ref{intromain-3}. In subsection  \S\ref{5.3}, we improve the result on the dimensional lower-bounds of modular representations for a direct sum of basic Lie superalgebras with nilpotent $p$-characters in Proposition \ref{sumdivisible} and Remark \ref{refine} (note that this conclusion does not depend on Conjecture \ref{conjecture}), which was originally discussed by Wang-Zhao in \cite[Remark 4.5]{WZ}. Then the accessibility of the lower bounds for the refined version is obtained under Conjecture \ref{conjecture}.  The subsection  \S\ref{proofmainthm0.4sub} is devoted to the proof  Theorem \ref{intromain-2}.
In virtue of the results obtained in \S\ref{5.3}, we further show that the lower bounds in the super Kac-Weisfeiler property with arbitrary $p$-characters in \cite{WZ} are also reachable under Conjecture \ref{conjecture}.  The main tool applied in this section is the method of nilpotent orbit theory, and also the techniques for the modular representation theory of restricted Lie superalgebras.

\subsection{} Throughout we work with the field of complex numbers ${\bbc}$, or the algebraically closed field ${\bbk}=\overline{\mathbb{F}}_p$ of positive characteristic $p$ as the ground field.

Let ${\bbz}_+$ be the set of all the non-negative integers in ${\bbz}$, and denote by ${\bbz}_2$ the residue class ring modulo $2$ in ${\bbz}$. A superspace is a ${\bbz}_2$-graded vector space $V=V_{\bar0}\oplus V_{\bar1}$, in which we call elements in $V_{\bar0}$ and $V_{\bar1}$ even and odd, respectively. Write $|v|\in{\bbz}_2$ for the parity (or degree) of $v\in V$, which is implicitly assumed to be ${\bbz}_2$-homogeneous. We will use the notations
$$\text{\underline{dim}}V=(\text{dim}V_{\bar0},\text{dim}V_{\bar1}),\quad\text{dim}V=\text{dim}V_{\bar0}+\text{dim}V_{\bar1}.$$
All Lie superalgebras ${\ggg}$ will be assumed to be finite-dimensional.

Recall that a superalgebra analogue  of Schur's Lemma states that the endomorphism ring of an irreducible module of a superalgebra is either one-dimensional  or two-dimensional  (in the latter case it is isomorphic to a Clifford algebra), cf. for example, Kleshchev \cite[Chapter 12]{KL}. An irreducible module is of type $M$ if its endomorphism ring is one-dimensional  and it is of type $Q$ otherwise.

By vector spaces, subalgebras, ideals, modules, and submodules etc., we mean in the super sense unless otherwise specified, throughout the paper.

\section{Basic Lie superalgebras and finite $W$-superalgebras}\label{2}

In this section, we will recall some knowledge on basic classical Lie superalgebras along with the corresponding algebraic supergroups, and finite $W$-(super)algebras for use in the sequel. We refer the readers to  \cite{CW}, \cite{K} and \cite{M} for Lie superalgebras,  \cite{FG} and \cite{SW} for algebraic supergroups, and \cite{P2}, \cite{P7}, \cite{W} and \cite{ZS2} for finite $W$-(super)algebras.

\subsection{Basic classical Lie superalgebras and the corresponding algebraic supergroups}\label{prel}
\subsubsection{Basic classical Lie superalgebras} \label{basicLiesuper}
Following \cite[\S1]{CW}, \cite[\S2.3-\S2.4]{K}, \cite[\S1]{K2} and \cite[\S2]{WZ},  we recall the list of basic classical Lie superalgebras over $\bbf$ for $\bbf=\bbc$ or $\bbf=\bbk$.
These Lie superalgebras, with even parts being Lie algebras of reductive algebraic groups, are simple over $\bbf$ (the
general linear Lie superalgebras, though not simple, are also included), and they admit an even non-degenerate supersymmetric invariant bilinear form in the following sense.
\begin{defn}\label{form}
Let $V=V_{\bar0}\oplus V_{\bar1}$ be a $\mathbb{Z}_2$-graded space and $(\cdot,\cdot)$ be a bilinear form on $V$.
\begin{itemize}
\item[(1)] If $(a,b)=0$ for any $a\in V_{\bar0}, b\in V_{\bar1}$, then $(\cdot,\cdot)$ is called even;
\item[(2)] if $(a,b)=(-1)^{|a||b|}(b,a)$ for any homogeneous elements $a,b\in V$, then $(\cdot,\cdot)$ is called supersymmetric;
\item[(3)] if $([a,b],c)=(a,[b,c])$ for any homogeneous elements $a,b,c\in V$, then $(\cdot,\cdot)$ is called invariant;
\item[(4)] if one can conclude from $(a,V)=0$ that $a=0$, then $(\cdot,\cdot)$ is called non-degenerate.
\end{itemize}
\end{defn}

Note that when $\bbf=\bbk$ is a field of characteristic $p>0$, there are restrictions on $p$, for example, as shown in \cite[Table 1]{WZ}. So we have the following list

\vskip0.3cm
\begin{center}\label{Table 1}

({\sl{Table 1}}): basic classical Lie $\bbk$-superalgebras
\vskip0.3cm
\begin{tabular}{ccc}
\hline
 $\frak{g}$ & $\ggg_{\bar 0}$  & the restriction of $p$ when $\bbf=\bbk$\\
\hline
$\frak{gl}(m|n$) &  $\frak{gl}(m)\oplus \frak{gl}(n)$                &$p>2$            \\
$\frak{sl}(m|n)$ &  $\frak{sl}(m)\oplus \frak{sl}(n)\oplus \bbk$    & $p>2, p\nmid (m-n)$   \\
$\frak{osp}(m|n)$ & $\frak{so}(m)\oplus \frak{sp}(n)$                  & $p>2$ \\
$\text{D}(2,1,\bar a)$   & $\frak{sl}(2)\oplus \frak{sl}(2)\oplus \frak{sl}(2)$        & $p>3$   \\
$\text{F}(4)$            & $\frak{sl}(2)\oplus \frak{so}(7)$                  & $p>15$  \\
$\text{G}(3)$            & $\frak{sl}(2)\oplus \text{G}_2$                    & $p>15$     \\
\hline
\end{tabular}

\end{center}

\vskip0.3cm

Throughout the paper, we will simply call all $\ggg$ listed above {\sl{``basic Lie superalgebras"}}.

\subsubsection{Algebraic supergroups and restricted Lie superalgebras}\label{2.1}
For a given basic Lie superalgebra listed in \S\ref{basicLiesuper}, there is an algebraic supergroup $G_\bbk$  satisfying $\Lie(G_\bbk)=\ggg_\bbk$ such that
\begin{itemize}
\item[(1)] $G_\bbk$ has a subgroup scheme $(G_\bbk)_\ev$ which is an ordinary connected reductive group with $\Lie((G_\bbk)_\ev)=(\ggg_\bbk)_\bz$;
\item[(2)] there is a well-defined action of $(G_\bbk)_\ev$ on $\ggg_\bbk$, reducing to the adjoint action of $(\ggg_\bbk)_\bz$.
    \end{itemize}
The above algebraic supergroup can be constructed as a Chevalley supergroup in \cite{FG}. The pair ($(G_\bbk)_\ev, \ggg_\bbk)$ constructed in this way is called a Chevalley super Harish-Chandra pair (cf.  \cite[Theorem 5.35]{FG} and \cite[\S3.3]{FG2}).
Partial results on $G_\bbk$ and $(G_\bbk)_\ev$ can be found in \cite[Ch. 2.2]{Ber}, \cite{FG}, \cite[\S3.3]{FG2} {\sl etc.}. In the present paper, we will call $(G_\bbk)_\ev$ the purely even subgroup of $G_\bbk$. When the ground field is  $\bbk$ of odd prime characteristic $p$, one easily knows that $\ggg_\bbk$ is a restricted Lie superalgebra (cf. \cite[Definition 2.1]{SW} and \cite{SZ}) in the following sense.

\begin{defn}\label{restricted}
A Lie superalgebra ${\ggg}_{\bbk}=({\ggg}_{\bbk})_{\bar{0}}\oplus({\ggg}_{\bbk})_{\bar{1}}$ over ${\bbk}$ is called a restricted Lie superalgebra,
if there is a $p$-th power map $({\ggg}_{\bbk})_{\bar{0}}\rightarrow({\ggg}_{\bbk})_{\bar{0}}$, denoted as $(-)^{[p]}$, satisfying

(a) $(k\bar x)^{[p]}=k^p\bar x^{[p]}$ for all $k\in{\bbk}$ and $\bar x\in({\ggg}_{\bbk})_{\bar{0}}$;

(b) $[\bar x^{[p]},\bar y]=(\text{ad}\bar x)^p(\bar y)$ for all $\bar x\in({\ggg}_{\bbk})_{\bar{0}}$ and $\bar y\in{\ggg}_{\bbk}$;

(c) $(\bar x+\bar y)^{[p]}=\bar x^{[p]}+\bar y^{[p]}+\sum\limits_{i=1}^{p-1}s_i(\bar x,\bar y)$ for all $\bar x,\bar y\in({\ggg}_{\bbk})_{\bar{0}}$, where $is_i(\bar x,\bar y)$ is the
coefficient of $\lambda^{i-1}$ in $(\text{ad}(\lambda \bar x+\bar y))^{p-1}(\bar x)$.
\end{defn}

Let $\ggg_\bbk$ be a restricted Lie superalgebra.  For each $\bar x\in (\ggg_\bbk)_\bz$, the element $\bar x^p-\bar x^{[p]}\in U(\ggg_\bbk)$ is central by Definition \ref{restricted}, and all of which generate a central subalgebra of $U(\ggg_\bbk)$.
Let $\{\bar w_1,\cdots,\bar w_c\}$ and $\{\bar w'_1,\cdots,\bar w'_d\}$ be the basis of $(\ggg_\bbk)_\bz$ and $(\ggg_\bbk)_\bo$ respectively.
For a given $\chi\in (\ggg_\bbk)_\bz^*$,  let $J_\chi$ be the ideal of the universal enveloping algebra $U(\ggg_\bbk)$ generated by the even central elements $\bar w^p-\bar w^{[p]}-\chi(\bar w)^p$ for all $\bar w\in (\ggg_\bbk)_\bz$. The quotient algebra $U_\chi(\ggg_\bbk) := U(\ggg_\bbk)\slash J_\chi$ is called the reduced enveloping algebra with $p$-character $\chi$. We often regard $\chi\in \ggg_\bbk^*$ by letting $\chi((\ggg_\bbk)_\bo) = 0$. If $\chi= 0$, then $U_0(\ggg_\bbk)$ is called
the restricted enveloping algebra. It is a direct consequence from the PBW theorem that the $\bbk$-algebra $U_\chi(\ggg_\bbk)$ is of dimension $p^c 2^d$, and has a basis
$$\{\bar w_1^{a_1}\cdots\bar w_c^{a_c}(\bar w'_1)^{b_1}\cdots(\bar w'_d)^{b_d}
\mid 0\leqslant a_i<p, b_j \in\{0, 1\} \mbox{ for  all }1\leqslant i\leqslant c, 1\leqslant j\leqslant d\}.$$

\subsection{Finite $W$-superalgebras over the field of complex numbers}\label{2.2}
\subsubsection{}\label{2.2.1}
Let ${\ggg}$ be a basic Lie superalgebra over ${\bbc}$ and $\mathfrak{h}$ be a typical Cartan subalgebra of ${\ggg}$. Let $\Phi$ be a root system of ${\ggg}$ relative to $\mathfrak{h}$ whose simple roots system $\Delta=\{\alpha_1,\cdots,\alpha_l\}$ is distinguished (cf. \cite[Proposition 1.5]{K2}). Let $\Phi^+$ be the corresponding positive system in $\Phi$, and put $\Phi^-:=-\Phi^+$. Let ${\ggg}=\mathfrak{n}^-\oplus\mathfrak{h}\oplus\mathfrak{n}^+$ be the corresponding triangular decomposition of ${\ggg}$. By \cite[\S3.3]{FG}, we can choose a Chevalley basis $B=\{e_\gamma\mid\gamma\in\Phi\}\cup\{h_\alpha\mid\alpha\in\Delta\}$ of ${\ggg}$ excluding the case $D(2,1;a) (a\notin{\bbz}$). Let ${\ggg}_{\bbz}$ denote the Chevalley ${\bbz}$-form in ${\ggg}$ and $U_{\bbz}$ the Kostant ${\bbz}$-form of $U({\ggg})$ associated to $B$. Given a ${\bbz}$-module $V$ and a ${\bbz}$-algebra $A$, we write $V_A:=V\otimes_{\bbz}A$.

Let $G$ be the algebraic supergroup of ${\ggg}$, then the even part of $G$ is reductive and denote it by $G_{\text{ev}}$, with $(G_{\text{ev}}, \ggg)$ being a super Harish-Chandra pair.  Let $e\in{\ggg}_{\bar0}$ be a nilpotent element in ${\ggg}$. By the Dynkin-Kostant theory, $\text{ad}\,G_{\text{ev}}.e$ interacts with $({\ggg}_{\bbz})_{\bar{0}}$ nonempty. Without loss of generality, one can assume that the nilpotent $e$ is in $({\ggg}_{\bbz})_{\bar{0}}$. By the same discussion as \cite[\S4.2]{P2}, for any given nilpotent element $e\in({\ggg}_{\bbz})_{\bar{0}}$ we can find $f,h\in({\ggg}_\mathbb{Q})_{\bar{0}}$ such that $(e,h,f)$ is an $\mathfrak{sl}_2$-triple in ${\ggg}$. \cite[Proposition 2.1]{ZS2} shows that there exists an even non-degenerate supersymmetric invariant bilinear form $(\cdot,\cdot)$, under which the Chevalley basis $B$ of ${\ggg}$ takes value in $\mathbb{Q}$, and $(e,f)=1$. Define $\chi\in{\ggg}^{*}$ by letting $\chi(x)=(e,x)$ for all $x\in{\ggg}$, then we have $\chi({\ggg}_{\bar{1}})=0$.

Following \cite[Definition 2.4]{ZS2} we call a commutative ring $A$ {\sl admissible} if $A$ is a finitely generated ${\bbz}$-subalgebra of ${\bbc}$, $(e,f)\in A^{\times}(=A\backslash \{0\})$ and all bad primes of the root system of ${\ggg}$ and the determinant of the Gram matrix of ($\cdot,\cdot$) relative to a Chevalley basis of ${\ggg}$ are invertible in $A$. It is clear by the definition that every admissible ring is a Noetherian domain. Given a finitely generated ${\bbz}$-subalgebra $A$ of ${\bbc}$, denote by $\text{Specm}A$ the maximal spectrum of $A$. It is well known that for every element $\mathfrak{P}\in\text{Specm}A$, the residue field $A/\mathfrak{P}$ is isomorphic to $\mathbb{F}_{q}$, where $q$ is a $p$-power depending on $\mathfrak{P}$. We denote by $\Pi(A)$ the set of all primes $p\in\mathbb{N}$ that occur in this way. Since the choice of $A$ does not depend on the super structure of ${\ggg}$, it follows from the arguments in the proof of \cite[Lemma 4.4]{P7} that the set $\Pi(A)$ contains almost all primes in $\mathbb{N}$.

Let ${\ggg}(i)=\{x\in{\ggg}\mid[h,x]=ix\}$, then ${\ggg}=\bigoplus_{i\in{\bbz}}{\ggg}(i)$. By the $\mathfrak{sl}_2$-theory, all subspaces ${\ggg}(i)$ are defined over $\mathbb{Q}$. Also, $e\in{\ggg}(2)_{\bar{0}}$ and $f\in{\ggg}(-2)_{\bar{0}}$. There exists a symplectic (resp. symmetric) bilinear form $\langle\cdot,\cdot\rangle$ on the ${\bbz}_2$-graded subspace ${\ggg}(-1)_{\bar{0}}$ (resp. ${\ggg}(-1)_{\bar{1}}$) given by $\langle x,y\rangle:=(e,[x,y])=\chi([x,y])$ for all $x,y\in{\ggg}(-1)_{\bar0}~(\text{resp.}\,x,y\in{\ggg}(-1)_{\bar1})$.
There exist bases $\{u_1,\cdots,u_{2s}\}$ of ${\ggg}(-1)_{\bar0}$ and $\{v_1,\cdots,v_r\}$ of ${\ggg}(-1)_{\bar1}$    contained in ${\ggg}_\mathbb{Q}$ such that $\langle u_i, u_j\rangle =i^*\delta_{i+j,2s+1}$ for $1\leqslant i,j\leqslant 2s$, where $i^*=\left\{\begin{array}{ll}-1&\text{if}~1\leqslant i\leqslant s;\\ 1&\text{if}~s+1\leqslant i\leqslant 2s\end{array}\right.$, and $\langle v_i,v_j\rangle=\delta_{i+j,r+1}$ for $1\leqslant i,j\leqslant r$.

Set $\mathfrak{m}:=\bigoplus_{i\leqslant -2}{\ggg}(i)\oplus{\ggg}(-1)^{\prime}$ with ${\ggg}(-1)^{\prime}={\ggg}(-1)^{\prime}_{\bar0}\oplus{\ggg}(-1)^{\prime}_{\bar1}$, where ${\ggg}(-1)^{\prime}_{\bar0}$ is the ${\bbc}$-span of $u_{s+1},\cdots,u_{2s}$ and
${\ggg}(-1)^{\prime}_{\bar1}$ is the ${\bbc}$-span of $v_{\frac{r}{2}+1},\cdots,v_r$ (resp. $v_{\frac{r+3}{2}},\cdots,v_r$) when $r=\text{dim}\,{\ggg}(-1)_{\bar{1}}$ is even (resp. odd), then $\chi$ vanishes on the derived subalgebra of $\mathfrak{m}$. Define ${\ppp}:=\bigoplus_{i\geqslant 0}{\ggg}(i)$, $\mathfrak{m}^{\prime}:=\left\{\begin{array}{ll}\mathfrak{m}&\text{if}~r~\text{is even;}\\
\mathfrak{m}\oplus {\bbc}v_{\frac{r+1}{2}}&\text{if}~r~\text{is odd.}\end{array}\right.$

Write ${\ggg}^e$ for the centralizer of $e$ in ${\ggg}$ and denote by $d_i:=\text{dim}\,{\ggg}_i-\text{dim}\,{\ggg}^e_i$ for $i\in{\bbz}_2$, then $r:=\dim\ggg(-1)_{\bar1}$ and $d_1$ always have the same parity by \cite[Theorem 4.3]{WZ}. This parity is a crucial factor deciding the structure of the finite $W$-superalgebras (cf. \cite[Theorem 4.5]{ZS2}), which is called the judging parity in \cite{ZS2}. We further have
$$ \text{\underline{dim}}\,\mathfrak{m}=\begin{cases}  (\frac{d_0}{2},\frac{d_1}{2}) &\mbox{ if }d_1\mbox{ is even;}\cr
   (\frac{d_0}{2},\frac{d_1-1}{2}) &\mbox{ if }d_1\mbox{ is odd.}
\end{cases}$$
    After enlarging $A$ one can assume that ${\ggg}_A=\bigoplus_{i\in{\bbz}}{\ggg}_A(i)$, and each ${\ggg}_A(i):={\ggg}_A\cap{\ggg}(i)$ is freely generated over $A$ by a basis of the vector space ${\ggg}(i)$. Then $\{u_1,\cdots,u_{2s}\}$ and $\{v_1,\cdots,v_r\}$ are free basis of $A$-module ${\ggg}_A(-1)_{\bar0}$ and  ${\ggg}_A(-1)_{\bar1}$, respectively. It is obvious that
$\mathfrak{m}_A:={\ggg}_A\cap\mathfrak{m}$, $\mathfrak{m}^{\prime}_A:={\ggg}_A\cap\mathfrak{m}^{\prime}$ and ${\ppp}_A:={\ggg}_A\cap{\ppp}$ are free $A$-modules and direct summands of ${\ggg}_A$. Moreover, one can assume $e,f\in({\ggg}_A)_{\bar0}$ after enlarging $A$ possibly; $[e,{\ggg}_A(i)]$ and $[f,{\ggg}_A(i)]$ are direct summands of ${\ggg}_A(i+2)$ and ${\ggg}_A(i-2)$ respectively, and ${\ggg}_A(i+2)=[e,{\ggg}_A(i)]$ for each $i\geqslant -1$ by the $\mathfrak{sl}_2$-theory.

As in \cite[\S2]{ZS2} we can choose a basis $x_1,\cdots,x_l,x_{l+1},\cdots,x_m\in({\ppp}_A)
_{\bar{0}}, y_1,\cdots, y_q, $\\$y_{q+1}, \cdots,y_n\in({\ppp}_A)_{\bar{1}}$ of the free $A$-module ${\ppp}_A=\bigoplus_{i\geqslant 0}{\ggg}_A(i)$ such that

(a) $x_i\in{\ggg}_A(k_i)_{\bar{0}}, y_j\in{\ggg}_A(k'_j)_{\bar{1}}$, where $k_i,k'_j\in{\bbz}_+$ with $1\leqslant i\leqslant m$ and $1\leqslant j\leqslant n$;

(b) $x_1,\cdots,x_l$ is a basis of $({\ggg}_A)^e_{\bar{0}}$ and $y_1,\cdots,y_q$ is a basis of $({\ggg}_A)^e_{\bar{1}}$;

(c) $x_{l+1},\cdots,x_m\in[f,({\ggg}_A)_{\bar{0}}]$ and $ y_{q+1},\cdots,y_n\in[f,({\ggg}_A)_{\bar{1}}]$.

\subsubsection{}\label{2.2.2}  Recall that a Gelfand-Graev ${\ggg}$-module associated to $\chi$ is defined by $$Q_\chi:=U({\ggg})\otimes_{U(\mathfrak{m})}{\bbc}_\chi,$$
 where ${\bbc}_\chi={\bbc}1_\chi$ is a one-dimensional  $\mathfrak{m}$-module such that $x.1_\chi=\chi(x)1_\chi$ for all $x\in\mathfrak{m}$. For $k\in\mathbb{Z}_+$, define
 \begin{equation*}
 \begin{array}{llllll}
 \mathbb{Z}_+^k&:=&\{(i_1,\cdots,i_k)\mid i_j\in\mathbb{Z}_+\},&
 \Lambda'_k&:=&\{(i_1,\cdots,i_k)\mid i_j\in\{0,1\}\}
 \end{array}
 \end{equation*}with $1\leqslant j\leqslant k$. For $\mathbf{i}=(i_1,\cdots,i_k)$ in $\mathbb{Z}_+^k$ or $\Lambda'_k$, set $|\mathbf{i}|=i_1+\cdots+i_k$. For any real number $a\in\mathbb{R}$, let $\lceil a\rceil$ denote the largest integer lower bound of $a$, and $\lfloor a\rfloor$ the least integer upper bound of $a$. Given $(\mathbf{a},\mathbf{b},\mathbf{c},\mathbf{d})\in{\bbz}^m_+\times\Lambda'_n\times{\bbz}^s_+\times\Lambda'_t$ (where $t:=\lfloor\frac{r}{2}\rfloor=\lfloor\frac{\text{dim}\,{\ggg}(-1)_{\bar1}}{2}\rfloor$), let $x^\mathbf{a}y^\mathbf{b}u^\mathbf{c}v^\mathbf{d}$ denote the monomial $x_1^{a_1}\cdots x_m^{a_m}y_1^{b_1}\cdots y_n^{b_n}u_1^{c_1}\cdots u_s^{c_s}v_1^{d_1}\cdots v_t^{d_t}$ in $U({\ggg})$. Set $Q_{\chi,A}:=U({\ggg}_A)\otimes_{U(\mathfrak{m}_A)}A_\chi$, where $A_\chi=A1_\chi$. By the definition $Q_{\chi,A}$ is a ${\ggg}_A$-stable $A$-lattice in $Q_{\chi}$ with $\{x^\mathbf{a}y^\mathbf{b}u^\mathbf{c}v^\mathbf{d}\otimes1_\chi\mid(\mathbf{a},\mathbf{b},\mathbf{c},\mathbf{d})\in{\bbz}^m_+\times\Lambda'_n\times{\bbz}^s_+\times\Lambda'_t\}$
as a free basis. Given $(\mathbf{a},\mathbf{b},\mathbf{c},\mathbf{d})\in{\bbz}_+^m\times\Lambda'_n\times{\bbz}_+^s\times\Lambda'_t$, set
$$|(\mathbf{a},\mathbf{b},\mathbf{c},\mathbf{d})|_e:=\sum_{i=1}^ma_i(k_i+2)+\sum_{i=1}^nb_i(k'_i+2)+\sum_{i=1}^sc_i+\sum_{i=1}^td_i.$$
Set
\[Y_i:=\left\{
\begin{array}{ll}
x_i&\text{if}~1\leqslant  i\leqslant  l;\\
y_{i-l}&\text{if}~l+1\leqslant  i\leqslant  l+q;\\
v_{\frac{r+1}{2}}&\text{if}~i=l+q+1,
\end{array}
\right.
\]
where $Y_i\in{\ggg}(m_i)$ with $m_i\in{\bbz}$ and the term $Y_{l+q+1}$ occurs only when $r=\dim\ggg(-1)_\bo$ is odd.
Then there are Lie superalgebra operator identities in the setting of $[Y_i,Y_j]=\sum\limits_{k=1}^{l+q}\alpha_{ij}^kY_k$ in ${\ggg}^e$ for $1\leqslant i,j\leqslant l+q$.
 Set $q'=q$ if $r$ (or equivalently, $d_1$) is even, and $q'=q+1$ if $r$ is odd.
 By \cite[Theorem 4.7]{ZS2}, the finite $W$-superalgebra $U({\ggg},e):=(\text{End}_{\ggg}Q_{\chi})^{\text{op}}$ can be determined by a data of generators and defining relations. Those generators are $\Theta_1,\cdots,\Theta_{l}\in U(\ggg,e)_\bz$ and $\Theta_{l+1},\cdots,\Theta_{l+q'}\in U(\ggg,e)_\bo$ with
 \begin{equation}\label{generators}
\begin{array}{ll}
&\Theta_k(1_\chi)\\=&(Y_k+\sum\limits_{\mbox{\tiny $\begin{array}{c}|\mathbf{a},\mathbf{b},\mathbf{c},\mathbf{d}|_e=m_k+2,\\|\mathbf{a}|
+|\mathbf{b}|+|\mathbf{c}|+|\mathbf{d}|\geqslant 2\end{array}$}}\lambda^k_{\mathbf{a},\mathbf{b},\mathbf{c},\mathbf{d}}x^{\mathbf{a}}
y^{\mathbf{b}}u^{\mathbf{c}}v^{\mathbf{d}}+\sum\limits_{|\mathbf{a},\mathbf{b},\mathbf{c},\mathbf{d}|_e<m_k+2}\lambda^k_{\mathbf{a},\mathbf{b},\mathbf{c},\mathbf{d}}x^{\mathbf{a}}
y^{\mathbf{b}}u^{\mathbf{c}}v^{\mathbf{d}})\otimes1_\chi
\end{array}
\end{equation}
for $1\leqslant  k\leqslant  l+q$, where $\lambda^k_{\mathbf{a},\mathbf{b},\mathbf{c},\mathbf{d}}\in\mathbb{Q}$, and $\lambda^k_{\mathbf{a},\mathbf{b},\mathbf{c},\mathbf{d}}=0$ if $a_{l+1}=\cdots=a_m=b_{q+1}=\cdots=b_n=c_1=\cdots=c_s=
d_1=\cdots=d_{\lceil\frac{r}{2}\rceil}=0$. When $r$ is odd, set $\Theta_{l+q+1}(1_\chi)=Y_{l+q+1}\otimes1_\chi=v_{\frac{r+1}{2}}\otimes1_\chi$.

By \cite[Theorem 4.5]{ZS2}, the monomials $\Theta_1^{a_1}\cdots\Theta_l^{a_l}\Theta_{l+1}^{b_1}\cdots\Theta_{l+q'}^{b_{q'}}$ with $a_i\in{\bbz}_+, b_j\in\{0,1\}$ for $1\leqslant i\leqslant l$ and $1\leqslant j\leqslant q'$ form a basis of the vector space $U({\ggg},e)$. The monomial $\Theta_1^{a_1}\cdots\Theta_l^{a_l}\Theta_{l+1}^{b_1}\cdots
\Theta_{l+q'}^{b_{q'}}$ is said to have Kazhdan degree
$\sum\limits_{i=1}^la_i(m_i+2)+\sum\limits_{i=1}^{q'}b_i(m_{l+i}+2)$. For $k\in{\bbz}_+$, let $\text{F}_kU({\ggg},e)$ denote the ${\bbc}$-span of all monomials $\Theta_1^{a_1}\cdots\Theta_l^{a_l}\Theta_{l+1}^{b_1}\cdots\Theta_{l+q'}^{b_{q'}}$ of Kazhdan degree $\leqslant k$. The subspaces $\text{F}_kU({\ggg},e)$ with $k\geqslant 0$ form an increasing exhaustive filtration of the algebra $U({\ggg},e)$, which is called the Kazhdan filtration. The corresponding graded algebra gr\,$U({\ggg},e)$ is a polynomial superalgebra in $\text{gr}\,\Theta_1,\cdots,\text{gr}\,\Theta_{l+q'}$. Recall \cite[Theorem 4.7]{ZS2} shows that there are super-polynomials $F_{ij}$'s with $i,j=1,\cdots,l+q'$ such that the defining relations on those generators can be described as
\begin{align}\label{refine1}
[\Theta_i,\Theta_j]=F_{ij}(\Theta_1,\cdots,\Theta_{l+q'}),\quad i,j=1,\cdots,l+q',
\end{align}
while
\begin{align}\label{fij}
F_{ij}(\Theta_1,\cdots,\Theta_{l+q'})\equiv\sum_{k=1}^{l+q}\alpha_{ij}^k\Theta_k+q_{ij}(\Theta_1,\cdots,\Theta_{l+q'})\,(\mbox{mod }\text{F}_{m_i+m_j+1}U({\ggg},e))
\end{align} for $i,j=1,\cdots,l+q$. For the case when one of the indices $i,j$ equals $l+q+1$, it follows from \cite[Remark 3.8]{ZS2} that there are no obvious formulas for the super-polynomials $F_{ij}$'s as shown in \eqref{fij}. However, by the same discussion as \cite[Theorem 4.5]{ZS2}, one can still choose the super-polynomials $F_{ij}$'s properly such that the Kazhdan degree for all the monomials in the $\bbc$-span of $F_{ij}$'s is less than $m_i+m_j+2$. Moreover,
\begin{align}\label{fij2}
F_{l+q+1,l+q+1}(\Theta_1,\cdots,\Theta_{l+q+1})=1\otimes1_\chi
\end{align}
when $r$ is odd.

In fact, some of the defining relations in \eqref{refine1} are equivalent to each other. By the same discussion as \cite[Remark 3.8(4)]{ZS2}, after deleting all the redundant commutating relations in \eqref{refine1}, the remaining ones are with indices $i,j$ satisfying $1\leqslant i<j\leqslant l$, $l+1\leqslant i\leqslant j\leqslant l+q'$, and $1\leqslant i\leqslant l<j\leqslant l+q'$.

In the sequent arguments, when we consider the corresponding counterparts of all above over the algebraic closured field ${\bbk}=\overline{\mathbb{F}}_p$ of positive characteristic $p$, {\sl we assume that the prime $p$ is large enough such that the admissible ring $A$ contains all $\lambda^k_{\mathbf{a},\mathbf{b},\mathbf{c},\mathbf{d}}$ in \eqref{generators} and all coefficients of the $F_{ij}$'s in \eqref{refine1}}, thereby the ``admissible procedure" developed by Premet for finite $W$-algebras can be reproduced in the super case.

For ${\bf a}=(a_1,\cdots,a_l)\in\bbz_+^l$ and ${\bf b}=(b_1,\cdots,b_{q'})\in \Lambda'_{q'}$, let
$U({\ggg}_A,e)$ be the $A$-span of the monomials
$$\{\Theta_1^{a_1}\cdots\Theta_l^{a_l}\cdot\Theta_{l+1}^{b_1}\cdots\Theta_{l+q'}^{b_{q'}}\mid ({\bf{a}},{\bf{b}})\in \bbz_+^l\times\Lambda'_{q'}\}.$$

\subsubsection{}\label{2.2.3}
Let $I_\chi$ denote the ${\bbz}_2$-graded ideal in $U({\ggg})$ generated by all $x-\chi(x)$ with $x\in\mathfrak{m}_i, i\in{\bbz}_2$. By construction, $I_\chi$ is a $(U({\ggg}),U(\mathfrak{m}))$-bimodule. The fixed point space $(U({\ggg})/I_\chi)^{\ad\,\mmm}$ carries a natural algebra structure given by $(x+I_\chi)\cdot(y+I_\chi):=(xy+I_\chi)$ for all $x,y\in U({\ggg})$. Then $Q_\chi\cong U({\ggg})/I_\chi$ as ${\ggg}$-modules via the ${\ggg}$-module map sending $1+I_\chi$ to $1_\chi$, and $Q_{\chi}^{\ad\,\mmm}\cong U({\ggg},e)$ as $\bbc$-algebras. Any element of $U({\ggg},e)$ is uniquely determined by its effect on the generator $1_\chi\in Q_\chi$, and it follows from \cite[Theorem 2.12]{ZS2} that the canonical isomorphism between $Q_{\chi}^{\ad\,\mmm}$ and $U({\ggg},e)$ is given by $u\mapsto u(1_\chi)$ for any $u\in Q_{\chi}^{\ad\,\mmm}$. In what follows we will often identify $Q_\chi$ with $U({\ggg})/I_\chi$ and $U({\ggg},e)$ with $Q_{\chi}^{\ad\,\mmm}$.

Let $w_1,\cdots, w_c$ be a basis of $\ggg$ over $\bbc$. Let $U({\ggg})=\bigcup_{i\in{\bbz}}\text{F}_iU({\ggg})$ be the Kazhdan filtration of $U({\ggg})$, where $\text{F}_iU({\ggg})$ is the ${\bbc}$-span of all $w_1\cdots w_c$ with $w_1\in{\ggg}(j_1),\cdots,w_c\in{\ggg}(j_c)$ and $(j_1+2)+\cdots+(j_c+2)\leqslant  i$. The Kazhdan filtration on $Q_{\chi}$ is defined by $\text{F}_iQ_{\chi}:=\pi(\text{F}_iU({\ggg}))$ with $\pi:U({\ggg})\twoheadrightarrow U({\ggg})/I_\chi$ being the canonical homomorphism, which makes $Q_{\chi}$ into a filtrated $U({\ggg})$-module. The  Kazhdan filtration of $Q_{\chi}$ has no negative components, and the Kazhdan filtration of $U({\ggg},e)$ defined above is nothing but the filtration of $U({\ggg},e)=Q_{\chi}^{\ad\,\mmm}$ induced from the Kazhdan filtration of $Q_{\chi}$ through the embedding $Q_{\chi}^{\ad\,\mmm}\hookrightarrow Q_{\chi}$ (see \cite[\S2.3]{ZS2}).

\subsection{Finite $W$-superalgebras in positive characteristic}\label{2.3}
\subsubsection{}\label{2.3.1}
Pick a prime $p\in\Pi(A)$ and denote by ${\bbk}=\overline{\mathbb{F}}_p$ the algebraic closure of $\mathbb{F}_p$. Since the bilinear form $(\cdot,\cdot)$ is $A$-valued on ${\ggg}_A$, it induces a bilinear form on the Lie superalgebra ${\ggg}_{\bbk}\cong{\ggg}_A\otimes_A{\bbk}$. In the following we still denote this bilinear form by $(\cdot,\cdot)$. If we denote by $G_{\bbk}$ the algebraic ${\bbk}$-supergroup of distribution algebra $U_{\bbk}=U_{\bbz}\otimes_{\bbz}{\bbk}$, then ${\ggg}_{\bbk}=\text{Lie}(G_{\bbk})$ (cf. \cite[\S{I}.7.10]{J} and \cite[\S2.2]{SW}). Note that the bilinear form $(\cdot,\cdot)$ is non-degenerate and $\text{Ad}\,(G_{\bbk})_{\text{ev}}$-invariant. For $x\in{\ggg}_A$, set $\bar{x}:=x\otimes1$, an element of ${\ggg}_{\bbk}$. To ease notation we identify $e,f,h$ with the nilpotent elements $\bar{e}=e\otimes1,~\bar{f}=f\otimes1$ and $\bar{h}=h\otimes1$ in ${\ggg}_{\bbk}$, and $\chi$ with the linear function $(e,\cdot)$ on ${\ggg}_{\bbk}$.

The Lie superalgebra ${\ggg}_{\bbk}=\text{Lie}\,(G_{{\bbk}})$ carries a natural $p$-mapping $x\mapsto x^{[p]}$ for all $x\in({\ggg}_{\bbk})_{\bar0}$, equivariant under the adjoint action of $(G_{{\bbk}})_{\text{ev}}$. The subalgebra of $U({\ggg}_{\bbk})$ generated by all $\bar x^p-\bar x^{[p]}$ with $\bar x\in({\ggg}_{\bbk})_{\bar{0}}$ is called the $p$-center of $U({\ggg}_{\bbk})$ and we denote it by $Z_p({\ggg}_{\bbk})$ for short.  It follows from the PBW theorem of $U({\ggg}_{\bbk})$ that $Z_p({\ggg}_{\bbk})$ is isomorphic to a polynomial algebra (in the usual sense) in $\text{dim}\,({\ggg}_{\bbk})_{\bar{0}}$ variables. For every maximal ideal $H$ of $Z_p({\ggg}_{\bbk})$ there is a unique linear function $\eta=\eta_H\in({\ggg}_{\bbk})_{\bar{0}}^*$ such that $$H=\langle \bar x^p-\bar x^{[p]}-\eta(\bar x)^p\mid\bar x\in({\ggg}_{\bbk})_{\bar{0}}\rangle.$$
Since the Frobenius map of ${\bbk}$ is bijective, this enables us to identify the maximal spectrum $\text{Specm}(Z_p({\ggg}_{\bbk}))$ of $Z_p({\ggg}_{\bbk})$ with $({\ggg}_{\bbk})_{\bar{0}}^*$.

For any $\xi\in({\ggg}_{\bbk})_{\bar{0}}^*$ we write $J_\xi$ the two-sided ideal of $U({\ggg}_{\bbk})$ generated by the even central elements $\{\bar x^p-\bar x^{[p]}-\xi(\bar x)^p\mid\bar x\in({\ggg}_{\bbk})_{\bar{0}}\}$. Then the quotient algebra $U_\xi({\ggg}_{\bbk}):=U({\ggg}_{\bbk})/J_\xi$ is called the reduced enveloping algebra with $p$-character $\xi$. We have $\dim U_\xi({\ggg}_{\bbk})=p^{\text{dim}\,({\ggg}_{\bbk})_{\bar{0}}}2^{\text{dim}\,({\ggg}_{\bbk})_{\bar{1}}}$ by construction. It follows from the Schur Lemma that any irreducible ${\ggg}_{\bbk}$-module $V$ is of $U_\xi({\ggg}_{\bbk})$ for a unique $\xi=\xi_V\in({\ggg}_{\bbk})_{\bar{0}}^*$. We often regard $\xi\in{\ggg}_{\bbk}^*$ by letting $\xi(({\ggg}_{\bbk})_{\bar{1}})=0$.
\subsubsection{}\label{2.3.2}
For $i\in{\bbz}$, set ${\ggg}_{\bbk}(i):={\ggg}_A(i)\otimes_A{\bbk}$ and put $\mathfrak{m}_{\bbk}:=\mathfrak{m}_A\otimes_A{\bbk}$, then \cite[Lemma 2.18]{ZS2} shows that $\mathfrak{m}_{\bbk}$ is a restricted subalgebra of $\mathfrak{g}_{\bbk}$. Denote by $\mathfrak{m}'_{\bbk}:=\mathfrak{m}'_A\otimes_A{\bbk}$ and $\mathfrak{p}_{\bbk}:=\mathfrak{p}_A\otimes_A{\bbk}$. Due to our assumptions on $A$, the elements $\bar{x}_1,\cdots,\bar{x}_l$ and $\bar{y}_1,\cdots,\bar{y}_q$ form a basis of the centralizer $({\ggg}^e_{\bbk})_{\bar{0}}$ and $({\ggg}^e_{\bbk})_{\bar{1}}$ of $e$ in ${\ggg}_{\bbk}$, respectively. It follows from \cite[\S4.1]{WZ} that the subalgebra $\mathfrak{m}_{\bbk}$ is $p$-nilpotent, and the linear function $\chi$ vanishes on the $p$-closure of $[\mathfrak{m}_{\bbk},\mathfrak{m}_{\bbk}]$. Set $Q_{\chi,{\bbk}}:=U({\ggg}_{\bbk})\otimes_{U(\mathfrak{m}_{\bbk})}{\bbk}_\chi$, where ${\bbk}_\chi=A_\chi\otimes_{A}{\bbk}={\bbk}1_\chi$. Clearly, ${\bbk}1_\chi$ is a one-dimensional  $\mathfrak{m}_{\bbk}$-module with the property $\bar x.1_\chi=\chi(\bar x)1_\chi$ for all $\bar x\in\mathfrak{m}_{\bbk}$, and it is obvious that $Q_{\chi,{\bbk}}\cong Q_{\chi,A}\otimes_A{\bbk}$ as ${\ggg}_{\bbk}$-modules. Define the finite $W$-superalgebra over ${\bbk}$ by $U({\ggg}_{\bbk},e):=(\text{End}_{{\ggg}_{\bbk}}Q_{\chi,{\bbk}})^{\text{op}}$.

Let ${\ggg}_A^*$ be the $A$-module dual to ${\ggg}_A$ and let $(\mathfrak{m}_A^\perp)_{\bar{0}}$ denote the set of all linear functions on $({\ggg}_A)_{\bar0}$ vanishing on $(\mathfrak{m}_A)_{\bar{0}}$. By the assumptions on $A$,  $(\mathfrak{m}_A^\perp)_{\bar{0}}$ is a free $A$-submodule and a direct summand of ${\ggg}^*_A$. Note that
$(\mathfrak{m}_A^\perp\otimes_A{\bbc})_{\bar{0}}$ and $(\mathfrak{m}_A^\perp\otimes_A{\bbk})_{\bar{0}}$ can be identified with the annihilators $\mathfrak{m}^\perp_{\bar{0}}
:=\{f\in{\ggg}^*_{\bar{0}}\mid f(\mathfrak{m}_{\bar{0}})=0\}$ and $(\mathfrak{m}_{\bbk}^\perp)_{\bar{0}}:=\{f\in({\ggg}_{\bbk})^*_{\bar{0}}\mid f((\mathfrak{m}_{\bbk})_{\bar{0}})=0\}$, respectively. Let $I_{\chi,A}$ denote the $A$-span of the left ideal of $U(\ggg_A)$ generated by all $x-\chi(x)$ with $x\in\mathfrak{m}_{A}$.

Given a linear function $\eta\in\chi+(\mathfrak{m}_{\bbk}^\perp)_{\bar{0}}$, set ${\ggg}_{\bbk}$-module $Q_{\chi}^\eta:=Q_{\chi,{\bbk}}/J_\eta Q_{\chi,{\bbk}}$. Each ${\ggg}_{\bbk}$-endomorphism $\bar\Theta_i$ of $Q_{\chi,\bbk}$ preserves $J_\eta Q_{\chi,{\bbk}}$, hence induces a ${\ggg}_{\bbk}$-endomorphism of $Q_{\chi}^\eta$ which is denoted by $\theta_i$.
 As in \cite{ZS2}, we set $U_\eta({\ggg}_{\bbk},e)=(\text{End}_{{\ggg}_{\bbk}}Q_{\chi}^\eta)^{\text{op}}$, and call it a {\sl reduced $W$-superalgebra}.

Since the restriction of $\eta$ to $\mathfrak{m}_{\bbk}$ coincides with that of $\chi$, the left ideal of $U({\ggg}_{\bbk})$ generated by all $x-\eta(x)$ with $x\in\mathfrak{m}_{\bbk}$ equals $I_{\chi,{\bbk}}:=I_{\chi,A}\otimes_A{\bbk}$ and ${\bbk}_\chi={\bbk}_\eta$ as $\mathfrak{m}_{\bbk}$-modules. We denote by $I_{\mathfrak{m}_{\bbk}}$ the left ideal of $U_\eta({\ggg}_{\bbk})$ generated by all $x-\eta(x)$ with $x\in\mathfrak{m}_{\bbk}$.

For a (left) $U_{\eta}({\ggg}_{\bbk})$-module $M$, define
$$M^{\mathfrak{m}_{\bbk}}:=\{v\in M\mid I_{\mathfrak{m}_{\bbk}}.v=0\}.$$  It follows from \cite[Proposition 2.21]{ZS2} that $U_{\eta}({\ggg}_{\bbk},e)$ can be identified with the $\bbk$-algebra $U_\eta({\ggg}_{\bbk})^{\text{ad}\,{\mmm}_{\bbk}}/ U_\eta({\ggg}_{\bbk})^{\text{ad}\,{\mmm}_{\bbk}}\cap I_{{\mmm}_{\bbk}}$. Therefore, any left $U_\eta({\ggg}_{\bbk})^{\text{ad}\,{\mmm}_{\bbk}}$-module can be considered as a $U_{\eta}({\ggg}_{\bbk},e)$-module with the trivial action of the ideal $U_\eta({\ggg}_{\bbk})^{\text{ad}\,{\mmm}_{\bbk}}\cap I_{{\mmm}_{\bbk}}$. Recall \cite[Theorem 2.24]{ZS2} shows that

\begin{theorem}\label{reducedfunctors} The correspondence $M$ to $M^{\mmm_\bbk}$ gives rise to a category equivalence between the category of $U_\eta({\ggg}_{\bbk})$-modules and the category of $U_\eta({\ggg}_{\bbk},e)$-modules:
$$U_\eta({\ggg}_{\bbk})\text{-mod}\longrightarrow U_{\eta}({\ggg}_{\bbk},e)\text{-mod}$$
with the inverse
$$U_{\eta}({\ggg}_{\bbk},e)\text{-mod}\longrightarrow U_\eta({\ggg}_{\bbk})\text{-mod},\qquad V\mapsto U_\eta({\ggg}_{\bbk})\otimes_{U_\eta({\ggg}_{\bbk})^
{\ad\,\mmm_{\bbk}}}V.$$
\end{theorem}

Furthermore, for any $U_{\eta}({\ggg}_{\bbk})$-module $M$, \cite[Lemma 2.22]{ZS2} shows that $M^{\mathfrak{m}_{\bbk}}$ is a free $U_{\eta}(\mathfrak{m}_{\bbk})$-module. By the same discussion as \cite[Proposition 4.2]{WZ}, one can conclude that there is an isomorphism of vector spaces 
$M\cong U_{\eta}(\mathfrak{m}_{\bbk})^*\otimes_{{\bbk}}M^{\mathfrak{m}_{\bbk}}$. It is immediate that
\begin{equation}\label{dimMn}
\text{dim}\,M^{\mathfrak{m}_{\bbk}}=\frac{\text{dim}\,M}
{\text{dim}\,U_{\eta}(\mathfrak{m}_{\bbk})}.
\end{equation}
In particular, we have $\text{dim}\,U_\eta({\ggg}_{\bbk})=\text{dim}\,U_{\eta}(\mathfrak{m}_{\bbk})\cdot\text{dim}\,U_\eta({\ggg}_{\bbk})^
{\ad\,\mmm_{\bbk}}$.

\section{Transition subalgebras of finite $W$-superalgebras in prime characteristic}\label{3}

Throughout the paper, we will maintain the notations and conventions as in \S1. Especially,  $\ggg$ is a given basic Lie superalgebra over $\bbc$, $A$ is an associated admissible ring, and $\bbk$ is an algebraically closed field of characteristic  $p\in \Pi(A)$. Throughout this section, fix a nilpotent element $e\in\ggg_\bz$, then we can define $d_i=\text{dim}\,{\ggg}_i-\text{dim}\,{\ggg}^e_i$ for $i\in{\bbz}_2$. Further recall that both $\text{dim}\,{\ggg}(-1)_{\bar1}$  and $d_1=\text{dim}\,{\ggg}_{\bar1}-\text{dim}\,{\ggg}^e_{\bar1}$ have the same parity by \S\ref{2.2.1}. We always set $q'=q$ if $d_1$ is even, and $q'=q+1$ if $d_1$ is odd.

In this section we will introduce a so-called transition subalgebra of  the finite $W$-superalgebra $U({\ggg}_{\bbk},e)$ over ${\bbk}$. Some structure relations between $U(\ggg_\bbk,e)$ and its transition subalgebra are presented, which enables us to connect the information of finite $W$-superalgebras over $\bbc$ with the modular representations of reduced enveloping algebras of basic Lie superalgebras over $\bbk$ in the following sections.

This section is somewhat a generalization of the Lie algebra case by Premet in \cite[\S 2]{P7}, with a few modifications. One can find that the emergence of odd parts in the Lie superalgebra ${\ggg}_{\bbk}$ makes the situation complicated.

\subsection{Transition subalgebras}Recall in \S\ref{2.2.2} we have defined \begin{equation*}
 \begin{array}{llllll}
 \mathbb{Z}_+^k&:=&\{(i_1,\cdots,i_k)\mid i_j\in\mathbb{Z}_+\},&
 \Lambda'_k&:=&\{(i_1,\cdots,i_k)\mid i_j\in\{0,1\}\}
 \end{array}
 \end{equation*}for $k\in\mathbb{Z}_+$ with $1\leqslant j\leqslant k$.
For ${\bf a}=(a_1,\cdots,a_l)\in\bbz_+^l$ and ${\bf b}=(b_1,\cdots,b_{q'})\in \Lambda'_{q'}$,
$U({\ggg}_A,e)$ denotes the $A$-span of the monomials
$$\{\Theta_1^{a_1}\cdots\Theta_l^{a_l}\cdot\Theta_{l+1}^{b_1}\cdots\Theta_{l+q'}^{b_{q'}}\mid ({\bf{a}},{\bf{b}})\in \bbz_+^l\times\Lambda'_{q'}\}$$ in \S\ref{2.2.2}.
 Our assumption on $A$ guarantees $U({\ggg}_A,e)$ to be an $A$-subalgebra of $U({\ggg},e)$ contained in
$(\text{End}_{{\ggg}_A}Q_{\chi,A})^{\text{op}}$. By the definition of $Q_{\chi,A}$ in \S\ref{2.2.2} and $I_{\chi,A}$ in \S\ref{2.3.2} we know that $Q_{\chi,A}$ can be identified with the ${\ggg}_A$-module $U({\ggg}_A)/I_{\chi,A}$. Hence $U({\ggg}_A,e)$ embeds into the $A$-algebra $(U({\ggg}_A)/I_{\chi,A})^{\text{ad}\,\mathfrak{m}_A}\cong (Q_{\chi,A})^{\text{ad}\,\mathfrak{m}_A}$. As $Q_{\chi,A}$ is a free $A$-module with basis $\{x^\mathbf{a}y^\mathbf{b}u^\mathbf{c}v^\mathbf{d}\otimes1_\chi\mid(\mathbf{a},\mathbf{b},\mathbf{c},\mathbf{d})\in{\bbz}^m_+
\times\Lambda'_n\times{\bbz}^s_+\times\Lambda'_t\}$, an easy induction on Kazhdan degree (based on \cite[Lemma 4.3]{ZS2} and the formulas displayed in \cite[Lemma 4.2, Theorem 4.5]{ZS2}) shows that
$$U({\ggg}_A,e)=(\text{End}_{{\ggg}_A}Q_{\chi,A})^{\text{op}}\cong Q_{\chi,A}^{\ad\,\mmm_A}.$$

\begin{defn}\label{reduced k} Set the $\bbk$-algebra
$T({\ggg}_{\bbk},e):=U({\ggg}_A,e)\otimes_A{\bbk}$. In the following we will call $T(\ggg_\bbk,e)$ a transition subalgebra.
\end{defn}
 It is notable that by the definition, $T({\ggg}_{\bbk},e)$  can be naturally identified with a subalgebra of the finite $W$-superalgebra $U({\ggg}_{\bbk},e)=(\text{End}_{{\ggg}_{\bbk}}Q_{\chi,{\bbk}})^{\text{op}}$ over ${\bbk}$. Moreover, $T({\ggg}_{\bbk},e)$ has a ${\bbk}$-basis consisting of all monomials $\bar{\Theta}_1^{a_1}\cdots\bar{\Theta}_{l+q'}^{b_{q'}}$, where $\bar{\Theta}_i:=\Theta_i\otimes1\in U({\ggg}_A,e)\otimes_A{\bbk}$ for $1\leqslant i\leqslant l+q'$.

Since all the coefficients of polynomials $F_{ij}$'s for $1\leqslant i,j\leqslant l+q'$ in \S\ref{2.2.2} are in $\mathbb{Q}$, one can assume the $F_{ij}$'s are over $A$ after enlarging $A$ if needed. Given a super-polynomial $g\in A[T_1,\cdots T_n]$, let $^p{\hskip-0.05cm}g$ denote the image of $g$ in the polynomial superalgebra ${\bbk}[T_1,\cdots T_n]=A[T_1,\cdots T_n]\otimes_A{\bbk}$. By the same discussion as \cite[Theorem 4.7]{ZS2}, we know that there exist super-polynomials $^p{\hskip-0.05cm}F_{ij}$'s of $l+q'$ indeterminants   over $\bbk$ ($i,j=1,\cdots,l+q'$) with the first $l$ indeterminants  being even, and the others being odd, such that
$$[\bar\Theta_i,\bar\Theta_j]=^p{\hskip-0.2cm} F_{ij}(\bar\Theta_1,\cdots,\bar\Theta_{l+q'}),~~i,j=1,\cdots,l+q',$$
while the $^p{\hskip-0.05cm}F_{ij}(\bar\Theta_1,\cdots,\bar\Theta_{l+q'})$'s satisfy the same relations as (\ref{fij}) and (\ref{fij2}).
\begin{theorem}\label{transition} Maintain the notations as above. Then the $\bar\Theta_i$'s and $^p{\hskip-0.05cm}F_{ij}$'s for $i,j=1,\cdots,l+q'$ constitute a data of generators and defining relations of $T(\ggg_{\bbk},e)$,
with $\bar{\Theta}_1,\cdots,\bar{\Theta}_{l}\in T({\ggg}_{\bbk},e)_{\bar0}$ and  $\bar{\Theta}_{l+1},\cdots,\bar{\Theta}_{l+q'}\in T({\ggg}_{\bbk},e)_{\bar1}$ as the generators of $T({\ggg}_{\bbk},e)$
subject to the relations
$$[\bar{\Theta}_i,\bar{\Theta}_j]=^p{\hskip-0.2cm}F_{ij}
(\bar{\Theta}_1,\cdots,\bar{\Theta}_{l+q'}),
$$
where $1\leqslant i,j\leqslant l+q'$.
\end{theorem}

\subsection{Revisit to reduced $W$-superalgebras with $p$-characters $\eta\in\chi+(\mathfrak{m}_{\bbk}^\bot)_{\bar{0}}$}\label{3.2}
Recall in \S\ref{2.3.2} we have defined the ${\ggg}_{\bbk}$-module $Q_{\chi}^\eta=Q_{\chi,{\bbk}}/J_\eta Q_{\chi,{\bbk}}$ and the reduced $W$-superalgebra $U_\eta({\ggg}_{\bbk},e)=(\text{End}_{{\ggg}_{\bbk}}Q_{\chi}^\eta)^{\text{op}}$, and \S\ref{2.2.1} shows that $\{x_1,\cdots,x_m,y_1,\cdots,$\\$y_n\}$ is an $A$-basis of ${\ppp}_A$. Set
\[X_i:=\left\{\begin{array}{ll}
x_{i+l}&\text{if}~1\leqslant  i\leqslant  m-l;\\
y_{l+q-m+i}&\text{if}~m-l+1\leqslant  i\leqslant  m+n-l-q;\\
u_{l+q-m-n+i}&\text{if}~m+n-l-q+1\leqslant  i\leqslant  m+n-l-q+s;\\
v_{l+q-m-n-s+i}&\text{if}~m+n-l-q+s+1\leqslant  i\leqslant  m+n-l-q+s+t',
\end{array}\right.
\]
where $t':=\lceil\frac{r}{2}\rceil=\lceil\frac{\text{dim}\,\ggg_\bbk(-1)_{\bar1}}{2}\rceil$.

\begin{conventions}\label{convention2.3} We will denote $\lceil\frac{r}{2}\rceil$ by $t'$ once and for all. It follows from \cite[Theorem 4.3]{WZ} that $\text{dim}\,U_\chi(\mathfrak{m}_{\bbk})=p^{\frac{d_0}{2}}2^{\lceil \frac{d_1}{2}\rceil}$ and we denote it by $\delta$ afterwards. By the assumption of the notations we have $\frac{d_0}{2}=m-l+s$ and $\lceil\frac{d_1}{2}\rceil=n-q+t'$.
\end{conventions}

For $(\mathbf{a},\mathbf{b},\mathbf{c},\mathbf{d})\in{\bbz}_+^{m-l}\times\Lambda'_{n-q}\times{\bbz}_+^s\times\Lambda'_{t'}$, define
\[\begin{array}{ccl}
X^{\mathbf{a},\mathbf{b},\mathbf{c},\mathbf{d}}:&=&X_1^{a_1}\cdots X_{m-l}^{a_{m-l}}X_{m-l+1}^{b_{1}}\cdots X_{m+n-l-q}^{b_{n-q}}X_{m+n-l-q+1}^{c_1}\cdots \\
&&\cdot X_{m+n-l-q+s}^{c_{s}}X_{m+n-l-q+s+1}^{d_{1}}\cdots X_{m+n-l-q+s+t'}^{d_{t'}}
\end{array}\]
and\[\begin{array}{ccl}
\bar{X}^{\mathbf{a},\mathbf{b},\mathbf{c},\mathbf{d}}:&=&\bar{X}_1^{a_1}\cdots \bar{X}_{m-l}^{a_{m-l}}\bar{X}_{m-l+1}^{b_{1}}\cdots\bar{X}_{m+n-l-q}^{b_{n-q}}\bar{X}_{m+n-l-q+1}^{c_1}\cdots\\ &&\cdot\bar{X}_{m+n-l-q+s}^{c_{s}}\bar{X}_{m+n-l-q+s+1}^{d_{1}}\cdots\bar{X}_{m+n-l-q+s+t'}^{d_{t'}},
\end{array}\] elements of $U({\ggg}_A)$ and $U({\ggg}_{\bbk})$, respectively. Denote by $\bar{1}_\chi$ the image of $1_\chi\in Q_{\chi,{\bbk}}$ in $Q^\eta_{\chi}$. For $k\in{\bbz}_+$, define$$\Lambda_k:=\{(i_1,\cdots,i_k)\mid i_j\in{\bbz}_+,~0\leqslant  i_j\leqslant  p-1\}$$ with $1\leqslant j\leqslant k$.
\begin{lemma}\label{right} The right modules $Q_{\chi,A}$ and $Q_\chi^\eta$ with
 $\eta\in\chi+(\mathfrak{m}_{\bbk}^\bot)_{\bar{0}}$ are free over $U({\ggg}_A,e)$ and $U_\eta({\ggg}_{\bbk},e)$ respectively. More precisely,
\begin{itemize}
\item[(1)] the set $\{X^{\mathbf{a},\mathbf{b},\mathbf{c},\mathbf{d}}\otimes1_\chi\mid(\mathbf{a},\mathbf{b},\mathbf{c},\mathbf{d})\in{\bbz}_+^{m-l}\times\Lambda'_{n-q}\times{\bbz}_+^s\times\Lambda'_{t'}\}$\, is a free basis of the $U({\ggg}_A,e)$-module $Q_{\chi,A}$;
\item[(2)]  the set $\{\bar{X}^{\mathbf{a},\mathbf{b},\mathbf{c},\mathbf{d}}\otimes\bar{1}_\chi\mid(\mathbf{a},\mathbf{b},\mathbf{c},\mathbf{d})\in\Lambda_{m-l}\times\Lambda'_{n-q}\times\Lambda_s\times\Lambda'_{t'}\}$\, is a free basis of the $U_\eta({\ggg}_{\bbk},e)$-module $Q_{\chi}^\eta$.
\end{itemize}
\end{lemma}

\begin{proof}
The proof is the same as the finite $W$-algebra case, thus will be omitted here (see \cite[Lemma 4.2]{P4} and \cite[Lemma 2.3]{P7}).
\end{proof}

\subsection{Transition for finite $W$-superalgebras}

Let $\rho_{\bbk}$ denote the representation of $U({\ggg}_{\bbk})$ in $\text{End}_{\bbk}Q_{\chi,{\bbk}}$. Given a subspace $V$ in ${\ggg}_{\bbk}$ we denote by $Z_p(V)$ the subalgebra of $p$-center $Z_p({\ggg}_{\bbk})$ generated by all $\bar x^p-\bar x^{[p]}$ with $\bar x\in V_{\bar{0}}$. Clearly, $Z_p(V)$ is isomorphic to an (ordinary) polynomial algebra in $\text{dim}\,V_{\bar{0}}$ variables. We will denote $Z_p({\ggg}_{\bbk})$ by $Z_p$ for short.

Let $\mathfrak{a}_{\bbk}$ be the ${\bbk}$-span of $\bar{X}_1,\cdots,\bar{X}_{m+n-l-q+s+t'}$ in ${\ggg}_{\bbk}$ and put $\widetilde{{\ppp}}_{\bbk}=\mathfrak{a}_{\bbk}\oplus{\ggg}_{\bbk}^e$ (resp. $\widetilde{{\ppp}}_{\bbk}=\mathfrak{a}_{\bbk}\oplus{\ggg}_{\bbk}^e\oplus{\bbk}v_{\frac{r+1}{2}}$) when $d_1$ is even (resp. odd), then $\ggg_\bbk=\mmm_\bbk\oplus \widetilde\ppp_\bbk$.

By our assumptions on $x_{l+1},\cdots,x_m, y_{q+1},\cdots,y_n$ and the inclusion ${\ggg}_{\bbk}^f\subseteq\bigoplus\limits_{i\leqslant 0}{\ggg}_{\bbk}(i)$, we have that
$$\mathfrak{a}_{\bbk}=\{\bar x\in\widetilde{{\ppp}}_{\bbk}\mid(\bar x,{\ggg}_{\bbk}^f)=0\}\qquad(\text{resp.}~ \mathfrak{a}_{\bbk}\oplus{\bbk}v_{\frac{r+1}{2}}=\{\bar x\in\widetilde{{\ppp}}_{\bbk}\mid(\bar x,{\ggg}_{\bbk}^f)=0\})$$ when $d_1$ is even (resp. odd).

\begin{theorem}\label{keyisotheorem}
For any even nilpotent element $e\in({\ggg}_{\bbk})_{\bar{0}}$, we have
\begin{itemize}
\item[(1)] $\rho_{\bbk}(Z_p)\cong Z_p(\widetilde{{\ppp}}_{\bbk})$ as ${\bbk}$-algebras.
\item[(2)] $U({\ggg}_{\bbk},e)$ is a free $\rho_{\bbk}(Z_p)$-module of rank $p^l2^{q'}$. Especially, $U({\ggg}_{\bbk},e)$ is generated by its subalgebras $T({\ggg}_{\bbk},e)$ and $\rho_{\bbk}(Z_p)$.
\item[(3)]  Furthermore, $U({\ggg}_{\bbk},e)\cong T({\ggg}_{\bbk},e)\otimes_{\bbk}Z_p(\mathfrak{a}_{\bbk})$ as ${\bbk}$-algebras.
\end{itemize}
\end{theorem}
This theorem is a generalization of the finite $W$-algebra case in \cite[Theorem 2.1]{P7}. Compared with finite $W$-algebras, the construction of finite $W$-superalgebras is much more complicated. In particular, some new phenomenon occurs when $d_1$ is odd. Now we will prove the theorem in detail.

\begin{proof}  (1) It follows from ${\ggg}_{\bbk}=\mathfrak{m}_{\bbk}\oplus\widetilde{{\ppp}}_{\bbk}$ that $Z_p({\ggg}_{\bbk})\cong Z_p(\mathfrak{m}_{\bbk})\otimes_{\bbk}Z_p(\widetilde{{\ppp}}_{\bbk})$ as ${\bbk}$-algebras. As $Z_p(\mathfrak{m}_{\bbk})\cap \text{Ker}\rho_{\bbk}$ is an ideal of codimension $1$ in $Z_p(\mathfrak{m}_{\bbk})$, one can conclude that $\rho_{\bbk}(Z_p)=\rho_{\bbk}(Z_p(\widetilde{{\ppp}}_{\bbk}))$. As the monomials $\bar{x}^\mathbf{a}\bar{y}^\mathbf{b}\bar{u}^\mathbf{c}\bar{v}^\mathbf{d}\otimes1_\chi$ with $(\mathbf{a},\mathbf{b},\mathbf{c},\mathbf{d})\in{\bbz}_+^m\times\Lambda'_n\times{\bbz}_+^s\times\Lambda'_t$ (recall that $t=\lfloor\frac{\text{dim}\,{\ggg}_{\bbk}(-1)_{\bar1}}{2}\rfloor$) form a basis of $Q_{\chi,{\bbk}}$, and $Z_p(\widetilde{{\ppp}}_{\bbk})$ is a polynomial algebra in $\{\bar{x}_i^p-\bar{x}_i^{[p]}\mid 1\leqslant  i\leqslant  m\}\cup\{\bar{u}_j^p-\bar{u}_j^{[p]}\mid 1\leqslant  j\leqslant  s\}$, we have $Z_p(\widetilde{{\ppp}}_{\bbk})\cap\text{Ker}\rho_{\bbk}=\{0\}$. It follows that $\rho_{\bbk}(Z_p)\cong Z_p(\widetilde{{\ppp}}_{\bbk})$ as ${\bbk}$-algebras. This completes the proof of statement (1).

(2) As the proof of the remaining statements are the same for both cases corresponding to $d_1$ being even and odd respectively, it is sufficient for us to make arguments under the assumption that $d_1$ is odd.

First recall that $S((\widetilde{{\ppp}}_{\bbk})_{\bar0})\cong{\bbk}[\chi+(\mathfrak{m}_{\bbk}^\bot)_{\bar{0}}]$ by the discussion of \cite[\S2.3]{ZS2}, hence $Z_p(\widetilde{{\ppp}}_{\bbk})\cong {\bbk}[(\chi+(\mathfrak{m}_{\bbk}^\bot)_{\bar{0}})^{(1)}]$, where $(\chi+(\mathfrak{m}_{\bbk}^\bot)_{\bar{0}})^{(1)}\subseteq({\ggg}_{\bbk}^*)_{\bar0}^{(1)}$ is the Frobenius twist of $\chi+(\mathfrak{m}_{\bbk}^\bot)_{\bar{0}}$. Then we will manage to construct a set of free basis of $U(\ggg_\bbk,e)$ as a $\rho_{\bbk}(Z_p)$-module, via the reduced $W$-superalgebra $U_\eta(\ggg_\bbk,e)$ with $\eta\in \chi+(\mathfrak{m}_{\bbk}^\bot)_{\bar{0}}$. Now we proceed by steps.

(2-i) Let us begin with understanding the $Z_p(\widetilde{{\ppp}}_{\bbk})$-module $Q_{\chi,{\bbk}}$.

As an immediate consequence of \cite[Theorem 4.5(1)]{ZS2},  we have
\begin{equation}\label{dede}
\bar{\Theta}_k^p(1_\chi)-(\bar{x}_k^p+\sum\limits_{|(\mathbf{a},\mathbf{0},\mathbf{c},\mathbf{0})|_e=m_k+2}\mu^k_{\mathbf{a},\mathbf{0},\mathbf{c},\mathbf{0}}\bar{x}^{p\mathbf{a}}
\bar{u}^{p\mathbf{c}})\otimes1_\chi\in(Q_{\chi,{\bbk}})_{p(m_k+2)-1}
\end{equation}for $1\leqslant  k\leqslant  l$, where $\mu^k_{\mathbf{a},\mathbf{0},\mathbf{c},\mathbf{0}}\in\mathbb{F}_p$. Since $\bar x^{[p]}\in{\ggg}_{\bbk}(pi)$ whenever $\bar x\in{\ggg}_{\bbk}(i)$ for all $i\in{\bbz}$ (see the proof of \cite[Lemma 2.18]{ZS2}),
discussing in the graded algebra $\text{gr}(U({\ggg}_{\bbk}))$ under the Kazhdan filtration we can obtain that
\begin{equation}\label{grad}
\text{gr}(\bar{x}_i^p-\bar{x}_i^{[p]})=\text{gr}(\bar{x}_i)^p,~~~~~~~~\text{gr}(\bar{u}_j^p-\bar{u}_j^{[p]})=\text{gr}(\bar{u}_j)^p~~~~(1\leqslant  i\leqslant m;~~ 1\leqslant  j\leqslant  s).
\end{equation}

On the other hand, Lemma~\ref{right}(1) implies that the vectors $\bar{X}^{(\mathbf{a},\mathbf{b},\mathbf{c},\mathbf{d})}\otimes1_\chi$ with\[\begin{array}{ccl}
(\mathbf{a},\mathbf{b},\mathbf{c},\mathbf{d})
&=&(a_1,\cdots,a_{m-l};b_1,\cdots,b_{n-q};c_1,\cdots,c_{s};d_1,\cdots,d_{t'})\\
&\in &{\bbz}_+^{m-l}\times\Lambda'_{n-q}\times{\bbz}_+^{s}\times\Lambda'_{t'}
\end{array}\]
form a free basis of the right $T({\ggg}_{\bbk},e)$-module $Q_{\chi,{\bbk}}$. As $Q_{\chi,{\bbk}}$ is a Kazhdan-filtrated $T({\ggg}_{\bbk},e)$-module, straightforward induction on filtration degree based on \eqref{dede} and \eqref{grad} shows that $Q_{\chi,{\bbk}}$ is generated as a $Z_p(\widetilde{{\ppp}}_{\bbk})$-module by the set
$$\{\bar{X}^{(\mathbf{a},\mathbf{b},\mathbf{c},\mathbf{d})}\bar{\Theta}^{(\mathbf{i},\mathbf{j})}\otimes1_\chi\mid (\mathbf{a},\mathbf{b},\mathbf{c},\mathbf{d},\mathbf{i},\mathbf{j})\in\Lambda_{m-l}\times\Lambda'_{n-q}\times\Lambda_{s}\times\Lambda'_{t'}\times\Lambda_{l}\times\Lambda'_{q+1}\}.$$

(2-ii) Let $h$ be an arbitrary element of $U({\ggg}_{\bbk},e)$. Then by the above discussion we can assume that
\[\begin{array}{cccl}
h(1_\chi)&=&\sum f_{\mathbf{a},\mathbf{b},\mathbf{c},\mathbf{d},\mathbf{i},\mathbf{j}}&
\bar{X}_1^{a_1}\cdots\bar{X}_{m-l}^{a_{m-l}}\bar{X}_{m-l+1}^{b_1}\cdots\bar{X}_{m+n-l-q}^{b_{n-q}}
\bar{X}_{m+n-l-q+1}^{c_1}\cdots\\
&&&\bar{X}_{m+n-l-q+s}^{c_{s}}\bar{X}_{m+n-l-q+s+1}^{d_1}\cdots
\bar{X}_{m+n-l-q+s+t'}^{d_{t'}}\cdot\bar{\Theta}_1^{i_1}\cdots\bar{\Theta}_l^{i_l}\\
&&&\bar{\Theta}_{l+1}^{j_1}\cdots\bar{\Theta}_{l+q}^{j_q}
\bar{\Theta}_{l+q+1}^{j_{q+1}}(1_\chi),
\end{array}\]
where
$f_{\mathbf{a},\mathbf{b},\mathbf{c},\mathbf{d},\mathbf{i},\mathbf{j}}\in Z_p(\widetilde{{\ppp}}_{\bbk})$ with $(\mathbf{a},\mathbf{b},\mathbf{c},\mathbf{d},\mathbf{i},\mathbf{j})$ in the set $\Lambda_{m-l}\times\Lambda'_{n-q}\times\Lambda_{s}\times\Lambda'_{t'}\times\Lambda_{l}\times\Lambda'_{q+1}$. For any $\eta\in\chi+(\mathfrak{m}_{\bbk}^\bot)_{\bar{0}}$, the image of $f_{\mathbf{a},\mathbf{b},\mathbf{c},\mathbf{d},\mathbf{i},\mathbf{j}}$ in $U_\eta({\ggg}_{\bbk})$ is a scalar in ${\bbk}$, which shall be denoted by $\eta(\mathbf{a},\mathbf{b},\mathbf{c},\mathbf{d},\mathbf{i},\mathbf{j})$.

Suppose $f_{\mathbf{a},\mathbf{b},\mathbf{c},\mathbf{d},\mathbf{i},\mathbf{j}}\neq0$ for a nonzero $(\mathbf{a},\mathbf{b},\mathbf{c},\mathbf{d})\in\Lambda_{m-l}\times\Lambda'_{n-q}\times\Lambda_{s}\times\Lambda'_{t'}$ and some $(\mathbf{i},\mathbf{j})
\in\Lambda_{l}\times\Lambda'_{q+1}$. Then there exists $\eta\in\chi+(\mathfrak{m}_{\bbk}^\bot)_{\bar{0}}$ such that $\eta(\mathbf{a},\mathbf{b},\mathbf{c},\mathbf{d},\mathbf{i},\mathbf{j})\neq0$. Let $h_\eta$ be the image of $h\in U({\ggg}_{\bbk},e)$ in $U_\eta({\ggg}_{\bbk},e)=(\text{End}_{{\ggg}_{\bbk}}Q_\chi^\eta)^{\text{op}}$. \cite[Remark 3.9]{ZS2} shows that there exist even elements $\theta_1,\cdots,\theta_l\in U_\eta({\ggg}_{\bbk},e)_{\bar0}$ and odd elements $ \theta_{l+1},\cdots,\theta_{l+q+1}\in U_\eta({\ggg}_{\bbk},e)_{\bar1}$ in the same sense as in  \cite[Corollary 3.6]{ZS2}, such that
 the monomials $$\theta_1^{a_1}\cdots\theta_{l}^{a_l}\theta_{l+1}^{b_1}\cdots\theta_{l+q+1}^{b_{q+1}}$$ with $0\leqslant  a_k\leqslant  p-1$ for $1\leqslant k\leqslant l$ and $0\leqslant b_k\leqslant 1$ for $1\leqslant k\leqslant q+1$ form a ${\bbk}$-basis of $U_\eta({\ggg}_{\bbk},e)$. Therefore, $h_\eta(\bar{1}_\chi)$ is a ${\bbk}$-linear combination of $\theta_1^{i_1}\cdots\theta_l^{i_l}\theta_{l+1}^{j_1}\cdots\theta_{l+q}^{j_{q}}\theta_{l+q+1}^{j_{q+1}}(\bar{1}_\chi)$ with$$(i_1,\cdots,i_l;j_1,\cdots,j_q;j_{q+1})\in\Lambda_l\times\Lambda'_q\times\Lambda'_1.$$ By Lemma~\ref{right}(2), the set $$\{\bar{X}^{(\mathbf{a},\mathbf{b},\mathbf{c},\mathbf{d})}\otimes\bar{1}_\chi\mid(\mathbf{a},\mathbf{b},\mathbf{c},\mathbf{d})\in\Lambda_{m-l}\times\Lambda'_{n-q}\times\Lambda_s\times\Lambda'_{t'}\}$$ is a free basis of the right $U_\eta({\ggg}_{\bbk},e)$-module $Q_{\chi}^\eta$. Since $\eta(\mathbf{a},\mathbf{b},\mathbf{c},\mathbf{d},\mathbf{i},\mathbf{j})\neq0$ and $\theta_1^{i_1}\cdots\theta_l^{i_l}\theta_{l+1}^{j_1}\cdots\theta_{l+q}^{j_q}\theta_{l+q+1}^{j_{q+1}}$ is the image of $\bar{\Theta}_1^{i_1}\cdots\bar{\Theta}_l^{i_l}\bar{\Theta}_{l+1}^{j_1}\cdots\bar{\Theta}_{l+q}^{j_q}\bar{\Theta}_{l+q+1}^{j_{q+1}}$ in $U_\eta({\ggg}_{\bbk},e)$, it is now evident that $h_\eta(\bar{1}_\chi)$ cannot be a ${\bbk}$-linear combination of $\theta_1^{i_1}\cdots\theta_l^{i_l}\theta_{l+1}^{j_1}\cdots$\\$\theta_{l+q}^{j_{q}}\theta_{l+q+1}^{j_{q+1}}(\bar{1}_\chi)$ with $$(i_1,\cdots,i_l;j_1,\cdots,j_q;j_{q+1})\in\Lambda_l\times\Lambda'_q\times\Lambda'_1.$$ This contradiction shows that $f_{\mathbf{a},\mathbf{b},\mathbf{c},\mathbf{d},\mathbf{i},\mathbf{j}}=0$ unless $(\mathbf{a},\mathbf{b},\mathbf{c},\mathbf{d})=\mathbf{0}$. As a consequence, \begin{align}
 \{\bar{\Theta}_1^{i_1}\cdots\bar{\Theta}_l^{i_l}\bar{\Theta}_{l+1}^{j_1}\cdots\bar{\Theta}_{l+q}^{j_q}\bar{\Theta}_{l+q+1}^{j_{q+1}}\mid
(\mathbf{i},\mathbf{j})\in\Lambda_{l}\times\Lambda'_{q+1}\}
\end{align}
generates $U({\ggg}_{\bbk},e)$ as a $Z_p(\widetilde{{\ppp}}_{\bbk})$-module. Specialising at a suitable $\eta\in\chi+(\mathfrak{m}_{\bbk}^\bot)_{\bar{0}}$ and applying
 \cite[Remark 3.9]{ZS2} we further deduce that the set
 \begin{align}\label{basis}
 \{\bar{\Theta}_1^{i_1}\cdots\bar{\Theta}_l^{i_l}\bar{\Theta}_{l+1}^{j_1}\cdots\bar{\Theta}_{l+q}^{j_q}\bar{\Theta}_{l+q+1}^{j_{q+1}}\mid
(\mathbf{i},\mathbf{j})\in\Lambda_{l}\times\Lambda'_{q+1}\}
\end{align}
 is a free basis of the $Z_p(\widetilde{{\ppp}}_{\bbk})$-module $U({\ggg}_{\bbk},e)$. Then $U({\ggg}_{\bbk},e)$ is a free $\rho_{\bbk}(Z_p)$-module of rank $p^l2^{q+1}$. Note that the elements of (\ref{basis}) are in the $\bbk$-algebra $T(\ggg_\bbk,e)$, then the second part of statement (2) follows.  We complete the proof of statement (2).

 (3)  We first claim that
 \begin{align}\label{statement3}
  u\in T({\ggg}_{\bbk},e)\cdot Z_p(\mathfrak{a}_{\bbk})~\text{for each}~u\in U({\ggg}_{\bbk},e).
  \end{align}
   We proceed the proof of Claim (\ref{statement3}) by steps, starting with some necessary preparations.

 (3-i) Note that every ${\ggg}_{\bbk}$-endomorphism of $Q_{\chi,{\bbk}}$ is uniquely determined by its value at $1_\chi$. For a nonzero $u\in U({\ggg}_{\bbk},e)$ with highest Kazhdan degree $n(u)$, we can write $$u(1_\chi)=\sum\limits_{|(\mathbf{a},\mathbf{b},\mathbf{c},\mathbf{d})|_e\leqslant  n(u)}\lambda_{\mathbf{a},\mathbf{b},\mathbf{c},\mathbf{d}}\bar{x}^\mathbf{a}\bar{y}^\mathbf{b}\bar{u}^\mathbf{c}\bar{v}^\mathbf{d}\otimes1_\chi,$$ where $\lambda_{\mathbf{a},\mathbf{b},\mathbf{c},\mathbf{d}}\neq0$ for at least one $(\mathbf{a},\mathbf{b},\mathbf{c},\mathbf{d})$ with $|(\mathbf{a},\mathbf{b},\mathbf{c},\mathbf{d})|_e=n(u)$. For $k\in{\bbz}_+$ put $$\Lambda^k(u):=\{(\mathbf{a},\mathbf{b},\mathbf{c},\mathbf{d})\in{\bbz}_+^m\times\Lambda'_n\times{\bbz}_+^s\times\Lambda'_t\mid
\lambda_{\mathbf{a},\mathbf{b},\mathbf{c},\mathbf{d}}\neq0~\&~|(\mathbf{a},\mathbf{b},\mathbf{c},\mathbf{d})|_e=k\},$$
and denote by $\Lambda^{\text{max}}(u)$ the set of all $
(\mathbf{a},\mathbf{b},\mathbf{c},\mathbf{d})\in\Lambda^{n(u)}(u)$ for which the quantity $n(u)-|\mathbf{a}|-|\mathbf{b}|-|\mathbf{c}|-|\mathbf{d}|$ assumes its maximum value. This maximum value will be denoted by $N(u)$. For each $(\mathbf{a},\mathbf{b},\mathbf{c},\mathbf{d})\in\Lambda^{\text{max}}$, let $\bar x_i\in{\ggg}_{\bbk}(k_i)_{\bar0},~\bar y_j\in{\ggg}_{\bbk}(k'_j)_{\bar1}$ for $1\leqslant i\leqslant m$ and $1\leqslant j\leqslant n$ with $k_i, k_j'\in{\bbz}_+$, then we have
\[
\begin{array}{cl}
&|(\mathbf{a},\mathbf{b},\mathbf{c},\mathbf{d})|_e-|\mathbf{a}|-|\mathbf{b}|-|\mathbf{c}|-|\mathbf{d}|\\
=&\sum\limits_{i=1}^{m}(k_i+2)a_i
+\sum\limits_{i=1}^{n}(k'_i+2)b_i+\sum\limits_{i=1}^{s}c_i+\sum\limits_{i=1}^{t}d_i
-|\mathbf{a}|-|\mathbf{b}|-|\mathbf{c}|-|\mathbf{d}|\geqslant 0.
\end{array}\]
Consequently, $n(u),~N(u)\in{\bbz}_+$ and $n(u)\geqslant  N(u)$.


(3-ii) The previous arguments in step (2) along with \cite[Theorem 4.5(1)]{ZS2} show that
\[
\begin{array}{ll}
\Lambda^{\text{max}}(\bar{\Theta}_i)=\{(\mathbf{e}_i,\mathbf{0},\mathbf{0},\mathbf{0})\}&\text{for}~1\leqslant  i\leqslant l;\\
\Lambda^{\text{max}}(\rho_{\bbk}(\bar{x}_i^p-\bar{x}_i^{[p]}))=\{(p\mathbf{e}_i,\mathbf{0},\mathbf{0},\mathbf{0})\}&\text{for}~1\leqslant  i\leqslant m;\\
\Lambda^{\text{max}}(\rho_{\bbk}(\bar{u}_j^p-\bar{u}_j^{[p]}))=\{(\mathbf{0},\mathbf{0},p\mathbf{e}_j,\mathbf{0})\}& \text{for}~1\leqslant  j\leqslant  s;\\
\Lambda^{\text{max}}(\bar{\Theta}_k)=\{(\mathbf{0},\mathbf{e}_{k-l},\mathbf{0},\mathbf{0})\}& \text{for}~l+1\leqslant  k\leqslant  l+q;\\
\Lambda^{\text{max}}(\bar{\Theta}_{l+q+1})=\{(\mathbf{0},\mathbf{0},\mathbf{0},\mathbf{e}_{t})\}.&
\end{array}\]
Here $\mathbf{e}_i=(\delta_{i1},\delta_{i2},\cdots,\delta_{ij})$ is a tuple with $j$ entries for $j\in{\bbz}_+$, and $\delta_{ik}$ is the Kronecker function for $k=1,\cdots,j$.

Since $Q_{\chi,{\bbk}}$ is a Kazhdan filtrated $U({\ggg}_{\bbk})$-module, this implies that
\[
\begin{array}{ll}
&\Lambda^{\text{max}}(\prod\limits_{i=1}^m\rho_{\bbk}(\bar{x}_i^p-\bar{x}_i^{[p]})^{a_i}\cdot
\prod\limits_{i=1}^s\rho_{\bbk}(\bar{u}_i^p-\bar{u}_i^{[p]})^{b_i}\cdot\bar{\Theta}_1^{c_1}\cdots\bar{\Theta}_l^{c_l}\bar{\Theta}_{l+1}^{d_1}\cdots\bar{\Theta}_{l+q}^{d_q}\bar{\Theta}_{l+q+1}^{d_{q+1}})\\
=&\{(\sum\limits_{i=1}^{m}pa_i\mathbf{e}_i+\sum\limits_{j=1}^{l}c_j\mathbf{e}_j,\sum\limits_{i=1}^{q}d_i\mathbf{e}_i,\sum\limits_{i=1}^{s}pb_i\mathbf{e}_i,d_{q+1}\mathbf{e}_t)\}
\end{array}\]
for all $(a_1,\cdots,a_m;b_1,\cdots,b_s;c_1,\cdots,c_l;d_1,\cdots,d_{q+1})\in{\bbz}_+^m\times
{\bbz}_+^s\times  \Lambda_l\times\Lambda'_{q+1}$. 

Thanks to statement (2),  $U({\ggg}_{\bbk},e)$ is generated as a $Z_p(\widetilde{{\ppp}}_{\bbk})$-module by the set $$\{\bar{\Theta}_1^{i_1}\cdots\bar{\Theta}_l^{i_l}\bar{\Theta}_{l+1}^{j_1}\cdots\bar{\Theta}_{l+q}^{j_q}\bar{\Theta}_{l+q+1}^{j_{q+1}}\mid
(\mathbf{i},\mathbf{j})\in\Lambda_{l}\times\Lambda'_{q+1}\}.$$ It follows that for every $u\in U({\ggg}_{\bbk},e)$ with $(n(u),N(u))=(d,d')$ there exists a ${\bbk}$-linear combination $u'$ of the endomorphisms
\begin{equation*}
u(\mathbf{a},\mathbf{b},\mathbf{c},\mathbf{d}):=\prod\limits_{i=1}^m\rho_{\bbk}(\bar{x}_i^p-\bar{x}_i^{[p]})^{a_i}\cdot
\prod\limits_{i=1}^s\rho_{\bbk}(\bar{u}_i^p-\bar{u}_i^{[p]})^{b_i}\cdot\bar{\Theta}_1^{c_1}\cdots\bar{\Theta}_l^{c_l}\bar{\Theta}_{l+1}^{d_1}\cdots\bar{\Theta}_{l+q}^{d_q}\bar{\Theta}_{l+q+1}^{d_{q+1}}
\end{equation*}
for all $(a_1,\cdots,a_m;b_1,\cdots,b_s;c_1,\cdots,c_l;d_1,\cdots,d_{q+1})\in{\bbz}_+^m\times
{\bbz}_+^s\times\Lambda_l\times\Lambda'_{q+1}$ with $\Lambda^{\text{max}}(u(\mathbf{a},\mathbf{b},\mathbf{c},
\mathbf{d}))\subseteq\Lambda^{\text{max}}(u)$ such that either $n(u-u')<d$, or $n(u-u')=d$ and $N(u-u')<d'$.

(3-iii) Let $\Omega:=\{(n_1,n_2)\in \bbz_+^2\mid n_1\geq n_2\}$ be a totally-ordered set with tuples ordered lexicographically.  Due to the arguments in (3-i),
 $(n(u),N(u))\in \Omega$ for all $u\in  U(\ggg_\bbk,e)$.   Now we prove Claim (\ref{statement3}) by induction on $(n(u), N(u))$ in the totally-ordered set $\Omega$.
 The claim is clearly valid when $(n(u),N(u))=(0,0)$.
  Assume that $u\in T({\ggg}_{\bbk},e)\cdot Z_p(\mathfrak{a}_{\bbk})$ for all nonzero $u\in
 U({\ggg}_{\bbk},e)$ with $(n(u),N(u))\prec(d,d')$.  Now let $u\in U({\ggg}_{\bbk},e)$ be such that $(n(u),N(u))=(d,d')$. By the preceding remark we know that there exists $u'=\sum\limits_{(\mathbf{a},\mathbf{b},\mathbf{c},\mathbf{d})}
  \lambda_{\mathbf{a},\mathbf{b},\mathbf{c},\mathbf{d}}
  u(\mathbf{a},\mathbf{b},\mathbf{c},\mathbf{d})$
  with $\Lambda^{\text{max}}(u(\mathbf{a},\mathbf{b},\mathbf{c},\mathbf{d}))\subseteq\Lambda^{\text{max}}(u)$ for all $(\mathbf{a},\mathbf{b},\mathbf{c},\mathbf{d})$ with $\lambda_{\mathbf{a},\mathbf{b},\mathbf{c},\mathbf{d}}\neq0$ such that $(n(u-u'),N(u-u'))\prec(d,d')$. Set
\[
\begin{array}{lll}
v(\mathbf{a},\mathbf{b},\mathbf{c},\mathbf{d})&:=&u((0,\cdots,0,a_{l+1},\cdots,a_m),\mathbf{b},\mathbf{0},\mathbf{0})\cdot\\
&&\prod\limits_{i=1}^{l}\bar{\Theta}_i^{pa_i}\cdot
(\bar{\Theta}_1^{c_1}\cdots\bar{\Theta}_l^{c_l}\bar{\Theta}_{l+1}^{d_1}\cdots\bar{\Theta}_{l+q}^{d_q}\bar{\Theta}_{l+q+1}^{d_{q+1}}).
\end{array}\]
Using \eqref{dede} it is easy to observe that $\Lambda^{\text{max}}(u(\mathbf{a},\mathbf{b},\mathbf{c},\mathbf{d}))=\Lambda^{\text{max}}(v(\mathbf{a},\mathbf{b},\mathbf{c},\mathbf{d}))$ and
\[
\begin{array}{lll}
&(n(u(\mathbf{a},\mathbf{b},\mathbf{c},\mathbf{d})-v(\mathbf{a},\mathbf{b},\mathbf{c},\mathbf{d})),N(u(\mathbf{a},\mathbf{b},\mathbf{c},\mathbf{d})-v(\mathbf{a},\mathbf{b},\mathbf{c},\mathbf{d})))\\
\prec&(n(u(\mathbf{a},\mathbf{b},\mathbf{c},\mathbf{d})),N(u(\mathbf{a},\mathbf{b},\mathbf{c},\mathbf{d}))).
\end{array}\]
We now put $u^{''}:=\sum\limits_{(\mathbf{a},\mathbf{b},\mathbf{c},\mathbf{d})}\lambda_{\mathbf{a},\mathbf{b},\mathbf{c},
\mathbf{d}}v(\mathbf{a},\mathbf{b},\mathbf{c},\mathbf{d})$, an element of $T({\ggg}_{\bbk},e)\cdot Z_p(\mathfrak{a}_{\bbk})$. Because $(n(u-u^{''}),N(u-u^{''}))\prec(n(u),N(u))$, the equality $U({\ggg}_{\bbk},e)=T({\ggg}_{\bbk},e)\cdot Z_p(\mathfrak{a}_{\bbk})$ follows by induction on the length of $(d,d')$ in the linearly ordered set $(\Omega,\prec)$. We complete the proof of Claim (\ref{statement3}).

Next we will finish the proof of statement (3). By Lemma~\ref{right}(1) and the procedure of ``modular $p$ reduction'' we know that the vectors $\bar{X}^{(\mathbf{a},\mathbf{b},\mathbf{c},\mathbf{d})}\otimes1_\chi$ with
\[\begin{array}{lll}
(\mathbf{a},\mathbf{b},\mathbf{c},\mathbf{d})&=&(a_1,\cdots,a_{m-l};b_1,\cdots,b_{n-q};c_1,\cdots,c_{s};d_1,\cdots,d_{t'})\\
&\in &{\bbz}_+^{m-l}\times\Lambda'_{n-q}\times{\bbz}_+^{s}\times\Lambda'_{t'}
\end{array}\]
form a free basis of the right $T({\ggg}_{\bbk},e)$-module $Q_{\chi,{\bbk}}$. Since \eqref{grad} shows that $\bar X_i^{p}$ and $\bar X_i^{p}-\bar X_i^{[p]}$ have the same Kazhdan degree in $U({\ggg}_{\bbk})$ for $1\leqslant i\leqslant m-l$ and $m+n-l-q+1\leqslant i\leqslant m+n-l-q+s$ respectively, and $Q_{\chi,{\bbk}}$ is a Kazhdan filtered $U({\ggg}_{\bbk})$-module, it follows that the vectors
$$\prod\limits_{i=1}^{m-l}\prod\limits_{j=m+n-l-q+1}^{m+n-l-q+s}\rho_{\bbk}(\bar X_i^{p}-\bar X_i^{[p]})^{a_i}
\rho_{\bbk}(\bar X_j^p-\bar X_j^{[p]})^{b_j}\cdot\bar{\Theta}_1^{c_1}\cdots\bar{\Theta}_l^{c_l}
\bar{\Theta}_{l+1}^{d_1}\cdots\bar{\Theta}_{l+q}^{d_q}\bar{\Theta}_{l+q+1}^{d_{q+1}}$$
are linearly independent, where $(\mathbf{a},\mathbf{b},\mathbf{c},\mathbf{d})\in{\bbz}_+^{m-l}\times{\bbz}_+^s\times{\bbz}_+^l\times\Lambda'_{q+1}$.

Combining all above, we have an isomorphism between ${\bbk}$-algebras$$U({\ggg}_{\bbk},e)\cong T({\ggg}_{\bbk},e)\otimes_{\bbk}Z_p(\mathfrak{a}_{\bbk}).$$

By the analysis at the beginning of step (2), we complete the proof.
\end{proof}

With the above theorem, we can define a ``reduced" quotient algebra of $U(\ggg_\bbk,e)$,  analogous to the reduced enveloping algebra of a restricted Lie (super)algebra, by
$$ \widetilde{U}_\eta({\ggg}_{\bbk},e):=
U({\ggg}_{\bbk},e)\otimes_{Z_p(\widetilde{{\ppp}}_{\bbk})}{\bbk}_\eta.
$$
\begin{lemma}\label{isoofreducedW-alg} For any given $\eta\in \chi+(\mathfrak{m}_{\bbk}^\bot)_{\bar{0}}$, the above ``reduced" quotient algebras are isomorphic to the reduced $W$-superalgebras, i.e.
 $\widetilde{U}_\eta({\ggg}_{\bbk},e)\cong U_\eta(\ggg_\bbk,e)$.
\end{lemma}
\begin{proof} The canonical projection $Q_{\chi,{\bbk}}\twoheadrightarrow Q_{\chi,{\bbk}}/J_\eta Q_{\chi,{\bbk}}=Q_\chi^\eta$ gives rise to an algebra homomorphism $\rho_\eta:\widetilde{U}_\eta({\ggg}_{\bbk},e)\rightarrow
(\text{End}_{{\ggg}_{\bbk}}Q_\chi^\eta)^{\text{op}}=U_\eta({\ggg}_{\bbk},e)$. As $\text{dim}\,\widetilde{U}_\eta({\ggg}_{\bbk},e)\leqslant  p^l2^{q'}$ by Theorem~\ref{keyisotheorem}(2), \cite[Remark 3.9]{ZS2} yields that $\rho_\eta$ is an algebra isomorphism. We complete the proof.
\end{proof}

\subsection{Transition for minimal dimensional representations}
At the beginning of this subsection, first notice the following important fact:

\begin{prop}\label{no1}
Assume that  $d_1$ is odd. Then the finite $W$-superalgebra $U({\ggg},e)$ over ${\bbc}$ can not afford a one-dimensional  representation.
\end{prop}

\begin{proof}
Recall that when $d_1$ is odd, \S\ref{2.2.2} shows that there is an element $\Theta_{l+q+1}=v_{\frac{r+1}{2}}\otimes1_\chi\in U({\ggg},e)_{\bar1}$, then $$\Theta_{l+q+1}^2(1_\chi)=v_{\frac{r+1}{2}}^2\otimes1_\chi=\frac{1}{2}[v_{\frac{r+1}{2}},v_{\frac{r+1}{2}}]\otimes1_\chi=\frac{1}{2}\chi([v_{\frac{r+1}{2}},v_{\frac{r+1}{2}}])\otimes1_\chi=\frac{1}{2}\otimes1_\chi.$$
Thus $\Theta_{l+q+1}^2=\frac{1}{2}\text{id}$.

For any $U({\ggg},e)$-module $M$, let $0\neq v\in M$ be a ${\bbz}_2$-homogeneous element. We claim that $\Theta_{l+q+1}.v \neq 0$. If not, i.e. $\Theta_{l+q+1}.v=0$, then $\Theta_{l+q+1}^2.v=0$. However, by the preceding remark we have $\Theta_{l+q+1}^2.v=\frac{1}{2}v$,  a contradiction. Therefore, $\Theta_{l+q+1}.v$ is a nonzero element in $M$, which obviously shares the different parity from that of $v$. Thus the dimension of any $U({\ggg},e)$-module (as a vector space) is at least two, and the algebra $U({\ggg},e)$ can not afford a one-dimensional   representation in this case.
\end{proof}

\begin{rem}
Recall that in \S\ref{2.2.3} we have obtained an algebra isomorphism
\[\begin{array}{lcll}
\phi:&(\text{End}_{\ggg}Q_{\chi})^{\text{op}}&\cong&Q_{\chi}^{\text{ad}\,\mathfrak{m}}\\ &\Theta&\mapsto&\Theta(1_\chi).
\end{array}
\]
 In the paper we will often identify $U({\ggg},e)=(\text{End}_{\ggg}Q_{\chi})^{\text{op}}$ with $Q_{\chi}^{\text{ad}\,\mathfrak{m}}$ as ${\bbc}$-algebras, and which will cause no confusion. For any ${\bbz}_2$-homogeneous elements $\Theta_1,\Theta_2\in Q_{\chi}^{\text{ad}\,\mathfrak{m}}$, since $\phi(\Theta_1\cdot\Theta_2)=\phi(\Theta_1)\phi(\Theta_2)$,
we can identify $\Theta_1\cdot\Theta_2$ in $U({\ggg},e)$ with $\Theta_1(1_\chi)\cdot\Theta_2(1_\chi)$ in $Q_{\chi}^{\text{ad}\,\mathfrak{m}}$. When $d_1$ is odd, the element $\Theta_{l+q+1}$ in $(\text{End}_{\ggg}Q_{\chi})^{\text{op}}$ can be considered as the element $v_{\frac{r+1}{2}}\otimes1_\chi$ in $Q_{\chi}^{\text{ad}\,\mathfrak{m}}$, and for any ${\bbz}_2$-homogeneous element $v$ in a $U({\ggg},e)$-module $M$, we have $\Theta_{l+q+1}^2.v=\frac{1}{2}v$ by Proposition \ref{no1}.
\end{rem}

Now we are in the position to talk about the transiting role of the transition subalgebras for the minimal dimensions of modular representations of basic Lie superalgebras.

\begin{prop}\label{transitionforminimaldim} Keep the above notations. If $p\gg0$ for the field $\bbk=\overline{\mathbb{F}}_p$, the following hold.

(1) Assume that $d_1$ is even. Then the following items are equivalent:
\begin{itemize}
\item[(1-i)]
the transition subalgebra $T(\ggg_\bbk, e)$ admits one-dimensional representations;
\item[(1-ii)] there exists $\eta\in \chi+(\mathfrak{m}_{\bbk}^\bot)_{\bar{0}}$ such that $U_\eta(\ggg_\bbk)$ admits irreducible representations of dimension $p^{\frac{d_0}{2}}2^{\frac{d_1}{2}}$.
\end{itemize}

(2) Assume that $d_1$ is odd. Then the following items are equivalent:
\begin{itemize}
\item[(2-i)] the transition subalgebra $T(\ggg_\bbk, e)$ admits two-dimensional representations;
\item[(2-ii)] there exists $\eta\in \chi+(\mathfrak{m}_{\bbk}^\bot)_{\bar{0}}$ such that $U_\eta(\ggg_\bbk)$ admits  irreducible representations of dimension $p^{\frac{d_0}{2}}2^{\frac{d_1+1}{2}}$.
\end{itemize}
\end{prop}

\begin{proof} Let us first prove statement (2).

 ``(2-i)$\Rightarrow$ (2-ii)":
Recall that there is a ${\bbk}$-algebra isomorphism $U({\ggg}_{\bbk},e)\cong T({\ggg}_{\bbk},e)\otimes_{\bbk}Z_p(\mathfrak{a}_{\bbk})$ by Theorem~\ref{keyisotheorem}(3). Thus the assumption (2-i) implies that the ${\bbk}$-algebra $U({\ggg}_{\bbk},e)$ affords a two-dimensional  representation too; we denote it by $\nu$ with the representation space $V$.

Let $v_{\bar{0}}\in V_{\bar{0}}$ be a nonzero even vector in $V$. The proof of Proposition~\ref{no1} shows that $\Theta_{l+q+1}.v_{\bar{0}}\in V_{\bar{1}}$ is a nonzero odd vector, and we denote it by $v_{\bar{1}}$. Then $V$ is $\bbk$-spanned by $v_{\bar{0}}$ and $v_{\bar{1}}$ as a vector space. For any $\bar x\in({\ggg}_{\bbk})_{\bar{0}}$, since $\bar x^p-\bar x^{[p]}\in Z_p({\ggg}_{\bbk})$ is central in  $U({\ggg}_{\bbk})$, we have $[\rho(\bar x^p-\bar x^{[p]}),\Theta_{l+q+1}]=0$.
Therefore, both ${\bbk}v_{\bar{0}}$ and ${\bbk}v_{\bar{1}}$
are one-dimensional  representations of $\rho_{\bbk}(Z_p)$, decided by the same function on $(\ggg_\bbk)_{\bar0}$.
By Theorem~\ref{keyisotheorem}, $\rho_{\bbk}(Z_p)\cap\text{ker}\nu$
is a maximal ideal of the algebra
$\rho_{\bbk}(Z_p)\cong Z_p(\widetilde{{\ppp}}_{\bbk})\cong {\bbk}[(\chi+(\mathfrak{m}_{\bbk}^\bot)_{\bar{0}})^{(1)}]$. Then there exists $\eta\in\chi+(\mathfrak{m}_{\bbk}^\bot)_{\bar{0}}$  such that $\rho_{\bbk}(\bar x^p-\bar x^{[p]}-\eta(\bar x)^p)\in\text{Ker}\nu$ for all
$\bar x\in({\ggg}_{\bbk})_{\bar{0}}$.  Our choice of $\eta$ ensures that the ``reduced" quotient algebra $\widetilde U_\eta({\ggg}_{\bbk},e)$
affords a two-dimensional  representation. It follows from Lemma \ref{isoofreducedW-alg} and Theorem \ref{reducedfunctors} that the $\bbk$-algebra $U_\eta({\ggg}_{\bbk})$ has an irreducible representation of dimension $p^\frac{d_0}{2}2^{\frac{d_1+1}{2}}$.

``(2-ii)$\Rightarrow$(2-i)": Conversely, under the assumption (2-ii) it follows from Theorem \ref{reducedfunctors} that the reduced $W$-superalgebra $U_\eta(\ggg_\bbk,e)$ admits a two-dimensional representation, thereby does the $\bbk$-algebra $\widetilde U_\eta(\ggg_\bbk,e)$, then the ``reduced" quotient of $U(\ggg_\bbk,e)$ by Lemma \ref{isoofreducedW-alg}. Since the transition subalgebra $T({\ggg}_{\bbk},e)$ is a subalgebra of $U({\ggg}_{\bbk},e)$ by the definition, $T(\ggg_\bbk,e)$ also affords a two-dimensional representation.

The same arguments also go through for the proof of statement (1), which will be omitted here. We complete the proof.
\end{proof}

\section{Conjectural one-dimensional representations for finite $W$-superalgebras when $d_1$ is even}\label{conjectureseceven}

In this and the next sections we proceed to investigate small representations for the finite $W$-superalgebra $U({\ggg}_\mathbb{F},e)$ both over the field of complex numbers $\bbf={\bbc}$ and over a field $\bbf={\bbk}$ of positive characteristic. We will find that the parity of $d_1$ plays a key role for the dimensions of the small representations of $U({\ggg}_\mathbb{F},e)$, which is significantly different from the finite $W$-algebras. We will present a plausible conjecture for such dimensions, and demonstrate the conjecture with some examples. Based on these results, we will discuss the accessibility of the lower bounds of dimensions predicted in the super Kac-Weisfeiler property \cite[Theorems 4.3 and 5.6]{WZ} in \S\ref{5}. For simplicity we will always assume that the characteristic of
the field ${\bbk}=\overline{\mathbb{F}}_p$ satisfies $p\gg0$ unless otherwise specified, in this and the next sections.

For the ordinary finite $W$-algebra counterpart of the above issue, there are some remarkable work and exciting progress (cf. \cite{P7}). 
However, when we turn to the study of the finite $W$-superalgebra case, the tool available is very limited. 
  Therefore, the issue of minimal dimensions for the representations of finite $W$-superalgebras over ${\bbc}$ is in a position of reasonably estimating, but far from the solution.

\subsection{On the minimal-dimension conjecture when $d_1$ is even} Recall that the parity of $d_1$ plays the key role for the construction of finite $W$-superalgebra $U({\ggg},e)$ in \cite[Theorem 4.5]{ZS2}. Based on the different parity of $d_1$, for each case we will consider separately. This section is devoted to the case when $d_1$ is even.

\begin{conj}\label{Conject1}
When $d_1$ is even, the $\bbc$-algebra $U({\ggg},e)$ affords a one-dimensional  representation.
\end{conj}

Under the assumption that Conjecture \ref{Conject1} holds, we can assume that this one-dimensional representation $\textsf{O}:={\bbc}\sfo$ is generated by a nonzero vector $\sfo$. Now we will go on investigating the consequences.

\subsection{The analogue of communicative quotients for finite $W$-superalgebras in the even case}\label{4.1.1}

Recall that the ${\bbc}$-algebra $U({\ggg},e)$ is generated by ${\bbz}_2$-homogeneous elements $\Theta_1,\cdots,\Theta_l\in U({\ggg},e)_{\bar0}$ and $\Theta_{l+1},\cdots,\Theta_{l+q}\in U({\ggg},e)_{\bar1}$ in \cite[Theorem 4.5]{ZS2}. Let $M$ be any $U({\ggg},e)$-module. For a given odd element $u\in U({\ggg},e)_{\bar1}$  and  a homogeneous vector $m\in M$, we know that $m$ and $u.m$ have different parity. As $\mathsf{O}$ is a one-dimensional superspace, then $\Theta_i.\sfo=0$ for $l+1\leqslant  i\leqslant  l+q$ by parity consideration. Set $\Theta_i.\sfo=c_i\sfo$ for $1\leqslant  i\leqslant  l$ with $c_i\in{\bbc}$. Recall \cite[Theorem 4.7]{ZS2} shows that the algebra $U({\ggg},e)$ is completely determined by the commuting relations of $\Theta_1,\cdots,\Theta_{l+q}$ (see also \S\ref{2.2.2}). Based on the parity of these generators, for each case we will consider separately.

(i) For $1\leqslant  i<j\leqslant  l$, the element $[\Theta_i,\Theta_j]$ is even since $\Theta_i,\Theta_j\in U({\ggg},e)_{\bar0}$. It is immediate from $[\Theta_i,\Theta_j].\sfo=(\Theta_i\cdot\Theta_j-\Theta_j\cdot\Theta_i).\sfo
=(c_ic_j-c_jc_i).\sfo=0$ that the polynomial superalgebra
$F_{ij}(\Theta_1,\cdots,\Theta_{l+q})$ in $l+q$ variables acts on $\sfO$ trivially. When we put each polynomial $F_{ij}(\Theta_1,\cdots,\Theta_{l+q})$ as a ${\bbc}$-linear combination of $\Theta_1^{a_1}\cdots\Theta_l^{a_l}\Theta_{l+1}^{b_1}\cdots\Theta_{l+q}^{b_q}$, deleting all the terms for which any of the odd elements $\Theta_{l+1},\cdots,\Theta_{l+q}$ occurs,
one can obtain an (ordinary) polynomial in $l$ variables, and denote it by $F'_{ij}(\Theta_1,\cdots,\Theta_{l})$. Since $\Theta_i.\sfo=0$ for $l+1\leqslant  i\leqslant  l+q$ by the preceding remark, then $F'_{ij}(\Theta_1,\cdots,\Theta_{l}).\sfo=0$ for $1\leqslant  i<j\leqslant  l$.

(ii) For $l+1\leqslant  i\leqslant  j\leqslant  l+q$, the element $[\Theta_i,\Theta_j]$ is still even since $\Theta_i,\Theta_j$ are both odd. As $\Theta_i.\sfo=0$ for $l+1\leqslant  i\leqslant  l+q$, we have $[\Theta_i,\Theta_j].\sfo=(\Theta_i\cdot\Theta_j+\Theta_j\cdot\Theta_i).\sfo=0$. By the same discussion as (i) we can also get polynomials $F'_{ij}(\Theta_1,\cdots,\Theta_{l})$ for $l+1\leqslant  i\leqslant  j\leqslant  l+q$, and the one-dimensional property of $\sfO$ entails that 
$F'_{ij}(\Theta_1,\cdots,\Theta_{l}).\sfo=0$.

(iii) For $1\leqslant  i\leqslant  l<j\leqslant  l+q$, the element $[\Theta_i,\Theta_j]$ is odd since $\Theta_i\in U({\ggg},e)_{\bar0}$ and $\Theta_j\in U({\ggg},e)_{\bar1}$. As $\Theta_i.\sfo=0$ for $l+1\leqslant  i\leqslant  l+q$, we have $[\Theta_i,\Theta_j].\sfo=(\Theta_i\cdot\Theta_j-\Theta_j\cdot\Theta_i).\sfo=0$, which entails that 
$F_{ij}(\Theta_1,\cdots,\Theta_{l+q})$ acts on $\sfO$ trivially. It is immediate from $[\Theta_i,\Theta_j]=F_{ij}(\Theta_1,\cdots,\Theta_{l+q})$ that all $F_{ij}(\Theta_1,\cdots,\Theta_{l+q})$'s are odd elements in $U(\ggg,e)$. Therefore, when we put each polynomial $F_{ij}(\Theta_1,\cdots,\Theta_{l+q})$ as a ${\bbc}$-linear combination of monomials $\Theta_1^{a_1}\cdots\Theta_l^{a_l}\Theta_{l+1}^{b_1}\cdots\Theta_{l+q}^{b_q}$, in each given monomial some odd element $\Theta_k$ with $l+1\leqslant  k\leqslant  l+q$ will occur at least once. Since $\Theta_k.\sfo=0$ for $l+1\leqslant  k\leqslant  l+q$, the equations $F_{ij}(\Theta_1,\cdots,\Theta_{l+q}).\sfo=0$ are trivial for $1\leqslant  i\leqslant  l<j\leqslant  l+q$. In this case no new equations are obtained.

Keep in mind all polynomials $F'_{ij}$ from the above arguments. Actually, since the polynomials $F_{ij}(\Theta_1,\cdots,\Theta_{l+q})$'s give rise to the defining relations of $U(\ggg,e)$, from all above one can conclude that the one-dimensional modules of $U(\ggg,e)$ are completely determined by the polynomials $F'_{ij}(\Theta_1,\cdots,\Theta_{l})$'s.
Set $U({\ggg},e)^{\text{ab}}$ to be the quotient algebra of $U({\ggg},e)$ by $R$, where $R$ is the ideal of $U({\ggg},e)$ generated by all the odd generators $\Theta_{l+1},\cdots,\Theta_{l+q}$ and all commutators $[a, b]$ with $a,b\in U({\ggg},e)$. Then $U({\ggg},e)^{\text{ab}}$ is isomorphic to the algebra ${\bbc}[T_1,\cdots,T_l]/\Lambda$, where ${\bbc}[T_1,\cdots,T_l]$ is an (ordinary) polynomial algebra in $l$ variables, and $\Lambda$  the ideal of ${\bbc}[T_1,\cdots,T_l]$ generated by all $F'_{ij}(T_1,\cdots,T_l)$'s for $1\leqslant  i<j\leqslant  l$ and $l+1\leqslant  i\leqslant  j\leqslant  l+q$.
Such a commutative quotient in the Lie algebra case is studied by Premet in \cite{P7}, which leads to a lot of understanding on the small representations of finite $W$-algebras. Now we exploit this machinery in the super case. In fact, combining with all the discussions above, Hilbert's Nullstellensatz shows that the maximal spectrum $\mathscr{E}:=\text{Specm}\,U({\ggg},e)^{\text{ab}}$ parameterizes the one-dimensional representations of $U({\ggg},e)$. Denoting by $\mathscr{E}(\bbc)$ the set of all common zeros of the polynomials $F'_{ij}$ for $1\leqslant  i<j\leqslant  l$ and $l+1\leqslant  i\leqslant  j\leqslant  l+q$ in the affine space $\mathbb{A}^l_{\bbc}$, we have

\begin{lemma}\label{d_1even}
When $d_1$ is even, the Zariski closed set $\mathscr{E}(\bbc)$ parameterizes the one-dimensional  representations of finite $W$-superalgebra $U({\ggg},e)$.
\end{lemma}

Note that for $1\leqslant  i<j\leqslant  l$, or $l+1\leqslant  i\leqslant  j\leqslant  l+q$,  all the coefficients of $F_{ij}$'s are over the admissible ring $A$. Thus all the coefficients of $F'_{ij}$'s are also over $A$ by the definition.
Set $^p{\hskip-0.05cm}F'_{ij}:=F'_{ij}\otimes_A{\bbk}$, the polynomials over ${\bbk}$, and denote by $\mathscr{E}({\bbk})$ the set of all common zeros of the polynomials $^p{\hskip-0.05cm}F_{ij}'$ in the affine space $\mathbb{A}^l_{\bbk}$ with $1\leqslant  i<j\leqslant  l$, or $l+1\leqslant  i\leqslant  j\leqslant  l+q$. Since the transition subalgebra $T({\ggg}_{\bbk},e)$ over ${\bbk}$ is induced from the $\bbc$-algebra $U({\ggg},e)$ by ``modular $p$ reduction'', Theorem \ref{transition} and Lemma \ref{d_1even} show that  the Zariski closed set $\mathscr{E}({\bbk})$ parametrises the one-dimensional  representations of the ${\bbk}$-algebra $T({\ggg}_{\bbk},e)$. By the same arguments as in \cite[Theorem 2.2(a)]{P7}, one can verify that

\begin{lemma}\label{trans1} Assume that  $d_1$ is even. If the finite $W$-superalgebra $U({\ggg},e)$ over ${\bbc}$ affords one-dimensional  representations, then the transition subalgebra $T({\ggg}_{\bbk},e)$ over ${\bbk}=\overline{\mathbb{F}}_p$ also admits one-dimensional representations.
\end{lemma}

Now we can talk about the small representations for the reduced enveloping algebra of a basic Lie superalgebra. The following result is an immediate consequence of Proposition \ref{transitionforminimaldim}(1) and Lemma \ref{trans1}.

\begin{lemma}\label{mindim1}
When $d_1$ is even, if the finite $W$-superalgebra $U({\ggg},e)$ over ${\bbc}$ affords a one-dimensional  representation, then for $p\gg0$ there exists $\eta\in\chi+(\mathfrak{m}_{\bbk}^\bot)_{\bar{0}}$  such that the reduced enveloping algebra $U_\eta({\ggg}_{\bbk})$ admits irreducible representations of dimension $p^{\frac{d_0}{2}}2^{\frac{d_1}{2}}$.
\end{lemma}

\subsection{Confirmation of Conjecture \ref{Conject1} for $\gl(M|N)$ and $\sssl(M|N)$}\label{4.1.3}
 In \cite{WZ} Wang and Zhao gave some explicit description for the Dynkin gradings of basic Lie superalgebras of all types.  In particular, they showed that $\dim\mathfrak{gl}(M|N)^e_{\bar1}$ is always an even number  for any nilpotent element $e\in\mathfrak{gl}(M|N)_{\bar0}$ (cf. \cite[\S3.2]{WZ}). As the dimension of $\mathfrak{gl}(M|N)_{\bar1}$ is also even, then $d_1=\dim\gl(M|N)_\bo-\dim\gl(M|N)^e_\bo$ is always an even number. It is notable that although Wang-Zhao's origin arguments are carried over the field $\bbk$ in positive characteristic (see Table 1 in \S\ref{basicLiesuper}), as showed in \cite[Remark 3.2]{WZ}, all the discussions there are still valid for the case of complex numbers. Actually, in order to confirm Conjecture \ref{Conject1} for these complex finite $W$-superalgebras of $\gl(M|N)$ and $\sssl(M|N)$, we will firstly consider whose transition subalgebra $T({\ggg}_{\bbk},e)$ in positive characteristic.
\begin{lemma}\label{glsl}
Let ${\ggg}_{{\bbk}}$ be a Lie superalgebra of type $\mathfrak{gl}(M|N)$ or $\mathfrak{sl}(M|N)$ with $M,N\in{\bbz}_+$. For any nilpotent element $e\in({\ggg}_{{\bbk}})_{\bar0}$, the ${\bbk}$-algebra $T({\ggg}_{\bbk},e)$ affords a one-dimensional  representation.
\end{lemma}

\begin{proof}
For the Lie superalgebra ${\ggg}_{{\bbk}}=\mathfrak{gl}(M|N)$ or $\mathfrak{sl}(M|N)$,
in \cite{ZS3} the authors showed that the reduced enveloping algebra $U_\chi({\ggg}_{\bbk})$ admits an irreducible representation of dimension $p^{\frac{d_0}{2}}2^{\frac{d_1}{2}}$ over ${\bbk}=\overline{\mathbb{F}}_p$ under the assumption that (i) $p>2$ when ${\ggg}=\mathfrak{gl}(M|N)$; (ii) $p>2$ and $p\nmid(M-N)$ when ${\ggg}=\mathfrak{sl}(M|N)$. Since $d_1$ is even in this case,
we have $\text{\underline{dim}}\,\mathfrak{m}=(\frac{d_0}{2},\frac{d_1}{2})$ by \textsection\ref{2.2.1}.
Hence Proposition \ref{transitionforminimaldim}(1) yields that the $\bbk$-algebra $T({\ggg}_{\bbk},e)$ admits a one-dimensional  representation.
\end{proof}
The following result is an immediate consequence of field theory.
\begin{lemma}\label{transtoc}
Let ${\ggg}$ be a basic Lie superalgebra over ${\bbc}$. When $d_1$ is even, if the transition subalgebra $T({\ggg}_{\bbk},e)$ with ${\bbk}=\overline{\mathbb{F}}_p$ affords one-dimensional  representations for infinitely many $p\in\Pi(A)$, then the finite $W$-superalgebra $U({\ggg},e)$ over ${\bbc}$ has a one-dimensional  representation.
\end{lemma}

\begin{proof}
Recall Lemma~\ref{d_1even} and what follows show that the one-dimensional representations of finite $W$-superalgebra $U(\ggg,e)$ over $\bbc$ can be parametrised by the Zariski closed set $\mathscr{E}({\bbc})$, and the transition subalgebra $T(\ggg_\bbk,e)$ by $\mathscr{E}({\bbk})$. The lemma follows the same means as the Lie algebra case \cite[Corollary 2.1]{P7} by the knowledge of Galois theory, thus will be omitted here.
\end{proof}

Now we are in a position to introduce the main result of this subsection.

\begin{prop}\label{glsl2}
Let ${\ggg}=\mathfrak{gl}(M|N)$ or $\mathfrak{sl}(M|N)$ ($M,N\in{\bbz}_+$) over ${\bbc}$. For any nilpotent element $e\in{\ggg}_{\bar0}$, the finite $W$-superalgebra $U({\ggg},e)$ affords a one-dimensional  representation.
\end{prop}

\begin{proof}
The proposition readily follows from Lemma~\ref{glsl} and Lemma~\ref{transtoc}.
\end{proof}

\section{Conjectural two-dimensional representations for finite $W$-superalgebras when $d_1$ is odd}\label{conjecturesecodd}

In this section we will always assume that $d_1$ is odd.

\subsection{On the minimal-dimension conjecture when $d_1$ is odd}

In virtue of Proposition~\ref{no1}, we can formulate the following conjecture on the minimal-dimension representations of $U({\ggg},e)$ when $d_1$ is odd.
\begin{conj}\label{Conject2}
When $d_1$ is odd, the finite $W$-superalgebra $U({\ggg},e)$ over $\bbc$ affords a two-dimensional representation of minimal dimension.
\end{conj}
In the section, we mainly investigate such two-dimensional modules under the assumption that the above conjecture holds, and also confirm the conjecture for the regular nilpotent elements of $\mathfrak{osp}(1|2n)$. Take such a module as in Conjecture \ref{Conject2} and denote it by $V$. Recall that in \cite[\S4.4]{ZS2} we introduced the refined finite $W$-superalgebra $W'_\chi:=Q_\chi^{\ad\mmm'}$ over $\bbc$, a proper subalgebra of $U({\ggg},e)$. In virtue of this algebra, we can formulate a more precise description for the two-dimensional  representation $V$ of the ${\bbc}$-algebra $U({\ggg},e)$.

\begin{prop}\label{typeq}
Any irreducible $U({\ggg},e)$-module in Conjecture~\ref{Conject2} can be regarded  as a $W'_\chi$-module of type $Q$.
\end{prop}

\begin{proof}Let $V$ be a two-dimensional  representation of the ${\bbc}$-algebra $U({\ggg},e)$ in Conjecture~\ref{Conject2}. Define the ${\bbc}$-mapping
\[\begin{array}{lcll}
\tau:&V&\rightarrow&V\\ &v&\mapsto&\sqrt{2}\Theta_{l+q+1}.v.
\end{array}
\]
It is easy to verify that the mapping $\tau$ is odd and surjective. In fact, $\tau$ is also injective since $\tau^2(v)=2\Theta_{l+q+1}^2.v=v$. We claim that $\tau$ is a homomorphism of $W'_\chi$-module $V$.

For any ${\bbz}_2$-homogeneous elements $\Theta\in W'_\chi$ and $v\in V$, we have $\tau(\Theta.v)=\sqrt{2}\Theta_{l+q+1}.\Theta.v$ by the definition. Since $\Theta_{l+q+1}\in\mathfrak{m}'$ and $\Theta\in Q_\chi^{\ad\mmm'}$, then $[\Theta_{l+q+1},\Theta]=0$. Moreover, $[\Theta_{l+q+1},\Theta]=\Theta_{l+q+1}\cdot\Theta-(-1)^{|\Theta|}\Theta\cdot\Theta_{l+q+1}$. It is immediate that $\Theta_{l+q+1}\cdot\Theta=(-1)^{|\Theta|}\Theta\cdot\Theta_{l+q+1}$, then $\Theta_{l+q+1}.\Theta.v=(-1)^{|\Theta|}\Theta.\Theta_{l+q+1}.v$, i.e. $\tau(\Theta.v)=(-1)^{|\Theta|}\Theta.\tau(v)$.

As $V$ is also irreducible as a $W'_\chi$-module, all the discussions above imply that $V$ is of type $Q$, completing the proof.
\end{proof}

\subsection{The analogue of communicative quotients for finite $W$-superalgebras in the odd case}\label{4.2}
Recall that in \cite[Theorem 4.7]{ZS2} we have chosen the ${\bbz}_2$-homogeneous elements $\Theta_1,\cdots,\Theta_{l+q+1}$ as a set of generators for the ${\bbc}$-algebra $U({\ggg},e)$ subject to the relations
$$[\Theta_i,\Theta_j]=F_{ij}(\Theta_1,\cdots,\Theta_{l+q+1}),~~1\leqslant i,j\leqslant l+q+1.$$
To ease notation, we will denote $F_{ij}(\Theta_1,\cdots,\Theta_{l+q+1})$ by $F_{ij}$ for short.

If the ${\bbc}$-algebra $U({\ggg},e)$ affords a two-dimensional  representation $V$, Proposition~\ref{typeq} yields that $V$ is ${\bbc}$-spanned by an even element $v\in V_{\bar0}$ and the odd element $\Theta_{l+q+1}.v\in V_{\bar1}$. Hence we can get $4(l+q+1)$ variables $k_i^0,k_i^1,K_i^0,K_i^1\in{\bbc}$ such that
\begin{equation}\label{kKdef}
\Theta_{i}.v=k_i^0v+k_i^1\Theta_{l+q+1}.v,\qquad\Theta_{i}.\Theta_{l+q+1}.v=K_i^0v+K_i^1\Theta_{l+q+1}.v,
\end{equation}where $1\leqslant i\leqslant l+q+1$.

Similarly, there exist $4(l+q+1)^2$ variables
$(F_{ij})^0_{\bar{0}},(F_{ij})^0_{\bar{1}},(F_{ij})^1_{\bar{0}},(F_{ij})^1_{\bar{1}}\in{\bbc}$ with $ 1\leqslant i,j\leqslant l+q+1$ such that
\begin{equation}\label{F_ijdefn}
F_{ij}.v=(F_{ij})^0_{\bar{0}}v+(F_{ij})^0_{\bar{1}}\Theta_{l+q+1}.v,\qquad F_{ij}.\Theta_{l+q+1}.v=(F_{ij})^1_{\bar{0}}v+(F_{ij})^1_{\bar{1}}\Theta_{l+q+1}.v.
\end{equation}

It is worth noting that each polynomial $F_{ij}$ is generated by the ${\bbz}_2$-homogeneous elements $\Theta_1,\cdots,\Theta_{l+q+1}$ of $U({\ggg},e)$ over $\mathbb{Q}$. After enlarging the admissible ring $A$ possibly, one can further assume that the $F_{ij}$'s for $1\leqslant i,j\leqslant l+q+1$ are defined over $A$. Since the action of $F_{ij}$ on $v$ and $\Theta_{l+q+1}.v$ is completely determined by the constants in \eqref{kKdef}, then $(F_{ij})^0_{\bar{0}},(F_{ij})^0_{\bar{1}},(F_{ij})^1_{\bar{0}},(F_{ij})^1_{\bar{1}}$ can be written as an $A$-linear combination of the monomials in $k_i^0,k_i^1,K_i^0,K_i^1$, thus there are no new variables appear in \eqref{F_ijdefn}.

Since each $\Theta_i$ is ${\bbz}_2$-homogeneous for $1\leqslant i\leqslant l+q+1$, it follows that $2(l+q+1)$ variables in \eqref{kKdef} equal zero. More precisely, $k_i^1=K_i^0=0$ for $1\leqslant i\leqslant l$  (in this case all $\Theta_i$'s are even) and $k_i^0=K_i^1=0$ for $l+1\leqslant i\leqslant l+q+1$ (in this case all $\Theta_i$'s are odd). By the definition it is obvious that all $F_{ij}$'s are ${\bbz}_2$-graded and have the same parity with $[\Theta_i,\Theta_j]$ for $1\leqslant i,j\leqslant l+q+1$. Therefore, $2(l+q+1)^2$ variables equal zero in \eqref{F_ijdefn}. More precisely, $(F_{ij})^0_{\bar{1}}=(F_{ij})^1_{\bar{0}}=0$ if $F_{ij}$ is even, and $(F_{ij})^0_{\bar{0}}=(F_{ij})^1_{\bar{1}}=0$ if $F_{ij}$ is odd.

In virtue of \eqref{kKdef}, simple calculation shows that
\begin{equation}\label{Thetaijcom}
\begin{split}
\Theta_i.\Theta_j.v&=(k_i^0 k_j^0+K_i^0k_j^1)v+(k_i^1k_j^0+K_i^1k_j^1)\Theta_{l+q+1}.v,\\
\Theta_i.\Theta_j.\Theta_{l+q+1}.v&=(k_i^0K_j^0+K_i^0K_j^1)v+(k_i^1K_j^0+K_i^1K_j^1)\Theta_{l+q+1}.v
\end{split}
\end{equation} for $1\leqslant i,j\leqslant l+q+1$.

Recall that the structure of ${\bbc}$-algebra $U({\ggg},e)$ is completely determined by a data of generators and defining relations (cf.  \cite[Theorem 4.7]{ZS2}). Hence, $V$ is completely decided by  the following equalities
\begin{equation}\label{Thetaijcomre}
\begin{split}
(\Theta_i.\Theta_j-(-1)^{|\Theta_i||\Theta_j|}\Theta_j\Theta_i-F_{ij}).v&=0, \\ (\Theta_i.\Theta_j-(-1)^{|\Theta_i||\Theta_j|}\Theta_j\Theta_i-F_{ij}).\Theta_{l+q+1}.v&=0
\end{split}
\end{equation}for $1\leqslant i,j\leqslant l+q+1$.
Simple calculation based on \eqref{Thetaijcom} and \eqref{Thetaijcomre} shows that the variables $k_i^0,k_i^1,K_i^0,K_i^1$ and $(F_{ij})^0_{\bar{0}},(F_{ij})^0_{\bar{1}},(F_{ij})^1_{\bar{0}},(F_{ij})^1_{\bar{1}}$ for $1\leqslant i,j\leqslant l+q+1$ should satisfy the following system of linear equations:
\[\begin{array}{rcl}
(k_i^0k_j^0+K_i^0k_j^1)-(-1)^{|\Theta_i||\Theta_j|}(k_i^0k_j^0+k_i^1K_j^0)-(F_{ij})^0_{\bar{0}}&=&0;\\
(k_i^1k_j^0+K_i^1k_j^1)-(-1)^{|\Theta_i||\Theta_j|}(k_i^0k_j^1+k_i^1K_j^1)-(F_{ij})^0_{\bar{1}}&=&0;\\
(k_i^0K_j^0+K_i^0K_j^1)-(-1)^{|\Theta_i||\Theta_j|}(K_i^0k_j^0+K_i^1K_j^0)-(F_{ij})^1_{\bar{0}}&=&0;\\
(k_i^1K_j^0+K_i^1K_j^1)-(-1)^{|\Theta_i||\Theta_j|}(K_i^0k_j^1+K_i^1K_j^1)-(F_{ij})^1_{\bar{1}}&=&0.
\end{array}\]

It is notable that $4(l+q+1)$ variables are involved in above equations (recall that $F_{ij}$'s ($1\leqslant i,j\leqslant l+q+1$) can be written as an $A$-linear combination of the monomials in $k_i^0,k_i^1,K_i^0,K_i^1$ with $1\leqslant i\leqslant l+q+1$), and the remark preceding \eqref{Thetaijcom} shows that $2(l+q+1)$ variables equal zero. Set $${\bbc}[X_1^0,X_2^0\cdots,X_{l}^0,X_{l+1}^1,\cdots,X_{l+q+1}^1,Y_1^1,Y_2^1,\cdots,Y_{l}^1,Y_{l+1}^0\cdots,Y_{l+q+1}^0]$$to
be an (ordinary) polynomial algebra in $2(l+q+1)$ variables over ${\bbc}$. For $1\leqslant i\leqslant l+q+1$, substitute the constants $k_i^0,k_i^1,K_i^0,K_i^1 $ for the variables $X_i^0,X_i^1,Y_i^0,Y_i^1$ respectively, and define the polynomials
\[\begin{array}{rcl}
A_{ij}&:=&(X_i^0X_j^0+X_j^1Y_i^0)-(-1)^{|\Theta_i||\Theta_j|}(X_i^0X_j^0+X_i^1Y_j^0)-S_{ij}^0;\\
B_{ij}&:=&(X_i^1X_j^0+X_j^1Y_i^1)-(-1)^{|\Theta_i||\Theta_j|}(X_i^0X_j^1+X_i^1Y_j^1)-S_{ij}^1;\\
C_{ij}&:=&(X_i^0Y_j^0+Y_i^0Y_j^1)-(-1)^{|\Theta_i||\Theta_j|}(X_j^0Y_i^0+Y_i^1Y_j^0)-T_{ij}^0;\\
D_{ij}&:=&(X_i^1Y_j^0+Y_i^1Y_j^1)-(-1)^{|\Theta_i||\Theta_j|}(X_j^1Y_i^0+Y_i^1Y_j^1)-T_{ij}^1
\end{array}\]for $1\leqslant i,j\leqslant l+q+1$,
where $S_{ij}^0,S_{ij}^1,T_{ij}^0,T_{ij}^1$ stand for the polynomials over $A$ obtained by substituting the variables $k_i^0,k_i^1,K_i^0,K_i^1$ in the polynomials $(F_{ij})^0_{\bar{0}},(F_{ij})^0_{\bar{1}},$\\$(F_{ij})^1_{\bar{0}},(F_{ij})^1_{\bar{1}}$ for the indeterminate $X_i^0,X_i^1,Y_i^0,Y_i^1$, respectively. By the preceding remark, for the terms in $A_{ij},B_{ij},C_{ij},D_{ij}$ with $1\leqslant i,j\leqslant l+q+1$, we have

(1) $X_i^1=Y_i^0=0$ for $1\leqslant i\leqslant l$ (in this case $\Theta_i's$ are even);

(2) $X_i^0=Y_i^1=0$ for $l+1\leqslant i\leqslant l+q+1$ (in this case $\Theta_i's$ are odd);

(3) $S_{ij}^1=T_{ij}^0=0$ when $1\leqslant i,j\leqslant l$, or $l+1\leqslant i,j\leqslant l+q+1$ (in this case $F_{ij}$'s are even);

(4) $S_{ij}^0=T_{ij}^1=0$ when $1\leqslant i\leqslant l<j\leqslant l+q+1$, or $1\leqslant j\leqslant l<i\leqslant l+q+1$ (in this case $F_{ij}$'s are odd).

It follows from \eqref{Thetaijcom}, \eqref{Thetaijcomre} and \cite[Theorem 4.7]{ZS2} that there is a $1$-$1$ correspondence between the two-dimensional  representations of ${\bbc}$-algebra $U({\ggg},e)$ and the points of all common zeros of the polynomials $A_{ij}, B_{ij}, C_{ij}, D_{ij}$ in $2(l+q+1)$ variables for $1\leqslant i,j\leqslant l+q+1$ subject to conditions (1)-(4).

Given a subfield $K$ of ${\bbc}$ containing $A$ we denote by $\mathscr{E}(K)$ the set of all common zeros of the polynomials
$A_{ij},B_{ij},C_{ij},D_{ij}$ for $1\leqslant i,j\leqslant l+q+1$ in the affine space $\mathbb{A}^{2(l+q+1)}_K$ subject to conditions (1)-(4). Clearly, the $A$-defined Zariski closed set $\mathscr{E}(\bbc)$ parameterizes the two-dimensional  representations of ${\bbc}$-algebra $U({\ggg},e)$. More precisely,

\begin{lemma}\label{defnodd}
 Assume that $d_1$ is odd. Then the two-dimensional  representations of ${\bbc}$-algebra $U({\ggg},e)$ are uniquely determined by all common zeros of the polynomials $A_{ij},B_{ij},C_{ij},D_{ij}$ ($1\leqslant i,j\leqslant l+q+1$) subject to conditions (1)-(4) in the affine space $\mathbb{A}^{2(l+q+1)}_{\bbc}$.
\end{lemma}

Similarly, let $\mathscr{E}({\bbk})$ be the set of common zeros of the polynomials $^p{\hskip-0.05cm}A_{ij},^p{\hskip-0.05cm}B_{ij},^p{\hskip-0.05cm}C_{ij},$\\$^p{\hskip-0.05cm}D_{ij}$ subject to the ``modular $p$'' version of the conditions (1)-(4) in the affine space $\mathbb{A}^{2(l+q+1)}_{\bbk}$ with $1\leqslant i,j\leqslant l+q+1$, where $^p{\hskip-0.05cm}A_{ij},^p{\hskip-0.05cm}B_{ij},^p{\hskip-0.05cm}C_{ij},^p{\hskip-0.05cm}D_{ij}$ stand for the polynomials over ${\bbk}$ obtained from $A_{ij},B_{ij},C_{ij},D_{ij}$ by ``modular $p$ reduction'', i.e.
\[\begin{array}{rl}
&{\bbk}[X_1^0,X_2^0\cdots,X_{l}^0,X_{l+1}^1,\cdots,X_{l+q+1}^1,Y_1^1,Y_2^1,\cdots,Y_{l}^1,Y_{l+1}^0\cdots,Y_{l+q+1}^0]\\
=&A[X_1^0,X_2^0\cdots,X_{l}^0,X_{l+1}^1,\cdots,X_{l+q+1}^1,Y_1^1,Y_2^1,\cdots,Y_{l}^1,Y_{l+1}^0\cdots,Y_{l+q+1}^0]\otimes_A{\bbk}. \end{array}\]
It follows from Theorem~\ref{transition} that the Zariski closed set $\mathscr{E}({\bbk})$ parameterizes the two-dimensional  representations of the ${\bbk}$-algebra $T({\ggg}_{\bbk},e)$.

Following Premet's treatment to the Lie algebra case in \cite[Theorem 2.2(a)]{P7}, we have

\begin{lemma}\label{ck2}
When $d_1$ is odd, if the ${\bbc}$-algebra $U({\ggg},e)$ affords two-dimensional  representations, then the transition subalgebra $T({\ggg}_{\bbk},e)$ also admits two-dimensional  representations.
\end{lemma}
\begin{proof} Taking  Lemma~\ref{defnodd} into account, we can prove the lemma by the same discussion as Lemma~\ref{trans1}. The detailed arguments are  omitted here.
\end{proof}

As a corollary of the above lemma and Proposition \ref{transitionforminimaldim}(2), we have

\begin{lemma}\label{cdim2}
When $d_1$ is odd, if the finite $W$-superalgebra  $U({\ggg},e)$ over ${\bbc}$ affords a two-dimensional representation, then for $p\gg0$ there exists $\eta\in\chi+(\mathfrak{m}_{\bbk}^\bot)_{\bar{0}}$  associated to which the reduced enveloping algebra $U_\eta({\ggg}_{\bbk})$ admits irreducible representations of dimension $p^{\frac{d_0}{2}}2^{\frac{d_1+1}{2}}$.
\end{lemma}

\subsection{Confirmation of Conjecture~\ref{Conject2} for ${\ggg}=\mathfrak{osp}(1|2n)$ with regular nilpotent elements}
We first need the following observation:

\begin{lemma}\label{bon}
Let ${\ggg}_{\bbk}=\mathfrak{osp}(1|2n)$ be a basic Lie superalgebra over $\bbk$. For any regular nilpotent element $e\in({\ggg}_{{\bbk}})_{\bar0}$, the transition subalgebra $T({\ggg}_{\bbk},e)$ affords a two-dimensional  representation.
\end{lemma}

\begin{proof}
First note that $d_1$ is odd in this case (cf. \cite[Corollary 2.10]{PS2}). Let ${\ggg}_{\bbk}=\mathfrak{n}_{\bbk}^+\oplus \mathfrak{h}_{\bbk}\oplus \mathfrak{n}_{\bbk}^-$ denote the canonical  triangular decomposition of Lie superalgebra $\mathfrak{osp}(1|2n)$. It follows from \cite[Corollary 5.8]{WZ} that $$\text{\underline{dim}}\,\mathfrak{n}_{\bbk}^-=\text{\underline{dim}}\,\mathfrak{m}'_{\bbk}=
(\text{dim}\,(\mathfrak{m}_{\bbk})_{\bar0},\text{dim}\,(\mathfrak{m}_{\bbk})_{\bar1}+1).$$ Moreover, the baby Verma module $Z_\chi(\lambda)$ of reduced enveloping algebra $U_\chi({\ggg}_{\bbk})$ associated to the regular $p$-character $\chi$ is irreducible (cf. \cite[Corollary 5.8]{WZ}), which has the same dimension as the vector space $U_\chi(\mathfrak{m}_{\bbk}')$. Theorem \ref{reducedfunctors} shows that $Z_\chi(\lambda)^{\mathfrak{m}_{\bbk}}$ is a $U_\chi(\mathfrak{osp}(1|2n),e)$-module, and there is an isomorphism of $U_\chi(\mathfrak{m}_{\bbk})$-modules $Z_\chi(\lambda)\cong U_\chi(\mathfrak{m}_{\bbk})^*\otimes_{\bbk}Z_\chi(\lambda)^{\mathfrak{m}_{\bbk}}$ by the proof of \cite[Proposition 4.2]{WZ}. Then we have
$$\text{dim}\,Z_\chi(\lambda)^{\mathfrak{m}_{\bbk}}=\frac{\text{dim}\,Z_\chi(\lambda)}
{\text{dim}\,U_\chi(\mathfrak{m}_{\bbk})}=\frac{\text{dim}\,U_\chi(\mathfrak{m}_{\bbk}')}
{\text{dim}\,U_\chi(\mathfrak{m}_{\bbk})}=2.$$
Therefore, the algebra $U_\chi({\ggg}_{\bbk},e)$ admits a two-dimensional  representation.
By the same discussion as the proof of Proposition \ref{transitionforminimaldim}(2), one can conclude that the transition subalgebra $T({\ggg}_{\bbk},e)$ also affords a two-dimensional  representation.
\end{proof}

Recall Lemma~\ref{defnodd} shows that the two-dimensional  representations of finite $W$-superalgebras over ${\bbc}$ can be parameterized by the Zariski closed set $\mathscr{E}({\bbc})$ for the case when $d_1$ is odd. By the same consideration as Lemma~\ref{transtoc}, one can also obtain that

\begin{lemma}\label{transtoc2}
Let ${\ggg}$ be a basic Lie superalgebra over ${\bbc}$. When $d_1$ is odd, if the transition subalgebra $T({\ggg}_{\bbk},e)$  affords two-dimensional  representations for infinitely many $p\in\Pi(A)$, then the finite $W$-superalgebra $U({\ggg},e)$ over ${\bbc}$ has a two-dimensional  representation.
\end{lemma}

Now we are in a position to introduce the main result of this subsection.

\begin{prop}\label{ii}
Let $e\in{\ggg}_{\bar0}$ be a regular nilpotent element in the Lie superalgebra ${\ggg}=\mathfrak{osp}(1|2n)$ over ${\bbc}$, then the finite $W$-superalgebra $U({\ggg},e)$ affords a two-dimensional  representation.
\end{prop}

\begin{proof}
The proposition readily follows from Lemma~\ref{bon} and Lemma~\ref{transtoc2}.
\end{proof}

\section{The realization of minimal-dimensional representations for reduced enveloping algebra $U_\xi({\ggg}_{\bbk})$}\label{5}

\subsection{Introduction} This section is devoted to the accessibility of dimensional lower bounds for the irreducible representations of $U_\xi({\ggg}_{\bbk})$ predicted by Wang-Zhao in \cite[Theorem 5.6]{WZ} as below:

\begin{prop}(\cite{WZ})\label{wzd2}
Let ${\ggg}_{\bbk}$ be a basic Lie superalgebra over ${\bbk}=\overline{\mathbb{F}}_p$, assuming that the prime $p$ satisfies the restriction imposed in \S\ref{prel}(Table 1). Let $\xi$ be arbitrary $p$-character in $({\ggg}_{\bbk})^*_{\bar0}$. Then the dimension of every $U_\xi({\ggg}_{\bbk})$-module $M$ is divisible by $p^{\frac{d_0}{2}}2^{\lfloor\frac{d_1}{2}\rfloor}$.
\end{prop}
Therefore, the number $p^{\frac{d_0}{2}}2^{\lfloor\frac{d_1}{2}\rfloor}$ becomes a lower bound of dimensions for the irreducible modules of $\bbk$-algebra $U_\xi(\ggg_\bbk)$.
A natural question is the accessibility of this number, i.e. whether or not there is any irreducible module of $ U_\xi(\ggg_\bbk)$ with dimension equaling such a lower bound.

For the Lie superalgebra ${\ggg}_{\bbk}$ of type $\mathfrak{gl}(M|N)$ or $\mathfrak{sl}(M|N)$ with $M,N\in{\bbz}_+$, in \cite{ZS3} the authors showed that every reduced enveloping algebra $U_\xi({\ggg}_{\bbk})$ has a ``small'' representation of dimension equaling the lower bound. The method applied there is to construct an appropriate parabolic subalgebra which has a one-dimensional module, then induce it to an irreducible representation of ${\ggg}_{\bbk}$. However, this method can not be easily exploited in the general case.
Thus the general attainableness of such lower bounds of dimensions in Proposition \ref{wzd2} is an open problem.

Now we first formulate Conjecture \ref{conjecture}, summarizing the previous two sections:

\vskip0.2cm

\noindent{\itshape{\bf Conjecture 0.3.}  Let ${\ggg}$ be a basic Lie superalgebra over ${\bbc}$.
\begin{itemize}
\item[(1)] When $d_1$ is even, the finite $W$-superalgebra $U({\ggg},e)$ affords a one-dimensional  representation;
\item[(2)] when $d_1$ is odd, the finite $W$-superalgebra $U({\ggg},e)$ affords a two-dimensional  representation.
\end{itemize}}

\vskip0.2cm

In virtue of Conjecture~\ref{conjecture}, we will prove that the lower bounds of the dimensions in Proposition \ref{wzd2} are accessible for $p\gg0$ in this section.
In the first part, we will deal with the case for nilpotent $p$-character $\chi\in(\ggg_\bbk)^*_{\bar0}$,  mainly following Premet's treatment for the Lie algebra case in \cite[\S2.8]{P7}, with a few modifications. One can observe that the emergence of odd part in the Lie superalgebra ${\ggg}_{\bbk}$ makes the situation much more complicated. In the second part, we will deal with the case for arbitrary $p$-character $\xi\in(\ggg_\bbk)^*_{\bar0}$, which may be not nilpotent. A lot of precise analysis on the modular representations of basic Lie superalgebras has to be done for the second case.
\subsection{On the dimensional lower bounds for the representations of basic Lie superalgebras with nilpotent $p$-characters} \label{proofmain0.5sub}
\subsubsection{}
Recall that in Lemmas \ref{mindim1}  and  \ref{cdim2} we have discussed the representations of minimal dimensions for the reduced enveloping algebra $U_\eta({\ggg}_{\bbk})$ associated to some $p$-character $\eta\in\chi+(\mathfrak{m}_{\bbk}^\bot)_{\bar0}$ based on the parity of $d_1$, respectively. It is notable that the $p$-character $\eta$ can only be guaranteed in $\chi+(\mathfrak{m}_{\bbk}^\bot)_{\bar0}$, with no further information apparently  contributing to $U_\chi(\ggg_{\bbk})$. Owing to Premet's treatment to finite $W$-algebras in \cite[Theorem 2.2]{P7}, one can translate the dimensional lower bounds for  $U_\eta({\ggg}_{\bbk})$ to the ones for $U_\chi({\ggg}_{\bbk})$ (see Lemma~\ref{etachi}). Taking such an approach, we will finally get at the desired result in Theorem~\ref{intromain-3}.

\begin{lemma}\label{etachi}
Let ${\ggg}_{\bbk}$ be a basic Lie superalgebra over ${\bbk}$. If the $\bbk$-algebra $U_\eta({\ggg}_{\bbk})$ affords a representation of dimension $p^{\frac{d_0}{2}}2^{\lfloor\frac{d_1}{2}\rfloor}$ for some $\eta\in\chi+(\mathfrak{m}_{\bbk}^\bot)_{\bar{0}}$, then $U_\chi({\ggg}_{\bbk})$ also admits a representation of dimension $p^{\frac{d_0}{2}}2^{\lfloor\frac{d_1}{2}\rfloor}$.
\end{lemma}

\begin{proof}First note that (1) $\lfloor\frac{d_1}{2}\rfloor=\frac{d_1}{2}$ when $d_1$ is even; and (2) $\lfloor\frac{d_1}{2}\rfloor=\frac{d_1+1}{2}$ when $d_1$ is odd. Since the proof is similar for both cases, we will just consider the situation when $d_1$ is odd.

Let $(G_{\bbk})_{\text{ev}}$ be the reductive algebraic group associated to the even part $({\ggg}_{\bbk})_{\bar{0}}$ of Lie superalgebra ${\ggg}_{\bbk}$.
For any $\xi\in({\ggg}_{\bbk})^*_{\bar0}$, it is well known that the construction of the ${\bbk}$-algebra $U_\xi({\ggg}_{\bbk})$ only depends on the orbit of $\xi$ under the coadjoint action of $(G_{\bbk})_{\text{ev}}$ up to an isomorphism. Therefore, if $\xi':=(\text{Ad}^*g)(\xi)$ for some $g\in(G_{\bbk})_{\text{ev}}$, then $U_\xi({\ggg}_{\bbk})\cong U_{\xi'}({\ggg}_{\bbk})$ as ${\bbk}$-algebras.

Let $\Xi$ denote the set of all $\xi\in({\ggg}_{\bbk})^*_{\bar{0}}$ for which the algebra $U_\xi({\ggg}_{\bbk})$ contains a two-sided ideal of codimension $p^{d_0}2^{d_1+1}$. It follows from \cite[Lemma 2.2]{Z1} that the set $\Xi$ is Zariski closed in $({\ggg}_{\bbk})^*_{\bar{0}}$. Moreover, the preceding remark shows that the set $\Xi$ is stable under the coadjoint action of $(G_{\bbk})_{\text{ev}}$.

(i) We claim that $\bar t\cdot\xi\in\Xi$ for all $\bar t\in{\bbk}^\times(={\bbk}\backslash\{0\})$. 

For any $\xi\in({\ggg}_{\bbk})^*_{\bar{0}}$,  we can regard $\xi\in{\ggg}_{\bbk}^*$ by letting $\xi(({\ggg}_{\bbk})_{\bar{1}})=0$. Now let $\xi=(\bar x,\cdot)$ for some $\bar x\in({\ggg}_{\bbk})_{\bar{0}}$. Let
$\bar x=\bar s+\bar n$ be the Jordan-Chevalley decomposition of $\bar x$ in the restricted Lie algebra $({\ggg}_{\bbk})_{\bar{0}}$ and put
$\xi_{\bar s}:=(\bar s,\cdot)$, $\xi_{\bar n}:=(\bar n,\cdot)$. Take a Cartan subalgebra $\mathfrak{h}_{\bbk}$ of ${\ggg}_{\bbk}$ which contains $\bar s$, and let ${\ggg}_{\bbk}^{\bar s}$ denote the centralizer of $\bar s$ in ${\ggg}_{\bbk}$. Let $\Phi$ be a root system of ${\ggg}_{\bbk}$ relative to $\mathfrak{h}_{\bbk}$. Then ${\ggg}_{\bbk}^{\bar s}:=\mathfrak{l}_{\bbk}=(\mathfrak{l}_{\bbk})_{\bar{0}}+(\mathfrak{l}_{\bbk})_{\bar{1}}$ also has a root space decomposition $\mathfrak{l}_{\bbk}=\mathfrak{h}_{\bbk}\oplus\bigoplus\limits_{\alpha\in\Phi(\mathfrak{l}_{\bbk})}({\ggg}_{\bbk})_\alpha$ with $\Phi(\mathfrak{l}_{\bbk}):=\{\alpha\in\Phi\mid\alpha(\bar s)=0\}$. From \cite[Proposition 5.1]{WZ} we know that there exists a system $\Delta$ of simple roots of ${\ggg}_{\bbk}$ such that $\Delta\cap\Phi(\mathfrak{l}_{\bbk})$ is a system of simple roots for $\Phi(\mathfrak{l}_{\bbk})$. In particular, $\mathfrak{l}_{\bbk}$ is always a direct sum of basic Lie superalgebras (note that a toral subalgebra of ${\ggg}_{\bbk}$ may also appear in the summand). Let $\mathfrak{b}_{\bbk}=\mathfrak{h}_{\bbk}\oplus\mathfrak{n}_{\bbk}$ be the Borel subalgebra associated to $\Delta$. Then we can define a parabolic subalgebra ${\ppp}_{\bbk}=\mathfrak{l}_{\bbk}+\mathfrak{b}_{\bbk}=\mathfrak{l}_{\bbk}\oplus\mathfrak{u}_{\bbk}$, where $\mathfrak{u}_{\bbk}$ denotes the nilradical of ${\ppp}_{\bbk}$. Recall \cite[\S5.1]{WZ} shows that $\xi(\mathfrak{u}_{\bbk})=0$, and $\xi|_{\mathfrak{l}_{\bbk}}=\xi_n|_{\mathfrak{l}_{\bbk}}$ is nilpotent. It is notable that all subalgebras here are naturally restricted subalgebras of $\ggg_\bbk$.

Since $\bar x=\bar s+\bar n$ is the Jordan-Chevalley decomposition of $\bar x$, it follows that $\bar t\bar x=\bar t\bar s+\bar t\bar n$ is the Jordan-Chevalley decomposition of $\bar t\bar x$ for $\bar t\in{\bbk}^\times$. Obviously, ${\ggg}_{\bbk}^{\bar t\bar s}=\mathfrak{l}_{\bbk}$.

It follows from \cite[Theorem 5.3]{WZ} that every irreducible $U_\xi({\ggg}_{\bbk})$-module is $U_\xi(\mathfrak{u}_{\bbk})$-projective. Since $\mathfrak{u}_{\bbk}$ is nilpotent in ${\ggg}_{\bbk}$ and $\xi|_{\mathfrak{u}_{\bbk}}=0$, it follows from \cite[Proposition 2.6]{WZ} that the ${\bbk}$-algebra $U_\xi(\mathfrak{u}_{\bbk})$ is local with trivial module as the unique irreducible module. Then every irreducible $U_\xi({\ggg}_{\bbk})$-module is $U_\xi(\mathfrak{u}_{\bbk})$-free, and the unique maximal ideal $N_{\mathfrak{u}_{\bbk}}$ of $U_\xi(\mathfrak{u}_{\bbk})$ is generated by the image of $\mathfrak{u}_{\bbk}$ in $U_\xi(\mathfrak{u}_{\bbk})$.
 We put $\bar1_\xi=1+N_{\mathfrak{u}_{\bbk}}$, the image of $1$ in ${\bbk}_\xi:=U_\xi(\uuu_\bbk)\slash N_{\uuu_\bbk}$. Consider the $\bbk$-algebra $U_\xi({\ppp}_{\bbk})$, which contains $U_\xi(\uuu_\bbk)$ as a subalgebra.
 Set $Q^0_{\ppp_{\bbk}}$ to be a $U_\xi(\ppp_{\bbk})$-module with the ground space $U_\xi({\ppp}_{\bbk})/U_\xi({\ppp}_{\bbk})N_{\mathfrak{u}_{\bbk}}$, then $Q^0_{\ppp_\bbk}
 =U_\xi({\ppp}_{\bbk})\cdot\bar1_\xi$.
For any $\bar q\in Q^0_{\ppp_\bbk}$ there is $\bar u\in U_\xi({\ppp}_{\bbk})$ such that $\bar q=\bar u\cdot \bar1_\xi$. Since $[{\ppp}_{{\bbk}},N_{\mathfrak{u}_{\bbk}}]\subseteq N_{\mathfrak{u}_{\bbk}}$, it follows from Jacobi identity that $[\bar n,\bar u]\in U_\xi({\ppp}_{\bbk})N_{\mathfrak{u}_{\bbk}}$ for any $\bar n\in N_{\mathfrak{u}_{\bbk}}$. Then
$$\bar n\cdot \bar q=([\bar n,\bar q]+(-1)^{|\bar n||\bar q|}\bar q\cdot \bar n)\cdot \bar1_\xi\subseteq U_\xi({\ppp}_{\bbk})N_{\mathfrak{u}_{\bbk}}$$ for any ${\bbz}_2$-homogeneous elements $\bar n\in N_{\mathfrak{u}_{\bbk}}$ and $\bar q\in Q_{\mathfrak{u}_{\bbk}}^0$. Thus
  we have the following algebra isomorphism
  $$U_\xi({\ppp}_{\bbk})/U_\xi({\ppp}_{\bbk})N_{\mathfrak{u}_{\bbk}}\cong U_\xi({\ppp}_{\bbk}/\mathfrak{u}_{\bbk}).$$
 Define $$\widetilde {Q}_\xi^\xi:=U_\xi({\ggg}_{\bbk})/U_\xi({\ggg}_{\bbk})N_{\mathfrak{u}_{\bbk}},$$
  which is clearly endowed with left $U_\xi({\ggg}_{\bbk})$-module structure. The algebra inclusion $U_\xi(\uuu_\bbk)\subseteq U_\xi(\ggg_\bbk)$ (resp. $U_\xi(\ppp_\bbk)\subseteq U_\xi(\ggg_\bbk)$) gives rise to the canonical imbedding of spaces $\bbk_\xi=\bbk\cdot \bar 1_\xi\hookrightarrow \widetilde {Q}_\xi^\xi$
  (resp. $Q^0_{\ppp_\bbk}\hookrightarrow \widetilde {Q}_\xi^\xi$).
 Those imbeddings clearly satisfy the property that
 for any $\bar q\in Q_{\mathfrak{u}_{\bbk}}^0$ there is a unique $h_{\bar q}\in\text{End}_{{\ggg}_{\bbk}}{\widetilde{Q}}_\xi^\xi$ such that $h_{\bar q}(\bar1_\xi)=\bar q$.
Put
\begin{align*}
&d'_0:=2\text{dim}\,(\mathfrak{u}_{\bbk})_{\bar{0}}=\dim(\ggg_\bbk)_{\bar{0}}-\dim
(\mathfrak{l}_{\bbk})_{\bar{0}},\cr
&d'_1:=2\text{dim}\,(\mathfrak{u}_{\bbk})_{\bar{1}}
=\text{dim}\,({\ggg}_{\bbk})_{\bar{1}}-\text{dim}\,(\mathfrak{l}_{\bbk})_{\bar{1}}.
 \end{align*}
 Since every irreducible $U_\xi({\ggg}_{\bbk})$-module is $U_\xi(\mathfrak{u}_{\bbk})$-free, by the same discussion as \cite[Theorem 4.4]{WZ} we can obtain a ${\bbk}$-algebra isomorphism:
\begin{equation}\label{isorr}
U_\xi({\ggg}_{\bbk})\cong \text{Mat}_{p^{\frac{d'_0}{2}}2^{\frac{d'_1}{2}}}((\text{End}_{{\ggg}_{\bbk}}
{\widetilde{Q}}_\xi^\xi)^{\text{op}}).
\end{equation}
Therefore,$$\text{dim}\,(\text{End}_{{\ggg}_{\bbk}}{\widetilde{Q}}_\xi^\xi)=p^{\text{dim}\,
({\ggg}_{\bbk})_{\bar{0}}-d'_0}2^{\text{dim}\,
({\ggg}_{\bbk})_{\bar{1}}-d'_1}=p^{\text{dim}\,
(\mathfrak{l}_{\bbk})_{\bar{0}}}2^{\text{dim}\,
(\mathfrak{l}_{\bbk})_{\bar{1}}}.$$Since $$\text{dim}\,U_\xi({\ppp}_{\bbk}/\mathfrak{u}_{\bbk})=p^{\text{dim}\,
({\ppp}_{\bbk})_{\bar{0}}-\text{dim}\,(\mathfrak{u}_{\bbk})_{\bar{0}}}2^{\text{dim}\,
({\ppp}_{\bbk})_{\bar{1}}-\text{dim}\,
(\mathfrak{u}_{\bbk})_{\bar{1}}}=p^{\text{dim}\,
(\mathfrak{l}_{\bbk})_{\bar{0}}}2^{\text{dim}\,
(\mathfrak{l}_{\bbk})_{\bar{1}}},$$
it follows that $\text{End}_{{\ggg}_{\bbk}}\widetilde{Q}_\xi^\xi=\{h_{\bar q}\mid\bar q\in Q_{\mathfrak{u}_{\bbk}}^0\}$. Define the mapping
\[\begin{array}{lcll}
\tau:&\text{End}_{{\ggg}_{\bbk}}{\widetilde{Q}}_{\xi}^\xi&\rightarrow&U_\xi({\ppp}_{\bbk}/\mathfrak{u}_{\bbk})^{\text{op}}\\ &\theta&\mapsto&\theta(\bar1_\xi)
\end{array}.
\]
It is obvious that $\tau$ is a homomorphism of ${\bbk}$-algebras. As both $\bbk$-algebras have the same dimension (as vector spaces), one can deduce that $\tau$ is an isomorphism. Taking the opposite algebras for both sides, we have an algebra isomorphism
\begin{equation}\label{iso2}
(\text{End}_{{\ggg}_{\bbk}}\widetilde{Q}_\xi^\xi)^{\text{op}}\cong U_\xi({\ppp}_{\bbk}/\mathfrak{u}_{\bbk})\cong U_\xi(\mathfrak{l}_{\bbk}).
\end{equation}

For the $p$-character $\bar t\xi$, repeating the same arguments as above for $\xi$, we can obtain that
\begin{equation}\label{iso3}
(\text{End}_{{\ggg}_{\bbk}}\widetilde{Q}_{\bar t\xi}^{\bar t\xi})^{\text{op}}\cong U_{\bar t\xi}(\mathfrak{l}_{\bbk}).
\end{equation}
And also we have 
an algebra isomorphism
\begin{equation}\label{iso4}
U_{\bar t\xi}({\ggg}_{\bbk})\cong \text{Mat}_{p^{\frac{d'_0}{2}}2^{\frac{d'_1}{2}}}((\text{End}_{{\ggg}_{\bbk}}\widetilde{Q}_{\bar t\xi}^{\bar t\xi})^{\text{op}}).
\end{equation}

Recall that $\mathfrak{l}_{\bbk}$ is a direct sum of basic Lie superalgebras. Set $\mathfrak{l}_{\bbk}=({\ggg}_{\bbk})_1\oplus\cdots\oplus({\ggg}_{\bbk})_r\oplus\mathfrak{t}'_{\bbk}$, where $({\ggg}_{\bbk})_i$ is a basic Lie superalgebra for each $1\leqslant  i\leqslant  r$, and $\mathfrak{t}'_{\bbk}$ is a toral subalgebra of ${\ggg}_{\bbk}$. For each $1\leqslant  i\leqslant  r$, let $(G_{\bbk})_i$ denote the algebraic supergroup associated to $({\ggg}_{\bbk})_i$. It is well known that the even part of $(G_{\bbk})_i$ ($1\leqslant  i\leqslant  r$) is a reductive algebraic group, and we denote it by $((G_{\bbk})_i)_{\text{ev}}$. Since $\xi|_{\mathfrak{l}_{\bbk}}=\xi_n|_{\mathfrak{l}_{\bbk}}$ is nilpotent, it follows from \cite[Lemma 2.10]{J3} that ${\bbk}^\times\cdot\xi|_{({\ggg}_{\bbk})_i}\subseteq(\text{Ad}^*((G_{\bbk})_i)_{\text{ev}})\xi|_{({\ggg}_{\bbk})_i}$.
For each $1\leqslant  i\leqslant  r$, since the reduced enveloping algebra $U_{\xi|_{({\ggg}_{\bbk})_i}}(({\ggg}_{\bbk})_i)$ depends only on the orbit of $\xi|_{({\ggg}_{\bbk})_i}$ under the coadjoint action of $((G_{\bbk})_i)_{\text{ev}}$, then $U_{\xi|_{({\ggg}_{\bbk})_i}}(({\ggg}_{\bbk})_i)\cong U_{\bar t\xi|_{({\ggg}_{\bbk})_i}}(({\ggg}_{\bbk})_i)$ as ${\bbk}$-algebras. With $i$ being arbitrary, we have
$\bigotimes\limits_{i=1}^r U_{\xi|_{({\ggg}_{\bbk})_i}}(({\ggg}_{\bbk})_i)\cong \bigotimes\limits_{i=1}^r U_{\bar t\xi|_{({\ggg}_{\bbk})_i}}(({\ggg}_{\bbk})_i)$.
As $\mathfrak{t}'_{\bbk}$ is a toral subalgebra of ${\ggg}_{\bbk}$, the reduced enveloping algebra $U_\psi(\mathfrak{t}'_{\bbk})$ is commutative and semisimple for every $\psi\in(\mathfrak{t}'_{\bbk})^*$ (Indeed, $\mathfrak{t}'_{\bbk}$ has a ${\bbk}$-basis $t_1,\cdots,t_d$ with $t_i^{[p]}=t_i$ for $1\leqslant i\leqslant d$. Therefore, $U_\psi(\mathfrak{t}'_{\bbk})\cong A_1\otimes\cdots\otimes A_d$ where $A_i\cong{\bbk}[X]/(X^p-X-\psi(t_i)^p)$ is a $p$-dimensional commutative semisimple ${\bbk}$-algebra). From this it is immediate that $U_{\xi|_{\mathfrak{t}'_{\bbk}}}(\mathfrak{t}'_{\bbk})\cong U_{\bar t\xi|_{\mathfrak{t}'_{\bbk}}}(\mathfrak{t}'_{\bbk})$ as $\bbk$-algebras. Since $\mathfrak{l}_{\bbk}=\bigoplus\limits_{i=1}^r({\ggg}_{\bbk})_i\oplus\mathfrak{t}'_{\bbk}$, we have $U_{\xi}(\bigoplus\limits_{i=1}^r({\ggg}_{\bbk})_i\oplus\mathfrak{t}'_{\bbk})\cong \bigotimes\limits_{i=1}^r U_{\xi|_{({\ggg}_{\bbk})_i}}(({\ggg}_{\bbk})_i)\otimes U_{\xi|_{\mathfrak{t}'_{\bbk}}}(\mathfrak{t}'_{\bbk})$ and $U_{\bar t\xi}(\bigoplus\limits_{i=1}^r({\ggg}_{\bbk})_i\oplus\mathfrak{t}'_{\bbk})\cong \bigotimes\limits_{i=1}^r U_{\bar t\xi|_{({\ggg}_{\bbk})_i}}(({\ggg}_{\bbk})_i)\otimes U_{\bar t\xi|_{\mathfrak{t}'_{\bbk}}}(\mathfrak{t}'_{\bbk})$. Therefore, $U_{\xi}(\mathfrak{l}_{\bbk})\cong U_{\bar t\xi}(\mathfrak{l}_{\bbk})$ as $\bbk$-algebras. It follows from \eqref{isorr}---\eqref{iso4} that there is a $\bbk$-algebras isomorphism\begin{equation}\label{tiso}
U_{\xi}({\ggg}_{\bbk})\cong U_{\bar t\xi}({\ggg}_{\bbk})\end{equation} for all $\bar t\in{\bbk}^\times$.
Now Claim (i) is an immediate consequence of \eqref{tiso}, i.e. if $\xi\in\Xi$, then $\bar t\cdot\xi\in\Xi$ for all $\bar t\in{\bbk}^\times$. Moreover,  combining with the arguments  prior to paragraph (i), we know that the affine variety  $\Xi$ is conical.

(ii) We claim that $\chi\in\Xi$.

 Recall the assumption of the lemma shows that $U_\eta({\ggg}_{\bbk})$ has an irreducible module of dimension $p^{\frac{d_0}{2}}2^{\frac{d_1+1}{2}}$, thus $\eta\in\Xi$. As $\eta\in\chi+(\mathfrak{m}_{\bbk}^\perp)_{\bar{0}}$, we can write $\eta=(e+\bar y,\cdot)$ for some $\bar y=\sum_{i\leqslant 1}\bar y_i$ with $\bar y_i\in{\ggg}_{\bbk}(i)_{\bar{0}}$, $i\leqslant 1$. There is a cocharacter $\lambda:{\bbk}^\times\longrightarrow(G_{\bbk})_{\text{ev}}$ such that $(\text{Ad}\lambda(\bar t))\bar x=\bar t^j\bar x$ for all $\bar x\in{\ggg}_{\bbk}(j)\,(j\in{\bbz})$ and $\bar t\in{\bbk}^\times$. For $i\leqslant 1$, set $\eta_i=(\bar y_i,\cdot)$, then $\eta=\chi+\sum_{i\leqslant 1}\eta_i$. For any even element $\bar y'\in{\ggg}_{\bbk}(j)_{\bar{0}}$, by the coadjoint action of $\lambda(\bar t)$ on $(\ggg_\bbk)_{\bar0}^*$ one has
 \begin{align*}
 (\text{Ad}^*(\lambda(\bar t))(\eta))(\bar y')&=\eta(\text{Ad}\lambda(\bar t)^{-1}\bar  y')\cr
 &=\bar t^{-j}(e+\sum_{i\leqslant 1}\bar y_i,\bar y')\cr
 &=\Bigg\{\begin{array}{ll}\bar t^{2}(e,\bar y')&(j=-2)\\
\bar t^{-j}\sum_{i\leqslant 1}\delta_{i,-j}(\bar y_i,\bar y')&(j\neq-2)\end{array}.
\end{align*}
  Note that $(\bar t^{2}\chi+\sum_{i\leqslant 1}\bar t^{i}\eta_i)(\bar y')=\bar t^{2}(e,\bar y')+\sum_{i\leqslant 1}\bar t^{i}(\bar y_i,\bar y')$. Then we have
$(\text{Ad}^*\lambda(\bar t))\eta=\bar t^{2}\chi+\sum_{i\leqslant 1}\bar t^{i}\eta_i$, and
$(\text{Ad}^*\lambda(\bar t))^{-1}\eta=\bar t^{-2}\chi+\sum_{i\leqslant 1}\bar t^{-i}\eta_i$. As $\Xi$ is conical and $\text{Ad}^*(G_{\bbk})_{\text{ev}}$-invariant by step (i), this implies that $$\bar t^{2}\cdot(\text{Ad}^*\lambda(\bar t))^{-1}\eta=\chi+\sum_{i\leqslant 1}\bar t^{2-i}(\bar y_i,\bar y')\in\Xi$$for all $\bar t\in{\bbk}^\times$. Since $\Xi$ is Zariski closed, this yields $\chi\in\Xi$. Then claim (ii) is proved.

From all above we know that  $U_\chi({\ggg}_{\bbk})$ admits a two-sided ideal $I$ of codimension $p^{d_0}2^{d_1+1}$, and all irreducible modules of the factor algebra $U_\chi({\ggg}_{\bbk})/I$ have dimensions $\leqslant  p^{\frac{d_0}{2}}2^{\frac{d_1+1}{2}}$.  Combining with Proposition \ref{wzd2}, one can conclude that the $\bbk$-algebra $U_\chi({\ggg}_{\bbk})$ really has an irreducible module of dimension $p^{\frac{d_0}{2}}2^{\frac{d_1+1}{2}}$. We complete the proof.
\end{proof}

\subsubsection{Proof of Theorem \ref{intromain-3}}
Let ${\ggg}$ be a basic Lie superalgebra over ${\bbc}$ and ${\ggg}_{\bbk}$ the corresponding Lie superalgebra over positive characteristic field ${\bbk}$. Let $\chi\in({\ggg}_{\bbk})^*_{\bar0}$ be a nilpotent $p$-character of ${\ggg}_{\bbk}$ such that $\chi(\bar y)=(e,\bar y)$ for any $\bar y\in{\ggg}_{\bbk}$. Under the assumption that  the finite $W$-superalgebra $U({\ggg},e)$ over ${\bbc}$ affords a one-dimensional  (resp. two-dimensional) representation when $d_1$ is even (resp. odd), we want to prove that the reduced enveloping algebra $U_\chi(\ggg_\bbk)$ with nilpotent $p$-character $\chi$ possesses an irreducible module whose dimension is exactly the lower bound predicted by the super Kac-Weisfeiler property in Proposition~\ref{wzd2}.
For the case of $d_1$ being even, the conclusion follows from Lemma~\ref{mindim1} and Lemma~\ref{etachi}. The odd case can be done by Lemma~\ref{cdim2} and Lemma~\ref{etachi}. Then Theorem \ref{intromain-3} follows. We complete the proof.

As a corollary of Theorem~\ref{intromain-3}, we have the following consequence on the  ``small representations'' of reduced $W$-superalgebra $U_\chi({\ggg}_{\bbk},e)$.

\begin{corollary}\label{1122}
Let ${\ggg}$ be a basic Lie superalgebra. When $d_1$ is even (resp. odd), if the finite $W$-superalgebra $U({\ggg},e)$ over ${\bbc}$ affords a one-dimensional  (resp. two-dimensional) representation, then for $p\gg0$, the reduced $W$-superalgebra $U_\chi({\ggg}_{\bbk},e)$ over ${\bbk}=\overline{\mathbb{F}}_p$ also admits one-dimensional  (resp. two-dimensional) representations.
\end{corollary}

\begin{proof}
For any $U_\chi({\ggg}_{\bbk})$-module $M$, it follows from Theorem~\ref{reducedfunctors} that $M^{\mathfrak{m}_{\bbk}}$ is a $U_\chi({\ggg}_{\bbk},e)$-module. For the case $\eta=\chi$, \eqref{dimMn} shows that
$$\text{dim}\,M^{\mathfrak{m}_{\bbk}}=\frac{\text{dim}\,M}{\text{dim}\,U_{\chi}(\mathfrak{m}_{\bbk})}=\frac{\text{dim}\,M}{p^{\frac{d_0}{2}}2^{\lceil\frac{d_1}{2}\rceil}}.$$ Then the desired result follows from Theorem~\ref{intromain-3}.
\end{proof}

\subsection{The case of a direct sum of basic Lie superalgebras with nilpotent $p$-characters}\label{5.3}
In this section we will consider the lower bounds of the Super KW Property for a direct sum of basic Lie superalgebras with nilpotent $p$-characters.

\subsubsection{}
First we make digression to  recall some known facts on finite dimensional superalgebras \cite[\S12]{KL}.

Let $\mathbb{F}$ be an algebraically closed field.
Now we will recall some basics on simple superalgebras over $\mathbb{F}$ (cf. \cite[\S12.1]{KL}). Let $V$ be a superspace with $\text{\underline{dim}}\,V=(m,n)$, then $\mathcal{M}(V):=\text{End}_\mathbb{F}(V)$ is a superalgebra with $\text{\underline{dim}}\,\mathcal{M}(V)=(m^2+n^2,2mn)$. The algebra $\mathcal{M}(V)$ is defined uniquely up to an isomorphism by the superdimension $(m,n)$ of $V$. So we can speak of the superalgebra $\mathcal{M}_{m,n}$. We have an isomorphism of superalgebras
\begin{equation}\label{MM}
\mathcal{M}_{m,n}\otimes\mathcal{M}_{k,l}\cong\mathcal{M}_{mk+nl,ml+nk}.
\end{equation}
Moreover,  \cite[Example 12.1.1]{KL} shows that $\mathcal{M}_{m,n}$ is a simple superalgebra.

Let $V$ be a superspace with $\text{\underline{dim}}\,V=(n,n)$ and $J$ be a degree $\bar{1}$ involution in $\text{End}_\mathbb{F}(V)$. Consider the superalgebra $Q(V,J):=\{f\in\text{End}_\mathbb{F}(V)\,|\,fJ=(-1)^{|f|}Jf\}$. Note that all degree $\bar{1}$ involutions in $\text{End}_\mathbb{F}(V)$ are conjugate to each other by an invertible element in $\text{End}_\mathbb{F}(V)_{\bar 0}$. Hence another choice of $J$ will yield an isomorphism superalgebra. So we can speak of the superalgebra $Q(V)$, defined up to an isomorphism. Pick a basis $\{v_1,\cdots, v_n\}$ of $V_{\bar0}$, and set $v'_i=J(v_i)$ for $1\leqslant i\leqslant n$. Then $\{v'_1,\cdots, v'_n\}$ is a basis of $V_{\bar1}$. With respect to the basis $\{v_1,\cdots, v_n,v'_1,\cdots, v'_n\}$, the elements of $Q(V,J)$ have matrices of the form
\begin{equation}\label{QAB}
\left(
\begin{array}{cl}
A&B\\
-B&A
\end{array}
\right),
\end{equation}
where $A$ and $B$ are arbitrary $n\times n$ matrices, with $B=0$ for even endomorphisms and $A=0$ for odd ones. In particular, $\text{\underline{dim}}\,Q(V)=(n^2,n^2)$. The superalgebra $Q(V,J)$ can be identified with the superalgebra $Q_n$ of all matrices of the form \eqref{QAB}. Moreover, \cite[(12.6), (12.7)]{KL} show that
\begin{equation}\label{MQ}
\mathcal{M}_{m,n}\otimes Q_k\cong Q_{(m+n)k}
\end{equation}and
\begin{equation}\label{QQ}
Q_m\otimes Q_n\cong \mathcal{M}_{mn,mn}
\end{equation}
as $\mathbb{F}$-algebras.
Moreover, \cite[Example 12.1.2]{KL} shows that   $Q_n$ is a simple superalgebra.

Given a finite dimensional superalgebra $A$ over $\mathbb{F}$, define the  parity change functor on  $A\text{-}mod$ (the $A$-module category)
$$\Upsilon:~A\text{-}mod\longrightarrow A\text{-}mod.$$
For an object $V$, $\Upsilon V$ is the same underlying vector space but with the opposite ${\bbz}_2$-grading. The new action
of a ${\bbz}_2$-homogeneous element $a\in A$ on $v\in V$ is defined in terms of the old action by $a\cdot v:=(-1)^{|a|}av$.
Given left modules $V$ and $W$ over $\mathbb{F}$-superalgebras $A$ and $B$ respectively, the (outer) tensor product $V\boxtimes W$ is the space $V\otimes W$ considered as an $A\otimes B$-module via
$$(a\otimes b)(v\otimes w)=(-1)^{|b||v|}av\otimes bw\qquad(a\in A,\,b\in B,\,v\in V,\,w\in W).$$

For irreducible representations of the $\mathbb{F}$-algebra $A\otimes B$, the following result was obtained by Kleshchev in \cite[Lemma 12.2.13]{KL}:

\begin{lemma}(\cite{KL})\label{AB}
Let $V$ be an irreducible $A$-module and $W$ be an irreducible $B$-module.

(i) If both $V$ and $W$ are of type $M$, then $V\boxtimes W$ is an irreducible $A\otimes B$-module of type $M$.

(ii) If one of $V$ or $W$ is of type $M$ and the other is of type $Q$, then $V\boxtimes W$ is an irreducible $A\otimes B$-module of type $Q$.

(iii) If both $V$ and $W$ are of type $Q$, then $V\boxtimes W\cong (V\circledast W)\oplus\Upsilon(V\circledast W)$ for an irreducible $A\otimes B$-module $V\circledast W$ of type $M$.

Moreover, all irreducible $A\otimes B$-modules arise as constituents of $V\boxtimes W$ for some choice of irreducibles $V$, $W$.

\end{lemma}

\subsubsection{}\label{5.3.2}
 Now we return to the representations of reduced enveloping algebras for a direct sum of basic Lie superalgebras with nilpotent $p$-characters. In \cite[Remark 4.6]{WZ}, Wang-Zhao showed that the statement in Proposition~\ref{wzd2} still holds  for the case when $\mathfrak{l}_{\bbk}$ is a direct sum of basic Lie superalgebras with nilpotent $p$-characters.

In fact, their result can be somewhat strengthened. Let $\mathfrak{l}_{\bbk}=\bigoplus\limits_{i=1}^r(\mathfrak{l}_{\bbk})_i$ be a direct sum of basic Lie superalgebras over ${\bbk}=\overline{\mathbb{F}}_p$, where $(\mathfrak{l}_{\bbk})_i$ is a basic Lie superalgebra for each $1\leqslant i\leqslant r$, and the characteristic $p$ of the field $\bbk$ satisfies the restriction imposed in \S\ref{prel}(Table 1). Let $\chi=\chi_1+\cdots+\chi_r$ be the decomposition of nilpotent $p$-character $\chi$ in $\mathfrak{l}_{\bbk}^*$ with $\chi_i\in(\mathfrak{l}_{\bbk})_i^*$ (each $\chi_i$ can be viewed in $\mathfrak{l}_{\bbk}^*$ by letting $\chi_i(\bar y)=0$ for all $\bar y\in\bigoplus\limits_{j\neq i}(\mathfrak{l}_{\bbk})_j$) for $1\leqslant i\leqslant r$. Set $\bar e=\bar e_1+\cdots+\bar e_r$ to be the corresponding decomposition of $\bar e\in (\mathfrak{l}_{\bbk})_{\bar0}$ with respect to the non-degenerate bilinear form $(\cdot,\cdot)$ on $\mathfrak{l}_{\bbk}$ such that $\chi_i(\cdot)=(\bar e_i,\cdot)$  for $1\leqslant i\leqslant r$.
Define
\begin{equation}\label{numsum}
\begin{array}{rllrlll}
d'_0&:=&\text{dim}\,(\mathfrak{l}_{\bbk})_{\bar0}-\text{dim}\,(\mathfrak{l}_{\bbk}^{\bar e})_{\bar0},&
d'_1&:=&\text{dim}\,(\mathfrak{l}_{\bbk})_{\bar1}-\text{dim}\,(\mathfrak{l}_{\bbk}^{\bar e})_{\bar1},\\
(d_0)_i&:=&\text{dim}\,((\mathfrak{l}_{\bbk})_i)_{\bar0}-\text{dim}\,((\mathfrak{l}_{\bbk})^{\bar e_i}_i)_{\bar0},&
(d_1)_i&:=&\text{dim}\,((\mathfrak{l}_{\bbk})_i)_{\bar1}-\text{dim}\,((\mathfrak{l}_{\bbk})^{\bar e_i}_i)_{\bar1},
\end{array}
\end{equation}where $\mathfrak{l}_{\bbk}^{\bar e}$ denotes the centralizer of $\bar e$ in $\mathfrak{l}_{\bbk}$, and $(\mathfrak{l}_{\bbk})^{\bar e_i}_i$ the centralizer of $\bar e_i$ in $(\mathfrak{l}_{\bbk})_i$ for $i\in\{1,\cdots,r\}$. It is obvious that $d'_0=\sum\limits_{i=1}^r(d_0)_i$ and $d'_1=\sum\limits_{i=1}^r(d_1)_i$. Rearrange the summands of $\mathfrak{l}_{\bbk}=\bigoplus\limits_{i=1}^r(\mathfrak{l}_{\bbk})_i$ such that each $(d_1)_i$ is odd for $1\leqslant i\leqslant l$ (if occurs) and each $(d_1)_i$ is even for $l+1\leqslant i\leqslant r$ (if occurs). In particular,  $d'_1$ and $l$ have the same parity.

Note that all arguments in the preceding sections remain valid for a direct sum of basic Lie superalgebras. Let $\mathfrak{m}_{\bbk}$ and $\mathfrak{m}'_{\bbk}$ be the subalgebras of $\mathfrak{l}_{\bbk}$ as defined in \S\ref{2.3.2}. Write $\mathfrak{m}_{\bbk}=\bigoplus\limits_{i=1}^r(\mathfrak{m}_{\bbk})_i$ and $\mathfrak{m}'_{\bbk}=\bigoplus\limits_{i=1}^r(\mathfrak{m}'_{\bbk})_i$ as the decomposition of $\mathfrak{m}_{\bbk}$ and $\mathfrak{m}'_{\bbk}$ in $\mathfrak{l}_{\bbk}$ respectively, where $(\mathfrak{m}_{\bbk})_i,\,(\mathfrak{m}'_{\bbk})_i\in(\mathfrak{l}_{\bbk})_i$ for $1\leqslant i\leqslant r$. As $\mathfrak{m}_{\bbk}$ is $p$-nilpotent and the linear function $\chi$ vanishes on the $p$-closure of $[\mathfrak{m}_{\bbk},\mathfrak{m}_{\bbk}]$, it follows from \cite[Proposition 2.6]{WZ} that $U_{\chi}(\mathfrak{m}_{\bbk})$ has a unique irreducible module and $U_{\chi}(\mathfrak{m}_{\bbk})/N_{\mathfrak{m}_{\bbk}}\cong {\bbk}$, where $N_{\mathfrak{m}_{\bbk}}$ is the Jacobson radical of $U_{\chi}(\mathfrak{m}_{\bbk})$ generated by all the elements $\bar x-\chi(\bar x)$ with $\bar x\in\mathfrak{m}_{\bbk}$.

Recall \cite[Proposition 4.1]{WZ} shows that every
${\bbk}$-algebra $U_{\chi_i}((\mathfrak{m}_{\bbk})_i)$ ($1\leqslant i\leqslant r$) has a unique irreducible module (there is a minor error in \cite[\S4.1]{WZ}; this irreducible module is not necessary a trivial module) which is one-dimensional  and of type $M$. Let $N_{(\mathfrak{m}_{\bbk})_i}$ denote the Jacobson radical of $U_{\chi_i}((\mathfrak{m}_{\bbk})_i)$ for $1\leqslant i\leqslant r$ (which is the ideal of $U_{\chi_i}((\mathfrak{m}_{\bbk})_i)$ generated by all the elements $\bar x-\chi(\bar x)$ with $\bar x\in(\mathfrak{m}_{\bbk})_i$). Then $U_{\chi_i}((\mathfrak{m}_{\bbk})_i)/N_{(\mathfrak{m}_{\bbk})_i}={\bbk}$ for $1\leqslant i\leqslant r$.

For the case when $(d_1)_i$ is odd (i.e. $1\leqslant i\leqslant l$), the ${\bbk}$-algebra
$U_{\chi_i}((\mathfrak{m}'_{\bbk})_i)$ also has a unique irreducible module; it is isomorphic to $V_i=U_{\chi_i}((\mathfrak{m}'_{\bbk})_i)\otimes_{U_{\chi_i}((\mathfrak{m}_{\bbk})_i)}\bar1_{\chi_i}$, which is two-dimensional  and of type $Q$. Let $N_{(\mathfrak{m}'_{\bbk})_i}$ denote the Jacobson radical of $U_{\chi_i}((\mathfrak{m}'_{\bbk})_i)$ (which is the ideal of $U_{\chi_i}((\mathfrak{m}'_{\bbk})_i)$ generated by all the elements $\bar x-\chi(\bar x)$ with $\bar x\in(\mathfrak{m}_{\bbk})_i$). Then $U_{\chi_i}((\mathfrak{m}'_{\bbk})_i)/N_{(\mathfrak{m}'_{\bbk})_i}$ is isomorphic to the simple superalgebra $Q_1$ for $1\leqslant i\leqslant l$. For the case when $(d_1)_i$ is even (i.e. $l+1\leqslant i\leqslant r$), one can also define $N_{(\mathfrak{m}'_{\bbk})_i}$ as the Jacobson radical of $U_{\chi_i}((\mathfrak{m}'_{\bbk})_i)$. However, we have $U_{\chi_i}((\mathfrak{m}'_{\bbk})_i)=U_{\chi_i}((\mathfrak{m}_{\bbk})_i)$ and $N_{(\mathfrak{m}'_{\bbk})_i}=N_{(\mathfrak{m}_{\bbk})_i}$ for $(\mathfrak{m}'_{\bbk})_i=(\mathfrak{m}_{\bbk})_i$ by construction, thus $U_{\chi_i}((\mathfrak{m}'_{\bbk})_i)/N_{(\mathfrak{m}'_{\bbk})_i}={\bbk}$.

Since $\mathfrak{l}_{\bbk}=\bigoplus\limits_{i=1}^r(\mathfrak{l}_{\bbk})_i$, it is easy to verify that
\begin{equation}\label{NN'}
U_\chi(\mathfrak{m}_{\bbk})\cong\bigotimes\limits_{i=1}^rU_{\chi_i}((\mathfrak{m}_{\bbk})_i),\quad U_\chi(\mathfrak{m}'_{\bbk})\cong\bigotimes\limits_{i=1}^rU_{\chi_i}((\mathfrak{m}'_{\bbk})_i)
\end{equation}as ${\bbk}$-algebras, respectively. For a $U_{\chi}(\mathfrak{l}_{\bbk})$-module $M$ set $$M^{\mathfrak{m}_{\bbk}}=\{v\in M\mid(\bar x-\chi(\bar x)).v=0~\text{for all}~\bar x\in\mathfrak{m}_{\bbk}\}.$$
As the finite dimensional restricted Lie superalgebra $\mathfrak{l}_{\bbk}$ is a  direct sum of basic Lie superalgebras, an analogous discussion of  \cite[Proposition 4.2]{WZ} shows that every $U_{\chi}(\mathfrak{l}_{\bbk})$-module $M$ is $U_{\chi}(\mathfrak{m}_{\bbk})$-free and $M\cong U_{\chi}(\mathfrak{m}_{\bbk})^*\otimes_{{\bbk}}M^{\mathfrak{m}_{\bbk}}$ as $U_{\chi}(\mathfrak{m}_{\bbk})$-modules. 

Let $N_{\mathfrak{m}'_{\bbk}}$ denote the ideal of $U_\chi(\mathfrak{m}'_{\bbk})$ generated by all the elements $\bar x-\chi(\bar x)$ with $\bar x\in\mathfrak{m}_{\bbk}$, then $M^{\mathfrak{m}_{\bbk}}$ is a $U_\chi(\mathfrak{m}'_{\bbk})/N_{\mathfrak{m}'_{\bbk}}$-module.  Since $\mathfrak{m}'_{\bbk}=\bigoplus\limits_{i=1}^r(\mathfrak{m}'_{\bbk})_i$, one can conclude from \eqref{NN'} and the remark prior to it that
\begin{equation*}
\begin{array}{rcccl}
U_\chi(\mathfrak{m}'_{\bbk})/N_{\mathfrak{m}'_{\bbk}}&\cong& U_\chi(\mathfrak{m}'_{\bbk})\otimes_{U_\chi(\mathfrak{m}_{\bbk})}\bar1_\chi
&\cong& U_\chi(\bigoplus\limits_{i=1}^r(\mathfrak{m}'_{\bbk})_i)\otimes_{U_\chi(\bigoplus\limits_{i=1}^r
(\mathfrak{m}_{\bbk})_i)}\bar1_\chi\\
&\cong&\bigotimes\limits_{i=1}^r(U_{\chi_i}((\mathfrak{m}'_{\bbk})_i)
\otimes_{U_{\chi_i}((\mathfrak{m}_{\bbk})_i)}\bar1_{\chi_i})
&\cong&\bigotimes\limits_{i=1}^rU_{\chi_i}((\mathfrak{m}'_{\bbk})_i)/N_{(\mathfrak{m}'_{\bbk})_i}\\
&\cong&\overbrace{Q_1\otimes\cdots\otimes Q_1}^{l}\otimes\overbrace{{\bbk}\otimes\cdots\otimes{\bbk}}^{r-l}
&\cong&\overbrace{Q_1\otimes\cdots\otimes Q_1}^{l}
\end{array}
\end{equation*}as ${\bbk}$-algebras.

Now we will introduce the refined super Kac-Weisfeiler property for a direct sum of basic Lie superalgebras with nilpotent $p$-characters.

\begin{prop}\label{sumdivisible}
Let $\mathfrak{l}_{\bbk}$ be a direct sum of basic Lie superalgebras over ${\bbk}=\overline{\mathbb{F}}_p$, and $\chi\in(\mathfrak{l}_\bbk)_\bz^*$ be a nilpotent $p$-character. Retain the notations as \eqref{numsum}, and assume that the prime $p$ satisfies the restriction imposed in \S\ref{prel}(Table 1). Then the dimension of every $U_\chi(\mathfrak{l}_{\bbk})$-module $M$ is divisible by $p^{\frac{d'_0}{2}}2^{\frac{d'_1+l}{2}}$ (appoint  $l=0$ if the $(d_1)_i$'s are all even for $1\leqslant i\leqslant r$).
\end{prop}

\begin{proof}
For each $U_\chi(\mathfrak{l}_{\bbk})$-module $M$, the arguments preceding the proposition show that the $U_{\chi}(\mathfrak{m}_{\bbk})$-module
\begin{equation}\label{MMm}
M\cong U_{\chi}(\mathfrak{m}_{\bbk})^*\otimes_{{\bbk}}M^{\mathfrak{m}_{\bbk}}\end{equation} is free. Let us first investigate the dimension of
$M^{\mathfrak{m}_{\bbk}}$.
Recall that $M^{\mathfrak{m}_{\bbk}}$ is a module over the superalgebra $U_\chi(\mathfrak{m}'_{\bbk})/N_{\mathfrak{m}'_{\bbk}}\cong\overbrace{Q_1\otimes\cdots\otimes Q_1}^{l}$. Based on the parity of $l$, for each case we will consider separately.

(i) When $l$ is odd, \eqref{MQ} and \eqref{QQ} imply that
$$U_\chi(\mathfrak{m}'_{\bbk})/N_{\mathfrak{m}'_{\bbk}}\cong\overbrace{Q_1\otimes\cdots\otimes Q_1}^{l}\cong Q_{2^{l-1\over2}}.$$
Since $Q_{2^{\frac{l-1}{2}}}$ is a simple superalgebra whose unique irreducible module is $2\cdot2^{\frac{l-1}{2}}=2^{\frac{l+1}{2}}$-dimensional, it follows from Wedderburn's Theorem \cite[Theorem 12.2.9]{KL} that every $Q_{2^{\frac{l-1}{2}}}$-module has dimension divisible by $2^{\frac{l+1}{2}}$. In particular, the dimension of $M^{\mathfrak{m}_{\bbk}}$ is divisible by $2^{\frac{l+1}{2}}$. By the same discussion as \cite[Theorem 4.3]{WZ} we can conclude that $\text{\underline{dim}}\,\mathfrak{m}_{\bbk}=(\frac{d'_0}{2},\frac{d'_1-1}{2})$, then dim\,$U_\chi(\mathfrak{m}_{\bbk})=p^{\frac{d'_0}{2}}2^{\frac{d'_1-1}{2}}$. Together with \eqref{MMm} this implies that each $U_\chi(\mathfrak{l}_{\bbk})$-module $M$ has dimension divisible by $p^{\frac{d'_0}{2}}2^{\frac{d'_1-1}{2}}\cdot2^{\frac{l+1}{2}}=p^{\frac{d'_0}{2}}2^{\frac{d'_1+l}{2}}$.

(ii) When $l$ is even, it follows from \eqref{MQ} and \eqref{QQ} that
$$U_\chi(\mathfrak{m}'_{\bbk})/N_{\mathfrak{m}'_{\bbk}}\cong\overbrace{Q_1\otimes\cdots\otimes Q_1}^{l}\cong
\mathcal{M}_{2^{\frac{l}{2}-1},2^{\frac{l}{2}-1}}.$$
Since $\mathcal{M}_{2^{\frac{l}{2}-1},2^{\frac{l}{2}-1}}$ is a simple superalgebra whose unique irreducible module is $2^{\frac{l}{2}}$-dimensional, it follows from Wedderburn's Theorem \cite[Theorem 12.2.9]{KL} that every $\mathcal{M}_{2^{\frac{l}{2}-1},2^{\frac{l}{2}-1}}$-module has dimension divisible by $2^{\frac{l}{2}}$. In particular, the dimension of $M^{\mathfrak{m}_{\bbk}}$ is divisible by $2^{\frac{l}{2}}$. The same discussion as \cite[Theorem 4.3]{WZ} shows that $\text{\underline{dim}}\,\mathfrak{m}_{\bbk}=(\frac{d'_0}{2},\frac{d'_1}{2})$, then dim\,$U_\chi(\mathfrak{m}_{\bbk})=p^{\frac{d'_0}{2}}2^{\frac{d'_1}{2}}$. Together with \eqref{MMm} this implies that each $U_\chi(\mathfrak{l}_{\bbk})$-module $M$ has dimension divisible by $p^{\frac{d'_0}{2}}2^{\frac{d'_1}{2}}\cdot2^{\frac{l}{2}}=p^{\frac{d'_0}{2}}2^{\frac{d'_1+l}{2}}$.

All the discussions in (i) and (ii) complete the proof.

\end{proof}

\begin{rem}\label{refine}
 Recall that $d'_1$ and $l$ have the same parity by the preceding discussion. Thus Proposition~\ref{sumdivisible} coincides with the consequence obtained by Wang-Zhao in \cite[Remark 4.6]{WZ} in the special occasion  when at most one of the $(d_1)_i$'s is odd for $1\leqslant i\leqslant r$. However, when more than two $(d_1)_i$'s are odd for $1\leqslant i\leqslant r$, the dimensional lower-bounds of the representations of $\mathfrak{l}_\bbk$  given here should be optimal (see the proof of the forthcoming Theorem \ref{sumreachable}), which are much larger than the ones mentioned in \cite[Remark 4.6]{WZ}. Also note that the conclusion we obtained in Proposition \ref{sumdivisible} does not depend on Conjecture \ref{conjecture}.
\end{rem}

\subsubsection{}
Under the assumption of Conjecture~\ref{conjecture}, the following theorem shows that the dimensional lower bounds for the representations of reduced enveloping algebra $U_\chi(\mathfrak{l}_{\bbk})$ with nilpotent $p$-character $\chi\in(\mathfrak{l}_{\bbk})^*_{\bar0}$ in Proposition~\ref{sumdivisible} are accessible.

\begin{theorem}\label{sumreachable}
Retain the assumptions as in Proposition~\ref{sumdivisible}. If Conjecture \ref{conjecture} holds and $p\gg0$, then
 the reduced enveloping algebra $U_\chi(\mathfrak{l}_{\bbk})$ admits irreducible representations of dimension $p^{\frac{d'_0}{2}}2^{\frac{d'_1+l}{2}}$ (appoint $l=0$ when the $(d_1)_i$'s are all even for $1\leqslant i\leqslant r$).
\end{theorem}

\begin{proof}
For each $1\leqslant i\leqslant r$, let $Q_{\chi_i}^{\chi_i}$ be the $(\mathfrak{l}_{\bbk})_i$-module as defined in \S\ref{2.3.2}, and denote by $U_{\chi_i}((\mathfrak{l}_{\bbk})_i,\bar e_i)=(\text{End}_{(\mathfrak{l}_{\bbk})_i}Q_{\chi_i}^{\chi_i})^{\text{op}}$ the reduced $W$-superalgebra of basic Lie superalgebra $(\mathfrak{l}_{\bbk})_i$ associated with nilpotent element $\bar e_i\in((\mathfrak{l}_{\bbk})_i)_{\bar0} $. Let $Q_\chi^\chi$ be the $\mathfrak{l}_{\bbk}$-module with the same definition as in \S\ref{2.3.2}, and $U_{\chi}(\mathfrak{l}_{\bbk},\bar e)$ the reduced $W$-superalgebra of $\mathfrak{l}_{\bbk}$ associated with nilpotent element $\bar e\in(\mathfrak{l}_{\bbk})_{\bar0}$. Then we have
\begin{equation}\label{dir w}
\begin{array}{lllll}
U_{\chi}(\mathfrak{l}_{\bbk},\bar e)&=&(\text{End}_{\mathfrak{l}_{\bbk}}Q_{\chi}^{\chi})^{\text{op}}&\cong& (\text{End}_{\bigoplus\limits_{i=1}^r(\mathfrak{l}_{\bbk})_i}\bigoplus\limits_{i=1}^rQ_{\chi_i}^{\chi_i})^{\text{op}}\\
&\cong&\bigotimes\limits_{i=1}^r
(\text{End}_{(\mathfrak{l}_{\bbk})_i}Q_{\chi_i}^{\chi_i})^{\text{op}}&=&\bigotimes\limits_{i=1}^r U_{\chi_i}((\mathfrak{l}_{\bbk})_i,\bar e_i)
\end{array}
\end{equation}as ${\bbk}$-algebras. Now we proceed the proof by steps.

(1) We first prove the conclusion for the $\bbk$-algebra $\bigotimes\limits_{i=1}^l U_{\chi_i}((\mathfrak{l}_{\bbk})_i,e_i)$. We will carry the proof by induction on $l$.

(1-i) For the beginning  of induction, let us first look into each single term $(\mathfrak{l}_\bbk)_i$ for $1\leqslant i\leqslant l$, and make some basic observation on the tensor product of two terms. Under the assumption of the theorem, Corollary~\ref{1122} shows that the ${\bbk}$-algebra $U_{\chi_i}((\mathfrak{l}_{\bbk})_i,\bar e_i)$ admits two-dimensional  representations for $1\leqslant i\leqslant l$.  Denote by $V_{1}$ and $V_{2}$ the two-dimensional  irreducible representations (of type $Q$) of the ${\bbk}$-algebras $U_{\chi_{1}}((\mathfrak{l}_{\bbk})_{1},\bar e_{1})$ and $U_{\chi_{2}}((\mathfrak{l}_{\bbk})_{2},\bar e_{2})$ (if occurs), respectively. It follows from Lemma~\ref{AB}(iii) that $V_{1}\boxtimes V_{2}\cong (V_{1}\circledast V_{2}) \oplus\Upsilon(V_{1}\circledast V_{2})$ as $U_{\chi_{1}}((\mathfrak{l}_{\bbk})_{1},\bar e_{1})\otimes U_{\chi_{2}}((\mathfrak{l}_{\bbk})_{2},\bar e_{2})$-modules, where $V_{1}\circledast V_{2}$ is an irreducible $U_{\chi_{1}}((\mathfrak{l}_{\bbk})_{1},\bar e_{1})\otimes U_{\chi_{2}}((\mathfrak{l}_{\bbk})_{2},\bar e_{2})$-module of type $M$.
Recall that $V_{1}\circledast V_{2}$ is the same underlying vector space as $\Upsilon(V_{1}\circledast V_{2})$, thereby sharing the same dimension, i.e. $\dim V_{1}\circledast V_{2}=\dim\Upsilon(V_{1}\circledast V_{2})$.
Since $V_{1}\boxtimes V_{2}$ is just $V_{1}\otimes V_{2}$ as vector spaces, from all above
we can conclude that $V_{1}\circledast V_{2}$ is an irreducible $U_{\chi_{1}}((\mathfrak{l}_{\bbk})_{1},\bar e_{1})\otimes U_{\chi_{2}}((\mathfrak{l}_{\bbk})_{2},\bar e_{2})$-module of type $M$ with dimension $2=2^{\frac{2}{2}}$.

Denote by $V_{3}$ a two-dimensional  irreducible representation (of type $Q$) of the ${\bbk}$-algebra $U_{\chi_{3}}((\mathfrak{l}_{\bbk})_{3},\bar e_{3})$ (if occurs). It follows from Lemma~\ref{AB}(ii) that $(V_{1}\circledast V_{2})\boxtimes V_{3}$ is an irreducible $(U_{\chi_{1}}((\mathfrak{l}_{\bbk})_{1},\bar e_{1})\otimes U_{\chi_{2}}((\mathfrak{l}_{\bbk})_{2},\bar e_{2}))\otimes U_{\chi_{3}}((\mathfrak{l}_{\bbk})_{3},\bar e_{3})$-module of type $Q$ with dimension $2\cdot2=4=2^{\frac{3+1}{2}}$.

(1-ii) On the basis of (1-i),  we can easily draw the conclusion by induction on the number of the terms for $\bigotimes\limits_{i=1}^l U_{\chi_i}((\mathfrak{l}_{\bbk})_i,e_i)$, summarizing it as:
\begin{itemize}
\item[(1-ii-i)] when $l$ is odd, the ${\bbk}$-algebra $\bigotimes\limits_{i=1}^l U_{\chi_i}((\mathfrak{l}_{\bbk})_i,\bar e_i)$ admits an irreducible representation of type $Q$ with dimension $2^{\frac{l+1}{2}}$, and we denote it by $V$;
\item[(1-ii-ii)] when $l$ is even, the ${\bbk}$-algebra $\bigotimes\limits_{i=1}^l U_{\chi_i}((\mathfrak{l}_{\bbk})_i,\bar e_i)$ admits an irreducible representation of type $M$ with dimension $2^{\frac{l}{2}}$ (assume that $l=0$ when the $(d_1)_i$'s are all even for $1\leqslant i\leqslant r$), and set it as $V'$.
\end{itemize}

(2) Now we consider the general case with $U_\chi(\mathfrak{l}_{\bbk},\bar e)\cong\bigotimes\limits_{i=1}^r U_{\chi_i}((\mathfrak{l}_{\bbk})_i,\bar e_i)$.

Under the assumption of the theorem, Corollary~\ref{1122} shows that the ${\bbk}$-algebra $U_{\chi_i}((\mathfrak{l}_{\bbk})_i,\bar e_i)$ admits one-dimensional  representations of type $M$ for $l+1\leqslant i\leqslant r$ (if occurs). Easy induction based on Lemma~\ref{AB}(i) shows that the ${\bbk}$-algebra $\bigotimes\limits_{i=l+1}^r U_{\chi_i}((\mathfrak{l}_{\bbk})_i,\bar e_i)$ admits a one-dimensional  representation of type $M$, denoted by  $W$.

First note that $\bigotimes\limits_{i=1}^r U_{\chi_i}((\mathfrak{l}_{\bbk})_i,\bar e_i)\cong\bigotimes\limits_{i=1}^l U_{\chi_i}((\mathfrak{l}_{\bbk})_i,\bar e_i)\otimes\bigotimes\limits_{i=l+1}^r U_{\chi_i}((\mathfrak{l}_{\bbk})_i,\bar e_i)$. Based on the parity of $l$, for each case we will consider separately.

(2-i) When $l$ is odd, (1-ii-i) shows that the ${\bbk}$-algebra $\bigotimes\limits_{i=1}^l U_{\chi_i}((\mathfrak{l}_{\bbk})_i,\bar e_i)$ admits an irreducible representation $V$ of type $Q$ with dimension $2^{\frac{l+1}{2}}$, and the ${\bbk}$-algebra $\bigotimes\limits_{i=l+1}^r U_{\chi_i}((\mathfrak{l}_{\bbk})_i,\bar e_i)$ admits a one-dimensional  representation $W$ of type $M$ by the preceding remark. It follows from Lemma~\ref{AB}(ii) that $V\boxtimes W$ is an irreducible module of type $Q$ with dimension  $2^{\frac{l+1}{2}}$ for the $\bbk$-algebra $\bigotimes\limits_{i=1}^l U_{\chi_i}((\mathfrak{l}_{\bbk})_i,\bar e_i)\otimes\bigotimes\limits_{i=l+1}^r U_{\chi_i}((\mathfrak{l}_{\bbk})_i,\bar e_i)\cong\bigotimes\limits_{i=1}^r U_{\chi_i}((\mathfrak{l}_{\bbk})_i,\bar e_i)$.

(2-ii) When $l$ is even, (1-ii-ii) shows that the ${\bbk}$-algebra $\bigotimes\limits_{i=1}^l U_{\chi_i}((\mathfrak{l}_{\bbk})_i,\bar e_i)$ admits an irreducible representation $V'$ of type $M$ with dimension $2^{\frac{l}{2}}$, and the ${\bbk}$-algebra $\bigotimes\limits_{i=l+1}^r U_{\chi_i}((\mathfrak{l}_{\bbk})_i,\bar e_i)$ admits a one-dimensional  representation $W$ of type $M$ by the preceding remark. It follows from Lemma~\ref{AB}(i) that $V'\boxtimes W$ is an irreducible module of type $M$ with dimension $2^{\frac{l}{2}}$ for the $\bbk$-algebra  $\bigotimes\limits_{i=1}^l U_{\chi_i}((\mathfrak{l}_{\bbk})_i,\bar e_i)\otimes\bigotimes\limits_{i=l+1}^r U_{\chi_i}((\mathfrak{l}_{\bbk})_i,\bar e_i)\cong\bigotimes\limits_{i=1}^r U_{\chi_i}((\mathfrak{l}_{\bbk})_i,\bar e_i)$.

(3) Keeping in mind the algebra isomorphism
\begin{equation}\label{isosum}
U_\chi(\mathfrak{l}_{\bbk})\cong \text{Mat}_{p^{\frac{d'_0}{2}}2^{\lceil\frac{d'_1}{2}\rceil}}(U_\chi(\mathfrak{l}_{\bbk},\bar e))
\end{equation}(cf. \cite[Remark 4.6]{WZ}), we can conclude:

(3-i) when $l$ is odd, it follows from (2-i) that the ${\bbk}$-algebra $U_\chi(\mathfrak{l}_{\bbk})$ affords irreducible representations of dimension $p^{\frac{d'_0}{2}}2^{\frac{d'_1-1}{2}}\cdot 2^{\frac{l+1}{2}}=p^{\frac{d'_0}{2}}2^{\frac{d'_1+l}{2}}$;

(3-ii) when $l$ is even, it follows from (2-ii) that the ${\bbk}$-algebra $U_\chi(\mathfrak{l}_{\bbk})$ affords irreducible representations of dimension $p^{\frac{d'_0}{2}}2^{\frac{d'_1}{2}}\cdot 2^{\frac{l}{2}}=p^{\frac{d'_0}{2}}2^{\frac{d'_1+l}{2}}$.

Summing up, we complete the proof.
\end{proof}

\begin{rem}\label{refine'} Recall that the dimensional lower-bounds introduced in Proposition~\ref{sumdivisible} are much larger than the boundary obtained by Wang-Zhao in \cite[Remark 4.6]{WZ} when more than two $(d_1)_i$'s are odd for $1\leqslant i\leqslant r$ (cf. Remark~\ref{refine}).  In fact, careful inspection on the proof of Theorem~\ref{sumreachable} shows that Wang-Zhao's estimation on these lower bounds  can never be reached in this case, and the ones introduced in Theorem~\ref{sumreachable} are optimal.
\end{rem}

\subsection{Discussion for basic Lie superalgebras with arbitrary $p$-characters} \label{proofmainthm0.4sub}
In the concluding subsection  we will consider the accessibility of the lower bounds of the Super KW Property with any $p$-characters in Proposition~\ref{wzd2}.
\subsubsection{}\label{5.4.1}
For a given $p$-character $\xi\in(\ggg_\bbk)^*_{\bar0}$,  we have its Jordan-Chevalley decomposition $\xi=\xi_{\bar s}+\xi_{\bar n}$ under the Ad\,$(G_\bbk)_\ev$-equivariant isomorphism $({\ggg}_{\bbk})^*_{\bar0}\cong({\ggg}_{\bbk})_{\bar0}$ induced by the non-degenerate bilinear form $(\cdot,\cdot)$ on $({\ggg}_{\bbk})_{\bar0}$. This is to say, the decomposition of $\xi$ can be identified with the usual Jordan decomposition $\bar x=\bar s+\bar n$ when $\xi$ corresponds to $\bar x$ under the isomorphism $(\ggg_\bbk)_\bz^*\cong (\ggg_\bbk)_\bz$.
 Let $\mathfrak{h}_{\bbk}$ be a Cartan subalgebra of ${\ggg}_{\bbk}$ which contains $\bar s$ and denote by $\mathfrak{l}_{\bbk}={\ggg}_{\bbk}^{\bar s}$ the centralizer of $\bar s$ in ${\ggg}_{\bbk}$. Let $\Phi$ be the root system of ${\ggg}_{\bbk}$ and $\Phi(\mathfrak{l}_{\bbk}):=\{\alpha\in\Phi\mid\alpha(\bar s)=0\}$. By \cite[Proposition 5.1]{WZ}, $\mathfrak{l}_{\bbk}$ is always a direct sum of basic Lie superalgebras with a system $\Delta$ of simple roots of ${\ggg}_{\bbk}$ such that $\Delta\cap \Phi(\mathfrak{l}_{\bbk})$ is a system of simple roots of $\Phi(\mathfrak{l}_{\bbk})$ (note that a toral subalgebra of ${\ggg}_{\bbk}$ may also appear in the summand).

Set $$\mathfrak{l}_{\bbk}={\ggg}_{\bbk}^{\bar s}=\bigoplus\limits_{i=1}^r({\ggg}_{\bbk})_i\oplus\mathfrak{t}'_{\bbk},$$ where each $({\ggg}_{\bbk})_i$ is a basic Lie superalgebra for $1\leqslant i\leqslant r$, and
$\mathfrak{t}'_{\bbk}$ is a toral subalgebra of ${\ggg}_{\bbk}$. Then by \cite[\S5.1]{WZ}, $\xi_{\bar n}=\xi_1+\cdots+\xi_r$ is a nilpotent $p$-character of $\mathfrak{l}_{\bbk}$ with $\xi_i\in({\ggg}_{\bbk})_i^*$ (each $\xi_i$  can be viewed in $\mathfrak{l}_{\bbk}^*$ by letting $\xi_i(\bar y)=0$ for all $\bar y\in\bigoplus\limits_{j\neq i}({\ggg}_{\bbk})_j\oplus\mathfrak{t}'_{\bbk}$) for $1\leqslant i\leqslant r$. Let $\bar n=\bar n_1+\cdots+\bar n_r$ be the corresponding decomposition of $\bar n$ in  $\mathfrak{l}_{\bbk}$ such that $\xi_i(\cdot)=(\bar n_i,\cdot)$  for $1\leqslant i\leqslant r$. We can obtain the reduced $W$-superalgebra $U_{\xi_i}(({\ggg}_{\bbk})_i,\bar n_i)$ of  $({\ggg}_{\bbk})_i$  associated with nilpotent element $\bar n_i\in({\ggg}_{\bbk})_i$. It is easily verified that
$$U_{\xi_{\bar n}}(\bigoplus\limits_{i=1}^r({\ggg}_{\bbk})_i,\bar n)\cong\bigotimes\limits_{i=1}^r U_{\xi_i}(({\ggg}_{\bbk})_i,\bar n_i)$$ by the same discussion as \eqref{dir w}.
Define
\begin{equation}\label{arbitdim}
\begin{array}{rllrlll}
d_0&:=&\text{dim}\,({\ggg}_{\bbk})_{\bar0}-\text{dim}\,({\ggg}_{\bbk}^{\bar x})_{\bar0},&
d_1&:=&\text{dim}\,({\ggg}_{\bbk})_{\bar1}-\text{dim}\,({\ggg}_{\bbk}^{\bar x})_{\bar1},\\
(d_0)_i&:=&\text{dim}\,(({\ggg}_{\bbk})_i)_{\bar0}-\text{dim}\,(({\ggg}_{\bbk})^{\bar n_i}_i)_{\bar0},&
(d_1)_i&:=&\text{dim}\,(({\ggg}_{\bbk})_i)_{\bar1}-\text{dim}\,(({\ggg}_{\bbk})^{\bar n_i}_i)_{\bar1},
\end{array}
\end{equation}where ${\ggg}_{\bbk}^{\bar x}$ denotes the centralizer of $\bar x$ in ${\ggg}_{\bbk}$, and $({\ggg}_{\bbk})^{\bar n_i}_i$ the centralizer of $\bar n_i$ in $({\ggg}_{\bbk})_i$ for each $i\in\{1,\cdots,r\}$.
Rearrange the summands of $\bigoplus\limits_{i=1}^r({\ggg}_{\bbk})_i$ such that each $(d_1)_i$ is odd for $1\leqslant i\leqslant l$ (if occurs), and each $(d_1)_i$ is even for $l+1\leqslant i\leqslant r$ (if occurs).

Let $\mathfrak{b}_{\bbk}=\mathfrak{h}_{\bbk}\oplus\mathfrak{n}_{\bbk}$ be the Borel subalgebra associated to $\Delta$. Let $\ppp_\bbk$ be a parabolic subalgebra of $\ggg_\bbk$ with Levi factor $\mathfrak{l}_\bbk$, i.e. ${\ppp}_{\bbk}=\mathfrak{l}_{\bbk}+\mathfrak{b}_{\bbk}
=\mathfrak{l}_{\bbk}\oplus\mathfrak{u}_{\bbk}$, where $\uuu_\bbk$ is the nilradical of $\ppp_\bbk$.
 Set $\mathfrak{u}^-_{\bbk}$ to be the complement nilradical of ${\ppp}_{\bbk}$ such that 
 ${\ggg}_{\bbk}={\ppp}_{\bbk}\oplus\mathfrak{u}^-_{\bbk}=\mathfrak{u}_{\bbk}\oplus\mathfrak{l}_{\bbk}\oplus\mathfrak{u}^-_{\bbk}$ as vector spaces.
Since $\xi(\mathfrak{u}_{\bbk})=0$ and $\xi|_{\mathfrak{l}_{\bbk}}=\xi_{\bar n}|_{\mathfrak{l}_{\bbk}}$ is nilpotent by \cite[\S5.1]{WZ}, any $U_\xi(\mathfrak{l}_{\bbk})$-mod can be regarded as a $U_\xi({\ppp}_{\bbk})$-mod with a trivial action of $\mathfrak{u}_{\bbk}$. Wang-Zhao proved that the ${\bbk}$-algebras $U_\xi({\ggg}_{\bbk})$ and $U_\xi(\mathfrak{l}_{\bbk})$ are Morita equivalent in \cite[Theorem 5.2]{WZ}, and any irreducible $U_\xi({\ggg}_{\bbk})$-module can be induced from an irreducible $U_\xi(\mathfrak{l}_{\bbk})$-mod (which is also a $U_\xi({\ppp}_{\bbk})$-mod with a trivial action of $\mathfrak{u}_{\bbk}$) by
\begin{equation}\label{gp}
U_\xi({\ggg}_{\bbk})\otimes_{U_\xi({\ppp}_{\bbk})}-:\qquad U_\xi(\mathfrak{l}_{\bbk})\text{-}mod\rightarrow U_\xi({\ggg}_{\bbk})\text{-}mod.
\end{equation}

\subsubsection{} Keep the notations and assumptions as above. We first recall the following result.

\vskip0.2cm

\noindent{\itshape{\bf Proposition 5.1. (\cite{WZ})}
Let ${\ggg}_{\bbk}$ be a basic Lie superalgebra over ${\bbk}=\overline{\mathbb{F}}_p$, assuming that the prime $p$ satisfies the restriction imposed in \S\ref{prel}(Table 1). Let $\xi$ be arbitrary $p$-character in $({\ggg}_{\bbk})^*_{\bar0}$. Then the dimension of every $U_\xi({\ggg}_{\bbk})$-module $M$ is divisible by $p^{\frac{d_0}{2}}2^{\lfloor\frac{d_1}{2}\rfloor}$.}

\vskip0.2cm

The above proposition was first verified by Wang-Zhao in \cite[Theorem 5.6]{WZ}. Compared with the proof we present below, Wang-Zhao's proof in \cite{WZ} is more concise since they did not consider the parity of the $(d_1)_i$'s ($1\leqslant i\leqslant r$) for the summands of $\mathfrak{l}_{{\bbk}}={\ggg}_{{\bbk}}^{\bar s}$ decomposed into $ \bigoplus\limits_{i=1}^r({\ggg}_{\bbk})_i\oplus\mathfrak{t}'_{{\bbk}}$.

As our approach to the main goal of this subsection will be much dependent on the analysis of the parities of those $(d_1)_i$'s, along with the above result, then we have to formulate another proof of Proposition \ref{wzd2}, based on Proposition~\ref{sumdivisible}. This will be important for us to get at the goal.

\begin{proof}
(1) First note that \cite[Theorem 5.6]{WZ} shows
\begin{equation}\label{gtol}
\begin{split}
\text{\underline{dim}}\,{\ggg}_{\bbk}-\text{\underline{dim}}\,{\ggg}_{\bbk}^{\bar x}&=\text{\underline{dim}}\,{\ggg}_{\bbk}-\text{\underline{dim}}\,\mathfrak{l}_{\bbk}^{\bar n}\\
&=2\text{\underline{dim}}\,\mathfrak{u}_{\bbk}^-+(\text{\underline{dim}}\,\mathfrak{l}_{\bbk}-\text{\underline{dim}}\,\mathfrak{l}_{\bbk}^{\bar n}),
\end{split}
\end{equation} where $\mathfrak{l}_{\bbk}^{\bar n}$ denotes the centralizer of $\bar n$ in $\mathfrak{l}_{\bbk}$.
Since $\mathfrak{l}_{\bbk}=\bigoplus\limits_{i=1}^r({\ggg}_{\bbk})_i\oplus\mathfrak{t}'_{\bbk}$ and $\bar n
\in \bigoplus\limits_{i=1}^r({\ggg}_{\bbk})_i$, it is obvious that $(\mathfrak{t}'_{\bbk})^{\bar n}=\mathfrak{t}'_{\bbk}$, then
\begin{equation}\label{ltosum}
\begin{split}
\text{\underline{dim}}\,\mathfrak{l}_{\bbk}-\text{\underline{dim}}\,\mathfrak{l}_{\bbk}^{\bar n}=&\sum\limits_{i=1}^r
(\text{\underline{dim}}\,({\ggg}_{\bbk})_i-\text{\underline{dim}}\,({\ggg}_{\bbk})_i^{\bar n_i})+
\text{\underline{dim}}\,\mathfrak{t}'_{\bbk}-\text{\underline{dim}}\,(\mathfrak{t}'_{\bbk})^{\bar n}\\=&(\sum\limits_{i=1}^r(d_0)_i,\sum\limits_{i=1}^r(d_1)_i).\end{split}
\end{equation} As $\text{\underline{dim}}\,\mathfrak{u}_{\bbk}^-=\text{\underline{dim}}\,\mathfrak{u}_{\bbk}$, \eqref{gtol} shows that
\begin{equation}\label{dimu}
\begin{array}{rclll}
\text{dim}\,(\mathfrak{u}^-_{\bbk})_{\bar0}&=&\text{dim}\,(\mathfrak{u}_{\bbk})_{\bar0}&=&\frac{d_0-\sum\limits_{i=1}^r(d_0)_i}{2};\\ \text{dim}\,(\mathfrak{u}^-_{\bbk})_{\bar1}&=&\text{dim}\,(\mathfrak{u}_{\bbk})_{\bar1}&=&\frac{d_1-\sum\limits_{i=1}^r(d_1)_i}{2}.
\end{array}
\end{equation}

(2) For  $\mathfrak{l}_{\bbk}=\bigoplus\limits_{i=1}^r({\ggg}_{\bbk})_i\oplus\mathfrak{t}'_{\bbk}$,
we have an algebra isomorphism
\begin{equation}\label{lsumt'}
U_{\xi_{\bar n}}(\mathfrak{l}_{\bbk})\cong U_{\xi_{\bar n}}(\bigoplus\limits_{i=1}^r({\ggg}_{\bbk})_i\oplus\mathfrak{t}'_{\bbk})\cong
U_{\xi_{\bar n}}(\bigoplus\limits_{i=1}^r({\ggg}_{\bbk})_i)\otimes U_0(\mathfrak{t}'_{\bbk}).
\end{equation} 

 Let us first consider  the representations of the $\bbk$-algebra $U_{\xi_{\bar n}}(\bigoplus\limits_{i=1}^r({\ggg}_{\bbk})_i)$. Since $\xi_{\bar n}|_{\mathfrak{l}_{\bbk}}$ is nilpotent and each $({\ggg}_{\bbk})_i$ is a basic Lie superalgebra for $1\leqslant i\leqslant r$, it follows from Proposition~\ref{sumdivisible} that every $U_{\xi_{\bar n}}(\bigoplus\limits_{i=1}^r({\ggg}_{\bbk})_i)$-module is divisible by $p^{\frac{\sum\limits_{i=1}^r(d_0)_i}{2}}2^{\frac{l+\sum\limits_{i=1}^r(d_1)_i}{2}}$.

 Next we look at the ${\bbk}$-algebra $U_0(\mathfrak{t}'_{\bbk})$. As $\mathfrak{t}'_{\bbk}$ is a toral subalgebra of ${\ggg}_{\bbk}$ with a basis $\{t_1,\cdots,t_d\}$ such that $t_i^{[p]}=t_i$ for all $1\leqslant i\leqslant d$, then $U_0(\mathfrak{t}'_{\bbk})\cong A_1^{\otimes d}$ where $A_1\cong{\bbk}[X]/(X^p-X)$ is a $p$-dimensional commutative semisimple algebra whose irreducible representations are one-dimensional  (naturally of type $M$). Hence we can conclude from Lemma~\ref{AB}(i) that all irreducible representations of $U_0(\mathfrak{t}'_{\bbk})$ are one-dimensional and of type $M$.

Recall that $U_{\xi_{\bar n}}(\mathfrak{l}_{\bbk})\cong U_{\xi_{\bar n}}(\bigoplus\limits_{i=1}^r({\ggg}_{\bbk})_i)\otimes U_0(\mathfrak{t}'_{\bbk})$.
Summing up, we can conclude that any $U_{\xi_{\bar n}}(\mathfrak{l}_{\bbk})$-module is divisible by $p^{\frac{\sum\limits_{i=1}^r(d_0)_i}{2}}2^{\frac{l+\sum\limits_{i=1}^r(d_1)_i}{2}}$.

(3) Recall that 
an object in the representation category of $U_{\xi}(\mathfrak{l}_{\bbk})$ can be regarded as one in the representation category of $U_\xi({\ppp}_{\bbk})$ with a trivial action of $\mathfrak{u}_{\bbk}$, then it follows from  \eqref{gp} and \eqref{dimu} that the dimension of every $U_\xi({\ggg}_{\bbk}$)-mod is divisible by
\begin{equation}\label{numb}
p^{\sum\limits_{i=1}^r\frac{(d_0)_i}{2}}2^{\frac{l}{2}+\sum\limits_{i=1}^r\frac{(d_1)_i}{2}}\cdot p^{\frac{d_0-\sum\limits_{i=1}^r(d_0)_i}{2}} 2^{\frac{d_1-\sum\limits_{i=1}^r(d_1)_i}{2}}
=p^{\frac{d_0}{2}}2^{\frac{d_1+l}{2}}.
\end{equation}

(4) 
We now claim that in (\ref{numb}) we have $l=0$ or $1$; this is to say,
\begin{align}\label{finalclaim}
\mbox{
at most one of the } (d_1)_i\mbox{'s is odd for }1\leqslant i\leqslant r.
 \end{align}
  The proof of Claim \eqref{finalclaim} will be given for each case separately.
 Recall that for $\mathfrak{l}_\bbk=\ggg_\bbk^{\bar s}$ with a direct-sum decomposition $\mathfrak{l}_\bbk= \bigoplus\limits_{i=1}^r({\ggg}_{\bbk})_i\oplus\mathfrak{t}'_{{\bbk}}$,
 we have:
(i) $d_1$ and $\sum\limits_{i=1}^r(d_1)_i$ share the same parity (see \eqref{arbitdim}, \eqref{gtol} and \eqref{ltosum}); (ii) $(d_1)_i$'s are odd for $1\leqslant i \leqslant l$, and $(d_1)_i$'s are even for $l+1\leqslant i \leqslant r$ (see \S\ref{5.4.1}). It follows that $d_1$ and $l$ have the same parity. Combining this with Claim \eqref{finalclaim}, Proposition \ref{wzd2} follows.
\end{proof}

  \vskip0.2cm
  \noindent\textsc{The proof of Claim \eqref{finalclaim}}.
  We will complete the proof by steps.

(1) The above discussion shows that we only need to consider the summands of $ \bigoplus\limits_{i=1}^r({\ggg}_{\bbk})_i$,
investigating the situation with non-zero odd parts.  Recall that an explicit list of non-$W$-equivalent systems of positive roots was found by Kac in \cite[\S2.5.4]{K} (a system of simple roots for $F(4)$ is missing; see the remark above \cite[Proposition 5.1]{WZ}). Note that in the examples given by Kac, the Cartan subalgebra $\mathfrak{h}_{\bbk}$ is a subspace of the space $D$ of diagonal matrices; the roots are expressed in terms of the standard basis $\epsilon_i$ of $D^*$ (more accurately, the restrictions of the $\epsilon_i$ to $\mathfrak{h}_{\bbk}$). In the following we assume that the semisimple element $\bar s\in \mathfrak{h}_{\bbk}$.

(2) Given any nilpotent element $e\in\mathfrak{sl}(M|N)_{\bar0}$, we have $\mathfrak{gl}(M|N)_{\bar1}=\mathfrak{sl}(M|N)_{\bar1}$ and $\mathfrak{gl}(M|N)^{e}_{\bar1}=\mathfrak{sl}(M|N)^{e}_{\bar1}$ for all $M, N\in{\bbz}_+$, where $\mathfrak{gl}(M|N)^{e}_{\bar1}$ and $\mathfrak{sl}(M|N)^{e}_{\bar1}$ denote the centralizer of $e$ in $\mathfrak{gl}(M|N)_{\bar1}$ and $\mathfrak{sl}(M|N)_{\bar1}$, respectively.
It follows from \cite[\S3.2]{WZ} that $(d_1)_i\,(1\leqslant i\leqslant r)$ is always even for the summand $({\ggg}_{\bbk})_i$ which is isomorphic to the basic Lie superalgebra of type $A(m,n)$ with $m,n\in{\bbz}_+$. In virtue of this consequence, completely elementary yet tedious case-by-case calculations show that

(i) When $\ggg_{\bbk}$ is of type $A(M,N)$, each summand of $\bigoplus\limits_{i=1}^r({\ggg}_{\bbk})_i$  with non-zero odd parts is always isomorphic to a basic Lie superalgebra of type $A(m,n)$ with $m,n\in{\bbz}_+$, thus for $1\leqslant i\leqslant r$ the $(d_1)_i$'s are all even in this case.

(ii) When $\ggg_\bbk$ is of type $B(M,N), C(M,N)$ or $D(M,N)$, the summands of $\bigoplus\limits_{i=1}^r({\ggg}_{\bbk})_i$ with non-zero odd parts are either isomorphic to $A(m,n)$ with $m,n\in{\bbz}_+$; or at most one summand is isomorphic to $B(m,n),\,C(m,n)$ or $D(m,n)$ ($m,n\in{\bbz}_+$) respectively (which is of the same type as $\mathfrak{g}_{\bbk}$).
Hence for $1\leqslant i\leqslant r$, at most one of the $(d_1)_i$'s is odd in this case.

(iii) When $\ggg_\bbk$ is of type $D(2,1;\bar a)$ or $G(3)$, the summands of $\bigoplus\limits_{i=1}^r({\ggg}_{\bbk})_i$  with non-zero odd parts are either isomorphic to $A(m,n)$ with $m,n\in{\bbz}_+$; or at most one summand is isomorphic to $B(m,n)$ ($m,n\in{\bbz}_+$). At extreme, $\mathfrak{l}_{\bbk}={\ggg}_{\bbk}^{\bar s}$ equals $D(2,1;\bar a)$ or $G(3)$ respectively when $\bar s=0$. Hence for $1\leqslant i\leqslant r$, at most one of the $(d_1)_i$'s is odd in this case.

(iv) When $\ggg_\bbk$ is of type $F(4)$, the summands of $\bigoplus\limits_{i=1}^r({\ggg}_{\bbk})_i$ with non-zero odd parts are either isomorphic to $A(m,n)$ with $m,n\in{\bbz}_+$; or at most one summand is either isomorphic to $B(m,n)$ with $m,n\in{\bbz}_+$, or to $D(2,1;\bar a)$. At extreme, $\mathfrak{l}_{\bbk}={\ggg}_{\bbk}^{\bar s}=F(4)$ when $\bar s=0$. Hence for $1\leqslant i\leqslant r$, at most one of the $(d_1)_i$'s is odd in this case.

Summing up the above discussions, we can conclude that at most one of the $(d_1)_i$'s ($1\leqslant i\leqslant r$) is odd in the summands of $\bigoplus\limits_{i=1}^r({\ggg}_{\bbk})_i$. Then we complete the proof of Claim \eqref{finalclaim}.
\begin{rem} The significance of the above new proof can be seen as below:

(1) from the precise analysis in the proof,
one can conclude that the lower bounds for the dimension of $U_\xi({\ggg}_{\bbk})$-modules are critically depending on the parity of the $(d_1)_i$'s for $1\leqslant i\leqslant r$. Without careful inspection on the summands of $\mathfrak{l}_{{\bbk}}\cong \bigoplus\limits_{i=1}^r({\ggg}_{\bbk})_i\oplus\mathfrak{t}'_{{\bbk}}$, we can not make it sure that the lower bounds introduced in \cite[Theorem 5.6]{WZ} are optimal.

(2) What we are concerned is the realization of  ``small representations" of dimensions equaling the lower bounds  in \cite[Theorem 5.6]{WZ} under the assumption of Conjecture~\ref{conjecture}.
By the above proof, we can demonstrate how to realize the $U_{\xi}({\ggg}_{\bbk})$-modules of minimal dimensions by the representation theory of the ${\bbk}$-algebra
$U_{\xi}(\mathfrak{l}_{\bbk})$; see \S\ref{5.4.3} for more details.
\end{rem}

\subsubsection{The proof of Theorem \ref{intromain-2}}\label{5.4.3}
Now we are in a position to prove Theorem \ref{intromain-2}, attacking the problem of the accessibility of the lower bounds in the super Kac-Weisfeiler property \cite[Theorem 5.6]{WZ}.

\begin{proof} Retain the notation as \eqref{arbitdim}.
For the subalgebra $\mathfrak{l}_{\bbk}={\ggg}_{\bbk}^{\bar s}=\bigoplus\limits_{i=1}^r({\ggg}_{\bbk})_i\oplus\mathfrak{t}'_{\bbk}$ of ${\ggg}_{\bbk}$, recall that Theorem~\ref{sumreachable} shows that the ${\bbk}$-algebra $U_{\xi_{\bar n}}(\bigoplus\limits_{i=1}^r({\ggg}_{\bbk})_i)$ admits an irreducible representation of dimension $p^{\sum\limits_{i=1}^r\frac{(d_0)_i}{2}}2^{\frac{l}{2}+\sum\limits_{i=1}^r\frac{(d_1)_i}{2}}$ under the assumption of Conjecture~\ref{conjecture}, and we denote it by $V$. By the arguments of step (4) in the proof of Proposition \ref{wzd2}, one can conclude that
$$\dim V=
p^{\sum\limits_{i=1}^r\frac{(d_0)_i}{2}}2^{\frac{l}{2}+\sum\limits_{i=1}^r\frac{(d_1)_i}{2}}
=p^{\sum\limits_{i=1}^r\frac{(d_0)_i}{2}}2^
{\lfloor\sum\limits_{i=1}^r\frac{(d_1)_i}{2}\rfloor}.$$
 Step (2) in the proof of Proposition~\ref{wzd2} shows that the ${\bbk}$-algebra $U_{0}(\mathfrak{t}'_{\bbk})$ affords an irreducible representation of dimension one, denoted by $W$ (note that it is of type $M$). Thus it follows from Lemma~\ref{AB} that $V\boxtimes W$ is an irreducible representation of the ${\bbk}$-algebra $U_{\xi_{\bar n}}(\bigoplus\limits_{i=1}^r({\ggg}_{\bbk})_i)\otimes U_{0}(\mathfrak{t}'_{\bbk})\cong U_{\xi}(\mathfrak{l}_{\bbk})$ with dimension $p^{\sum\limits_{i=1}^r\frac{(d_0)_i}{2}}2^
{\lfloor\sum\limits_{i=1}^r\frac{(d_1)_i}{2}\rfloor}$.

Recall that  any $U_\xi(\mathfrak{l}_{\bbk})$-mod can be regarded as a $U_\xi({\ppp}_{\bbk})$-mod with a trivial action of $\mathfrak{u}_{\bbk}$.  It follows from \eqref{gp} that the $U_\xi({\ggg}_{\bbk})$-module induced from the ${U_\xi(\mathfrak{p}_{\bbk})}$-module $V\boxtimes W$ is irreducible.
By \eqref{dimu} we can conclude that the dimension of this induced module is equal to  %
\begin{equation}\label{final}
p^{\sum\limits_{i=1}^r\frac{(d_0)_i}{2}}2^{\lfloor\sum\limits_{i=1}^r\frac{(d_1)_i}{2}\rfloor}\cdot p^{\frac{d_0-\sum\limits_{i=1}^r(d_0)_i}{2}} 2^{\frac{d_1-\sum\limits_{i=1}^r(d_1)_i}{2}}
=p^{\frac{d_0}{2}}2^{\frac{d_1}{2}+(\lfloor\sum\limits_{i=1}^r\frac{(d_1)_i}{2}\rfloor-\sum\limits_{i=1}^r\frac{(d_1)_i}{2})}.
\end{equation}
Owing to \eqref{arbitdim}, \eqref{gtol} and \eqref{ltosum},  $\sum\limits_{i=1}^r(d_1)_i$ and $d_1$ have the same parity.
Hence, the desired result follows from \eqref{final}.
\end{proof}

\begin{rem}
For the reduced enveloping algebra $U_\xi({\ggg}_{\bbk})$ with arbitrary $p$-character $\xi\in({\ggg}_{\bbk})^*_{\bar0}$, the formulation of the super Kac-Weisfeiler property in Proposition \ref{wzd2} is dependent on the one on ${\ggg}_{\bbk}^{\bar s}$ which equals a direct sum of some basic Lie superalgebras and a toral subalgebra, and set it to be $\bigoplus\limits_{i=1}^r({\ggg}_{\bbk})_i\oplus \mathfrak{t}'_{\bbk}$ (cf. \cite{WZ}).
However, after the super Kac-Weisfeiler property for the ${\bbk}$-algebra $U_{\xi_{\bar n}}(\bigoplus\limits_{i=1}^r({\ggg}_{\bbk})_i)$ being refined in Proposition~\ref{sumdivisible}, one may be worried whether the real minimal dimensions of the representations of $U_\xi({\ggg}_{\bbk})$ are much larger than the ones introduced in Proposition \ref{wzd2}. Fortunately, Theorem \ref{intromain-2} certifies  that they are exactly the real minimal dimensions of those representations under the assumption of Conjecture~\ref{conjecture}.
\end{rem}

{\bf Acknowledgements}\quad
The authors would like to thank Weiqiang Wang and Lei Zhao whose work on super version of Kac-Weisfeiler conjecture stimulated them to do the present research.  The authors got much help  from the discussion with Weiqiang Wang, as well as from Yung-Ning Peng and Lei Zhao who  explained some results in their papers \cite{Peng3} and \cite{WZ} respectively. The authors express great thanks to them.

\end{document}